\documentclass[reqno]{amsart}

\usepackage{amsmath,amssymb}
\usepackage{mathptmx}      
\usepackage{mathrsfs}
\usepackage{latexsym}
\usepackage{amssymb}
\usepackage{bm}
\usepackage[usenames,dvipsnames]{xcolor}
\usepackage{graphicx,pifont}
\usepackage{multicol,multirow}
\usepackage{ulem,cancel}
\usepackage{tikz}
\usetikzlibrary[patterns]
\usepackage{comment}
\usepackage{hyperref}
\usepackage{array}
\usepackage{caption}
\usepackage{subcaption}
\usepackage{txfonts}
\usepackage{mathtools}
\usepackage[shortlabels]{enumitem}
\usepackage{cite}

\newtheorem{theorem}{Theorem}[section]
\newtheorem{lemma}[theorem]{Lemma}
\newtheorem{corollary}[theorem]{Corollary}

\theoremstyle{definition}
\newtheorem{definition}[theorem]{Definition}

\theoremstyle{remark}

\numberwithin{equation}{section}
\numberwithin{figure}{section}

\newcommand{\bs}[1]{{\boldsymbol{#1}}}

\newcommand{\oname}[1]{\textrm{#1}}
\newcommand{\abs}[1]{\left|#1\right|}

\newcommand{\eqdef}{\stackrel{\mathrm{def}}{=\joinrel=}}

\begin{document}

\title[Stability of FDM for Advection-Diffusion Equations]{On the Stability of Explicit Finite Difference Methods for Advection-Diffusion Equations}

\author[X.~Zeng]{Xianyi Zeng}
\address{Department of Mathematical Sciences,\\
         Computational Science Program, University of Texas at El Paso, El Paso, TX 79902, United States.\\
         Tel.: +1-915-747-6759}
\email[Corresponding author, X.~Zeng]{xzeng@utep.edu}

\author[M.~Hasan]{Md Mahmudul Hasan}
\address{Computational Science Program, University of Texas at El Paso, El Paso, TX 79902, United States.}
\email[M.~Hasan]{mhasan5@miners.utep.edu}

\date{\today}



\begin{abstract}
In this paper we study the stability of explicit finite difference discretizations of linear advection-diffusion equations (ADE) with arbitrary order of accuracy in the context of method of lines.
The analysis first focuses on the stability of the system of ordinary differential equations (ODE) that is obtained by discretizing the ADE in space and then extends to fully discretized methods where explicit Runge-Kutta methods are used for integrating the ODE system.
In particular, it is proved that all stable semi-discretization of the ADE gives rise to a conditionally stable fully discretized method if the time-integrator is at least first-order accurate, whereas high-order spatial discretization of the advection equation cannot yield a stable method if the temporal order is too low.
In the second half of this paper, we extend the analysis to a partially dissipative wave system and obtain the stability results for both semi-discretized and fully-discretized methods.
Finally, the major theoretical predictions are verified numerically.
\end{abstract}

\maketitle

\section{Introduction}
\label{sec:intro}
Numerical methods for partial differential equations that arise in engineering applications and physics problems have flourished in the past decades. 
In reality, these equations are usually complicated and involve terms that have different mathematical characteristics, such as advection and diffusion; to this end, a common practice is to select independent discretization operators to handle each term separately.
On the one hand, these operators are usually well studied in solving simple model equations -- such as the upwind or upwind-biased methods for linear advection equations and central schemes for diffusion equations.
In the context of method of lines, yet another ``dimension'' of the overall strategy is the time integrator, which has been extensively discussed in many texts on solving ordinary differential equations (ODE).
On the other hand, combining these numerical components may yield properties that are different from those of the individual methods when applied to their corresponding model equations.
A well-known example is that central difference in space and forward Euler in time is unstable for advection equations; however, when it is combined with the central difference for the diffusion term, the resulting method is conditionally stable for solving advection-diffusion equations (ADE) and is known as the FTCS method (Forward-Time Central-Space) in early literature, see for example~\cite{KWMorton:1980a,ARigal:1989a,BJNoye:1990a} and the references therein.

Hence when choosing numerical components to solve a more complicated problem, it is very important to understand the accuracy and stability properties (especially the latter) of the combined method.
In this work, we make an effort in this direction by analyzing general finite difference methods (FDM) discretizing the linear ADEs and a partially dissipative wave system in the context of method of lines.
In particular, it is assumed that an optimally accurate and stable finite-difference discretization operator (FDO) is used to discretize the advection term and an optimally accurate central FDO is chosen for the diffusion term; otherwise we do not impose any restriction on how these FDOs are selected and they can have arbitrary orders of accuracy.
Such a combination reflects a common practice in application areas including fluid mechanics, weather and climate predictions, and cell dynamics in tumor modeling; hence it excludes the central ones like FTCS and more recent Pad\'{e}-type compact methods~\cite{SKLele:1992a,XLiu:2013a}, which have enjoyed popularity in wave propagation and acoustics problems due to their very low numerical dissipation. 
Nevertheless, the authors do not see major difficulty extending the methodology presented here to central schemes.

Finite difference methods for linear ADEs have always been an active research area; however, most existing works concentrate on particular low to moderate-order schemes, where the von Neumann stability analysis or the spectral analysis are relatively easy to conduct as the characteristic function takes a simple form, see for example the inexhausted list of publications~\cite{JLSiemieniuch:1978a,KWMorton:1980a,DFGriffiths:1980a,ARigal:1989a,BJNoye:1990a,MLWitek:2008a,AMohebbi:2010a}.
In an earlier work by Tony F. Chan~\cite{TFChan:1984a}, the author proposed a recursive approach that is based on the Schur-Cohn theory to verify the stability of a method of arbitrary order; however, no direct stability result is derived for these general schemes.
To the best of the knowledge of the authors, the present work is a first attempt of the kind to derive a theory on the stability of a very general class of FDMs for linear ADEs and a derived partially dissipative wave system.

To this end, our analysis is carried out in three parts.
The first part focuses on the semi-discretized schemes for linear ADEs.
In particular, Section~\ref{sec:prob} introduces the model Cauchy problem of a linear ADE and the notations that are used throughout the paper.
We also explicitly construct in this section the FDOs with optimal accuracy given a stencil with arbitrary width for both the advection term and the diffusion term.
The stability analysis of the ODE system obtained by discretizing the linear ADE in space is provided in Section~\ref{sec:stab}; and we show that if a stable FDO is chosen for the advection term, then any central FDO for the diffusion term results in a stable ODE system.
The proof is based on a careful examine of the trajectory of eigenvalues (denoted by $\Lambda$) underlying this ODE system and showing that it stays in the left complex plane using classical theories by Iserles and Strang~\cite{AIserles:1983a} and a result due to Vietoris~\cite{LVietoris:1959a,RAskey:1974a} in positive trigonometric polynomials.

At the end of Section~\ref{sec:stab}, we obtain a global bound on $\Lambda$ as well its behavior close to the origin of the complex plane.
These results help us to prove the main theorems in Section~\ref{sec:full}, which composes the second part of this work.
In particular, we show that for the linear ADE, the stable spatial discretizations can be combined with any time-integrator to yield a conditionally stable fully-discretized method, as long as the temporal scheme is at least first-order accurate.
Additionally, we obtain an interesting instability result in the vanishing viscosity limit -- a high-order spatial discretization of the advection equation cannot be paired with some very popular low-order time-integrators to give a stable fully-discretized scheme, which include the first Euler method and the second-order two-stage Runge-Kutta scheme.
Although we focus on single-step and multi-stage explicit Runge-Kutta methods in this section, the analysis easily extends to other schemes such as the implicit and multi-step ones.

In the third part, the previous analysis is extended to a partially dissipative wave system, which serves as a model for flow equations where viscosity presents in the momentum equation but not in the pressure or energy equation.
Our analysis shows that even though dissipation appears only in one of two coupled equations, the trajectory of eigenvalues exhibits similar trait as that of a scalar ADE; hence it gives rise to conditionally stable fully discretized methods of arbitrary order accuracy.

An important simplification that we make is a periodic domain for both equations; hence the effects of boundary conditions are omitted in all three parts of the analysis.
However, our results remain valuable in the case of initial boundary value problems (IBVP), due to a classical theory by Godunov and Ryabenkii~\cite{SKGodunov:1963a}, see also~\cite{KWMorton:1980a}.
In particular, it was proved therein that in the limit $h\to0$ where $h$ is the grid size, the stability of a method for a periodic problem is necessary for the stability of this method when it is applied to solve an IBVP, no matter how the boundary condition is handled. 
Extending the present stability analysis to IBVPs along this line is work in progress and we hope to present it in a future publication soon.

The remainder of the paper is organized as follows.
The main analysis results are presented in Section~\ref{sec:prob}--Section~\ref{sec:wave}, as described in the three parts before.
All our major theoretical results are verified numerically in Section~\ref{sec:num}.
Finally, Section~\ref{sec:concl} concludes this paper and offers some further discussions.

\section{A Model Equation and Discrete Differential Operators}
\label{sec:prob}
We consider the Cauchy problem of the one dimensional (1D) linear advection-diffusion equation:
\begin{equation}\label{eq:prob_ade}
  w_t + w_x - \nu w_{xx} = 0
\end{equation}
on a closed interval $x\in\Omega=[0,\;1]$ and $t\in[0,\;T]$, where $\nu>0$ is the constant diffusivity.
The periodic boundary conditions $w(0,t) = w(1,t)$ and $w_x(0,t) = w_x(1,t)$ are supposed so that the analysis focuses on the spatial discretization of interior points.

The computational domain $\Omega$ is divided into $N$ uninform intervals with grid points $x_j=jh\;,\ j=0,\cdots,N$, where $h=1/N$ is the uniform cell size.
The semi-discretized solutions and the fully-discretized solutions are denoted $w_j(t)\approx w(x_j,t)$ and $w_j^n\approx w(x_j,t^n)$, respectively; here $t^n=n\Delta t$ is the $n$-th time stage and $\Delta t>0$ is the uniform time step size.
Due to the periodic boundary conditions, we follow the convention that $w_j\equiv w_{j+N}$ and $w_j^n\equiv w_{j+N}^n$ for all $j\in\mathbb{Z}$ and $n\ge0$. 
The method of lines (MOL) is adopted to first discretize (\ref{eq:prob_ade}) in space and then integrate the resulting system of ordinary differential equations (ODE) along the time ordinate.
In particular, the discrete approximation of the first-derivative in $x$ is denoted $\mathcal{D}_x$ and that of the second-derivative is denoted $\mathcal{D}_{xx}$; hence the ODE reads:
\begin{equation}\label{eq:prob_semi}
  \frac{dw_j}{dt} + \mathcal{D}_xw_j - \nu\mathcal{D}_{xx}w_j = 0\;,\quad\forall j\;.
\end{equation}
In this paper, we consider finite-difference differential operators (FDO) $\mathcal{D}_x$ and $\mathcal{D}_{xx}$ that are constructed with optimal accuracy using a continuous stencil.
In particular, the FDO $\mathcal{D}_x$ is given in general form by:
\begin{equation}\label{eq:prob_dx}
  \mathcal{D}_xw_j = \frac{1}{h}\sum_{k=-l}^ra_kw_{j+k}\;,
\end{equation}
where $l,r\ge0\;, l+r>0$ are the stencils to the left and the right, respectively; for the FDO $\mathcal{D}_{xx}$, we consider those with centered stencils $q>0$:
\begin{equation}\label{eq:prob_dxx}
  \mathcal{D}_{xx}w_j = \frac{1}{h^2}\sum_{k=-q}^qb_kw_{j+k}\;.
\end{equation}
The coefficients $\{a_k\}$ and $\{b_k\}$ are usually determined by accuracy requirement; and they can be uniquely determined if optimal accuracy is desired (see a later section).

Denoting the semi-discrete solution vector by:
\begin{equation}\label{eq:prob_sol}
  \bs{W} = [w_0,\;w_1,\;\cdots,\;w_{N-1}]^t\;,
\end{equation}
where $w_N$ is omitted due to the periodic boundary conditions, the ODE system determined by (\ref{eq:prob_semi}) is written in matrix form:
\begin{equation}\label{eq:prob_semi_mat}
  \frac{d\bs{W}}{dt} = -\frac{1}{h}\bs{A}\bs{W} + \frac{\nu}{h^2}\bs{B}\bs{W}\;.
\end{equation}
Here $\bs{A}$ and $\bs{B}$ are circulant matrices:
\begin{equation}\label{eq:prob_mat_coef}
  \bs{A} = \sum_{k=-l}^ra_k\bs{S}^k\;,\quad 
  \bs{B} = \sum_{k=-q}^qb_k\bs{S}^k\;,
\end{equation}
with $\bs{S}$ being given by:
\begin{equation}\label{eq:prob_mat_s}
  \bs{S} = \left[\begin{array}{ccccc}
    0 & 1 & \cdots & 0 & 0 \\
    0 & 0 & \cdots & 0 & 0 \\ \vspace*{-.24in} \\
    \vdots & \vdots & \ddots & \vdots & \vdots \\
    0 & 0 & \cdots & 0 & 1 \\
    1 & 0 & \cdots & 0 & 0 
  \end{array}\right]\;.
\end{equation}

The stability of the solutions to (\ref{eq:prob_semi_mat}) is determined from that of the coefficient matrix on the right hand side.
Defining $\bs{M} = -\bs{A}+R\bs{B}$, where $R=\nu/h$ is the reciprocal of the cell Reynolds number, a main focus is on the stability of the matrix $\bs{M}$.
It is clear that $0$ is an eigenvalue of $\bs{A}$, $\bs{B}$, and $\bs{M}$, as any consistent discretization preserves constant solutions.
To this end, we adopt the notion of {\it semistable} matrices, see for example~\cite{SLCampbell:1979a,DSBernstein:1995a}.
\begin{definition}\label{def:prob_semistab}
  A matrix $\bs{M}$ is {\it semistable} if any eigenvalue $\lambda$ of $\bs{M}$ satisfies either $\oname{Re}\lambda<0$ or $\lambda=0$ and it is regular.
\end{definition}
An equivalent definition of semistability is that the Jordan normal form of $\bs{M}$ can be arranged as $\left[\begin{array}{cc}\bs{J} & \bs{0} \\ \bs{0} & \bs{0}\end{array}\right]$, where the diagonal elements of $\bs{J}$ all have negative real parts.
It is well known that $\bs{M}$ is semistable if and only if the solution to the ODE system $d\bs{W}/dt = \bs{M}\bs{W}$ has a well defined limit as $t\to\infty$ for any initial data $\bs{W}(0)$.

At the end of this section we compute the FDO coefficients explicitly using Lagrangian interpolation polynomials for optimal accuracy.
The basic idea is that if $\mathcal{D}_x$ is $m$-th order accurate, then for all polynomial $P(x)\in\mathbb{P}^m$, where $\mathbb{P}^m$ denotes the space of polynomials of degree $\le m$, there is:
\begin{equation}\label{eq:prob_ddo_dx_exact}
  \mathcal{D}_xP_j = P'(x_j)\;,
\end{equation}
with $P_k=P(x_k)$ on the left-hand side.
Let the stencil $(l,r)$ of (\ref{eq:prob_dx}) be given, it is well known that the optimal order of accuracy for such a $\mathcal{D}_x$ is $m=l+r$. 
To find out the corresponding coefficients $\{a_k\}$, we define the Lagrangian interpolation polynomials for the points $\{x_{j+k}:\;-l\le k\le r\}$ as $l_k$:
\begin{equation}\label{eq:prob_ddo_dx_lag}
  l_k(x) = \frac{\prod_{-l\le\nu\le r,\,\nu\ne k}(x-x_{j+\nu})}{\prod_{-l\le\nu\le r,\,\nu\ne k}(x_{j+k}-x_{j+\nu})}\in\mathbb{P}^m\;,\quad -l\le k\le r\;,
\end{equation}
and $\{l_k\}$ composes a basis of $\mathbb{P}^m$.
For all $P(x)\in\mathbb{P}^m$, there is:
\begin{equation}\label{eq:prob_ddo_dx_lag_inter}
  P(x) = \sum_{k=-l}^rP(x_k)l_k(x) = \sum_{k=-l}^rP_kl_k(x)\;;
\end{equation}
combining it with (\ref{eq:prob_dx}) and (\ref{eq:prob_ddo_dx_exact}), we obtain:
\begin{equation}\label{eq:prob_ddo_dx_lsys}
  \frac{1}{h}\sum_{k=-l}^ra_kP_k = \sum_{k=-l}^rP_kl_k'(x_j)\;,\quad\forall (P_{-l},\,\cdots,\,P_r)\in\mathbb{R}^{m+1}\;.
\end{equation}
Thus the coefficients are given by:
\begin{equation}\label{eq:prob_ddo_dx}
  a_k = hl_k'(x_j) = \left\{\begin{array}{lcl}
    -\frac{(-1)^k}{k}\frac{l!r!}{(l+k)!(r-k)!} & & \textrm{ if } k\ne0\;, \\ \vspace*{-.1in} \\
    -\sum_{-l\le\nu\le r,\,\nu\ne0}\frac{1}{\nu} & & \textrm{ if } k=0\;.
  \end{array}\right.
\end{equation}
Similarly, given the stencil $q$ the optimal accuracy for $\mathcal{D}_{xx}$ is obtained when 
\begin{equation}\label{eq:prob_ddo_dxx_exact}
  \mathcal{D}_{xx}P_j = P''(x_j)\;,
\end{equation}
for all $P(x)\in\mathbb{P}^m$.
Note that on general grids this order is $2q-1$ whereas on uniform grids (as in this paper), the optimal order is $m=2q$.
Again, defining the Lagrangian interpolation polynomials corresponding to $\{x_{j+k}:\;-q\le k\le q\}$ as:
\begin{equation}\label{eq:prob_ddo_dxx_lag}
  \hat{l}_k(x) = \frac{\prod_{-q\le\nu\le q,\,\nu\ne k}(x-x_{j+\nu})}{\prod_{-q\le\nu\le q,\,\nu\ne k}(x_{j+k}-x_{j+\nu})}\in\mathbb{P}^m\;,\quad -q\le k\le q\;,
\end{equation}
one has:
\begin{equation}\label{eq:prob_ddo_dxx_lsys}
  \frac{1}{h^2}\sum_{k=-q}^qb_kP_k = \sum_{k=-q}^qP_k\hat{l}_k''(x_j)\;,\quad\forall (P_{-q},\,\cdots,\,P_q)\in\mathbb{R}^{m+1}\;.
\end{equation}
It follows immediately that the corresponding coefficients are:
\begin{equation}\label{eq:prob_ddo_dxx}
  b_k = h^2\hat{l}''_k(x_j) = \left\{\begin{array}{lcl}
    -\frac{2(-1)^k}{k^2}\frac{q!q!}{(q+k)!(q-k)!} & & \textrm{ if } k\ne0\;, \\ \vspace*{-.1in} \\
    -\sum_{k=1}^q\frac{2}{k^2} & & \textrm{ if } k=0\;.
  \end{array}\right.
\end{equation}
Later, we shall use these coefficients to prove the general stability result regarding the discretization~(\ref{eq:prob_semi}).

\section{Stability Analysis}
\label{sec:stab}
A benefit of using periodic boundary conditions is the circulant structure of the matrices $\bs{S}$, $\bs{A}$, $\bs{B}$, and $\bs{M}$.
In particular, the eigenvalues of $\bs{S}$ are $s_k = e^{i2k\pi/N}\,,\;k=1,\cdots,N$; hence the matrix $\bs{M}$ is diagonalizable with eigenvalues:
\begin{equation}\label{eq:stab_eigs}
  -\sum_{k=-l}^ra_ks^k + R\sum_{k=-q}^qb_ks^k\;,\quad s = s_1,\,s_2,\,\cdots,\,s_N\;.
\end{equation}
The stability analysis thusly reduces to studying whether the trajectory (fixing $R>0$):
\begin{equation}\label{eq:stab_traj}
  \Lambda(R) \eqdef \left\{\lambda_R(s) = -\sum_{k=-l}^ra_ks^k+R\sum_{k=-q}^qb_ks^k:\;s\in\mathbb{C},\;\abs{s}=1\right\}
\end{equation}
is contained in the left complex plane in the sense of Definition~\ref{def:prob_semistab}.
In addition, we denote by $\Lambda^\ast(R)$ the subset of $\Lambda(R)$ that is defined by excluding $\lambda_R(1)$, which is always $0$ by the consistency of the method.

For convenience, we also consider two extreme situations: when $R=0$, $\Lambda(0)$ is again given by (\ref{eq:stab_traj}), whereas when $R=\infty$, $\Lambda(\infty)$ is defined as:
\begin{equation}\label{eq:stab_traj_adv}
  \Lambda(\infty) \eqdef \left\{\lambda_\infty(s)=\sum_{k=-q}^qb_ks^k:\;s\in\mathbb{C},\;\abs{s}=1\right\}\;.
\end{equation}
Their subsets $\Lambda^\ast(0)$ and $\Lambda^\ast(\infty)$ are defined similarly.
It is fairly easy to see that the eigenvalues of $\bs{M}$ are pairwise sums of that of $-\bs{A}$ and $R\bs{B}$.
Thus if both components are semistable then $\bs{M}$ is likely to be semistable as well. 
This is to be made precise later. 

The full categorization of semistable discretization of the advection equation is accomplished decades ago by Iserles and Strang~\cite{AIserles:1983a} using the theory of order stars and revisited recently using more elementary techniques by Despr\'{e}s~\cite{BDespres:2009a}.
In short, the conclusion is that the optimally accurate $\mathcal{D}_x$ gives rise to a stable discretization if and only if $r\le l\le r+2$.
The case $r=l$ corresponds to a central-difference approximation to $\partial_x$, which is rarely used in practice for solving advection problems with explicit time integrators as the resulting scheme is unconditionally unstable.
In this paper, we suppose $\mathcal{D}_x$ is given by either $l=r+1$ or $l=r+2$, and provide a simple proof that the corresponding $-\bs{A}$ is semistable:
\begin{lemma}\label{lm:stab_dx}
  If $r+1\le l\le r+2$, then $\Lambda^\ast(0)$ is contained in the open left complex plane; hence in combination with the fact that $\Lambda(0) = \Lambda^\ast(0)\cup\{0\}$, one concludes that the corresponding coefficient matrix $-\bs{A}$ is semistable.
\end{lemma}
\begin{proof}
  Let us write $s=e^{i\theta}$, $0<\theta<2\pi$.
  Then following (\ref{eq:prob_ddo_dx}):
  \begin{equation}\label{eq:prob_stab_dx_real}
    \oname{Re}\,\lambda_0(s) = -\sum_{k=-l}^ra_k\cos k\theta = \sum_{-l\le k\le r,\,k\ne0}\frac{1}{k}+\sum_{-l\le k\le r,k\ne0}\frac{(-1)^k}{k}\frac{l!r!}{(l+k)!(r-k)!}\cos k\theta\;.
  \end{equation}
  In the case $l=r+1$, we have:
  \begin{align*}
    \oname{Re}\,\lambda_0(s) 
    =&\ -\frac{1}{r+1} + \frac{(-1)^{-r-1}}{-(r+1)}\frac{(r+1)!r!}{(2r+1)!}\cos(r+1)\theta - \sum_{k=1}^r\frac{(-1)^k2(r+1)!r!}{(r+1+k)!(r+1-k)!}\cos k\theta \\
    =&\ -\frac{(r+1)!r!}{(2r+2)!}\sum_{k=-r-1}^{r+1}\frac{(2r+2)!}{(r+1+k)!(r+1-k)!}(-1)^k\cos k\theta \\
    =&\ -\frac{(r+1)!r!}{(2r+2)!}(-1)^{r+1}\oname{Re}\,\left[e^{-i(r+1)\theta}\left(1-e^{i\theta}\right)^{2r+2}\right] \\
    =&\ -\frac{2^{2(r+1)}(r+1)!r!}{(2r+2)!}\left(\sin\frac{\theta}{2}\right)^{2(r+1)} < 0\;,\quad\forall\ 0 < \theta < 2\pi\;.
  \end{align*}
  Similarly in the case $l=r+2$, there is:
  \begin{displaymath}
    \oname{Re}\,\lambda_0(s)
    = -\frac{2^{2(r+2)}(2r+3)(r+2)!r!}{(2r+4)!}\left(\sin\frac{\theta}{2}\right)^{2(r+2)} < 0\;,\quad\forall\ 0 < \theta < 2\pi\;.
  \end{displaymath}
  Hence in both scenarios, $\Lambda^\ast(0)$ is contained in the open left complex plane. 
  Lastly, since $\lambda_0(1) = -\sum_{k=-l}^ra_k = -h\sum_{k=-l}^rl'_k(x_j) = 0$, the semistability of $-\bs{A}$ follows from the fact that the eigenvalues are given by $\lambda_0(e^{i2k\pi/N})\;, 1\le k\le N$.
\end{proof}

Next we consider the diffusion term.
Early work categorizing stable finite difference discretizations of the diffusion equation includes the work by Iserles on Pad\'{e}-type methods~\cite{AIserles:1985c}.
The technique therein is again to use order stars, which seems an overkill for this work in the context of method of lines.
Therefore, we use the theory of trigonometric polynomials to prove the related stability results regarding the semi-discretization $\mathcal{D}_{xx}$.
Particularly, the following result by Vietoris~\cite{LVietoris:1959a,RAskey:1974a} will be handy.
\begin{lemma}\label{lm:stab_viet}
  If $c_1\ge\cdots\ge c_n>0$ and $(2k)c_{2k}\le(2k-1)c_{2k-1}$ for all $k\ge1$, then:
  \begin{displaymath}
    \sum_{k=1}^nc_k\sin\,k\theta > 0\;,\quad\forall\ 0<\theta<\pi\;.
  \end{displaymath}
\end{lemma}
\noindent
Note that a sufficient but more convenient condition to verify is $kc_k\le(k-1)c_{k-1},\ \forall k\ge2$.

\begin{lemma}\label{lm:stab_dxx}
  Let $\mathcal{D}_{xx}$ with stencil $q>0$ be constructed according to (\ref{eq:prob_ddo_dxx}), then the trajectory $\Lambda^\ast(\infty)$ is contained in the open left complex plane, and $\Lambda(\infty) = \Lambda^\ast(\infty)\cup\{0\}$.
\end{lemma}
\begin{proof}
  To show $\Lambda(\infty)$ intersects the imaginary axis at $s=1$ is easy:
  \begin{displaymath}
    \lambda_\infty(1) = \sum_{k=-q}^qb_k1^k = \sum_{k=-q}^q\hat{l}_k''(x_j) = 0\;,
  \end{displaymath}
  where we used the fact that $\sum_{k=-q}^q\hat{l}_k(x)\equiv1$.

  Now let us focus on $\Lambda^\ast(\infty)$ and write $s=e^{i\theta}$, $0<\theta<2\pi$. 
  By direct computation and the symmetry $b_k=b_{-k}$, which is clearly seen from (\ref{eq:prob_ddo_dxx}), we have:
  \begin{displaymath}
    \lambda_\infty(s) = b_0 + \sum_{k=1}^qb_k(s^k+s^{-k}) = b_0 + 2\sum_{k=1}^qb_k\cos\,k\theta\in\mathbb{R}\;;
  \end{displaymath}
  and the purpose is to show the right-hand side is negative for all $0<\theta<2\pi$.
  To this end, we distinguish among three cases.

  \smallskip

  {\bf Case 1: $\theta=\pi$.}
  Now we have $s=-1$ and:
  \begin{displaymath}
    \lambda_\infty(-1) = b_0 + 2\sum_{k=1}^q(-1)^kb_k = -\sum_{k=1}^q\frac{2}{k^2}-\sum_{k=1}^q\frac{4}{k^2}\frac{q!q!}{(q+k)!(q-k)!} < 0\;.
  \end{displaymath}

  \smallskip

  {\bf Case 2: $\pi<\theta<2\pi$.}
  By defining $\phi = 2\pi-\theta\in(0,\;\pi)$, there is:
  \begin{displaymath}
    \lambda_\infty(s) = b_0+2\sum_{k=1}^qb_k\cos\,k(2\pi-\phi) = b_0+2\sum_{k=1}^qb_k\cos\,k\phi\;.
  \end{displaymath}
  Hence the situation reduces to the next one.

  \smallskip

  {\bf Case 3: $0<\theta<\pi$.}
  Proving $-b_0-2\sum_{k=1}^qb_k\cos\,k\theta>0$ is a topic in positive trigonometric polynomials; and a difficulty here is $b_k$ has alternating signs.
  To get around, let us change the varible $\theta\mapsto\pi-\theta$, so that the problem equivalently converts to show for all $0<\theta<\pi$:
  \begin{equation}\label{eq:stab_dxx_pos}
    -b_0-2\sum_{k=1}^qb_k\cos\,k(\pi-\theta) = \abs{b_0} + 2\sum_{k=1}^q\abs{b_k}\cos\,k\theta > 0\;.
  \end{equation}
  Let us define the right hand side as $f(\theta)$, then we have $f(\pi)=0$ (i.e., $\lambda_\infty(1)=0$) and:
  \begin{equation}\label{eq:stab_dxx_dec}
    f'(\theta) = -2\sum_{k=1}^qk\abs{b_k}\sin\,k\theta = -\sum_{k=1}^q\frac{4}{k}\frac{q!q!}{(q+k)!(q-k)!}\sin\,k\theta\;.
  \end{equation}
  If we can show $f'(\theta)<0$ for all $0<\theta<\pi$, then combining with $f(\pi)=0$ it follows immediately that $f(\theta)>0$ on $(0,\;\pi)$; whereas for the former, we just need to verify the condition below Lemma~\ref{lm:stab_viet}, i.e., for all $k\ge2$:
  \begin{displaymath}
    k\times\frac{4}{k}\frac{q!q!}{(q+k)!(q-k)!} \le (k-1)\times\frac{4}{k-1}\frac{q!q!}{(q+k-1)!(q-k+1)!}
    \ \Leftrightarrow\ 
    q-k+1\le q+k\;,
  \end{displaymath}
  which clearly holds and thusly ends the proof.
\end{proof}

To this end, we obtain the following stability theorem for linear ADEs:
\begin{theorem}\label{thm:stab_ade}
  Let $\mathcal{D}_x$ and $\mathcal{D}_{xx}$ be of optimal accuracy; and for the former there is $l=r+1$ or $l=r+2$; then the corresponding coefficient matrix $\bs{M}$ is semistable.
\end{theorem}
\begin{proof}
  The eigenvalues of $\bs{M}$ are given by:
  \begin{displaymath}
    \lambda_R(e^{i2k\pi/N}) = \lambda_0(e^{i2k\pi/N})+R\lambda_\infty(e^{i2k\pi/N})\;,\quad k=1,\cdots,N\;.
  \end{displaymath}
  By Lemma~\ref{lm:stab_dx}, $\oname{Re}\,\lambda_0(e^{i2k\pi/N})<0$ for all $1\le k\le N-1$ and $\lambda_0(1)=0$; and by Lemma~\ref{lm:stab_dxx}, $\oname{Re}\,\lambda_\infty(e^{i2k\pi/N})<0$ for all $1\le k\le N-1$ and $\lambda_\infty(1)=0$.
  Hence the desired result comes from the fact that $R>0$.
\end{proof}
{\it In the remainder of the paper, we only consider $\mathcal{D}_x$ and $\mathcal{D}_{xx}$ that satisfy the requirements of this theorem} -- hence by $\mathcal{D}_x$ we mean an optimally accurate FDO with stencil $l=r+1$ or $l=r+2$, even if such a construction is not explicitly stated\footnote{Similarly, in the case of a left-going wave, $\mathcal{D}_x$ refers to an optimally accurate FDO with $r=l+1$ or $r=l+2$, see Section~\ref{sec:wave}.}.

\smallskip

Lastly, we establish some results that will be useful in the stability analysis of fully-discretized methods in the next section.
The first one concerns the asymptotic behavior of the eigenvalue trajectory $\Lambda(R)$ near $s=1$.
\begin{theorem}\label{thm:stab_asym}
  Denote $x_R(\theta) = \oname{Re}\,\lambda_R(e^{i\theta})$ and $y_R(\theta) = \oname{Im}\,\lambda_R(e^{i\theta})$.
  Then as $\theta\to0$:
  \begin{enumerate}[(i)]
    \item There exists a $C_1>0$ that is determined by $\mathcal{D}_x$, such that $x_0=-C_1y_0^{2l}+O(y_0^{2l+1})$.
    \item If $R>0$, there exists a $C_2>0$ that is determined by both $\mathcal{D}_x$ and $\mathcal{D}_{xx}$, such that $x_R=-C_2y_R^{2l}+O(y_R^{2l+2})+R\left(-y_R^2+O(y_R^{\min(2r+4,2q+2)})\right)$.
  \end{enumerate}
\end{theorem}
\begin{proof}
  First of all, noticing that $\lambda_R(1)=0$, we have $x_R(0)=y_R(0)=0$ and the big-O terms makes sense.
  Now let us assume $R=0$, then by the construction of $\mathcal{D}_x$ there is:
  \begin{displaymath}
    \sum_{k=-l}^rk^ma_k = m!\delta_{1m}\;,\quad m=0,1,\cdots,l+r\;,
  \end{displaymath}
  where $\delta_{1m}$ is the Kronecker symbol that equals $1$ if $m=1$ and $0$ otherwise; and:
  \begin{displaymath}
    \sum_{k=-l}^rk^{l+r+1}a_k = (l+r+1)!c_1\;,\quad c_1\ne0\;.
  \end{displaymath}
  To this end on the one hand:
  \begin{displaymath}
    x_0(\theta)+iy_0(\theta) = -\sum_{k=-l}^ra_ke^{ik\theta} = -\sum_{k=-l}^ra_k\sum_{m=0}^\infty\frac{k^m}{m!}(i\theta)^m = -i\theta - c_1(i\theta)^{l+r+1} + O(\theta^{l+r+2})\;,
  \end{displaymath}
  and it follows that:
  \begin{displaymath}
    y_0(\theta) = -\theta + O(\theta^{2r+3})\quad\Rightarrow\quad
    \abs{y_0}^{2l} = \theta^{2l} + O(\theta^{2(l+r+1)})\ \textrm{ and }\ 
    \abs{y_0}^{2l+2} = \theta^{2l+2} + O(\theta^{2(l+r+2)})\;.
  \end{displaymath}
  On the other hand by Lemma~\ref{lm:stab_dx}:
  \begin{displaymath}
    x_0(\theta) = -c_2\left(\sin\frac{\theta}{2}\right)^{2l} = -c_2\left(\frac{\theta}{2}\right)^{2l}+O(\theta^{2l+2})\;.
  \end{displaymath}
  where $c_2>0$ depends only on $l$ and $r$.
  Combining these results, one has:
  \begin{displaymath}
    x_0(\theta) = - \frac{c_2}{2^{2l}}\abs{y_0(\theta)}^{2l} + O(\abs{y_0(\theta)}^{2l+2})\;,
  \end{displaymath}
  which completes the proof of the first part with $C_1=c_2/2^{2l}$.

  \medskip

  Now we suppose $R>0$; following Lamma~\ref{lm:stab_dxx}, $\lambda_\infty(s)$ is real and thusly:
  \begin{displaymath}
    x_R(\theta) = x_0(\theta) + R x_\infty(\theta)\quad\textrm{ and }\quad
    y_R(\theta) = y_0(\theta)\;.
  \end{displaymath}
  Because $\mathcal{D}_{xx}$ is optimally accurate and the coefficients $b_k$ are symmetric, one has:
  \begin{displaymath}
    \sum_{k=-q}^qk^mb_k = m!\delta_{2m}\;,\quad
    m=0,1,\cdots,2q+1\;,
  \end{displaymath}
  where $\delta_{2m}$ is the Kronecker delta symbol that equals $1$ when $m=2$ and $0$ otherwise; and:
  \begin{displaymath}
    \sum_{k=-q}^qk^{2q+2}b_k = (2q+2)!c_3\;,\quad c_3\ne0\;.
  \end{displaymath}
  To this end:
  \begin{displaymath}
    x_\infty(\theta) = \sum_{k=-q}^qb_ke^{ik\theta} = \sum_{k=-q}^qb_k\sum_{m=0}^\infty\frac{k^m}{m!}(i\theta)^m = -\theta^2 + (-1)^{q+1}c_3\theta^{2q+2} + O(\theta^{2q+4})\;.
  \end{displaymath}
  Combining with the estimates in the previous case, we obtain:
  \begin{align*}
    x_R(\theta) = -C_1\theta^{2l}+O(\theta^{2l+2}) + R\left(-\theta^2+O(\theta^{2q+2})\right)\;,\quad
    y_R(\theta) = -\theta + O(\theta^{2r+3})\;,
  \end{align*}
  and it follows immediately that:
  \begin{displaymath}
    x_R(\theta) = -C_1\abs{y_R(\theta)}^{2l}+O(\abs{y_R(\theta)}^{2l+2})+R\left(-\abs{y_R(\theta)}^2+O(\abs{y_R(\theta)}^{\min(2r+4,2q+2)})\right)\;,
  \end{displaymath}
  which completes the proof.
\end{proof}

The second result concerns a global bound on the trajectory $\Lambda(R)$.
\begin{theorem}\label{thm:stab_bound}
  There exists a positive number $L$ that only depends on $\mathcal{D}_x$ and $\mathcal{D}_{xx}$, such that for all $\theta\in[-\pi,\pi]$:
  \begin{equation}\label{eq:stab_bound}
    x_R(\theta)\le-RL(y_R(\theta))^2\;,
  \end{equation}
  where $x_R(\theta)$ and $y_R(\theta)$ are defined the same way as in Theorem~\ref{thm:stab_asym}.
\end{theorem}
\begin{proof}
  Seeing $x_R=x_0+Rx_\infty\le Rx_\infty$, we focus on the existence of such an $L$, so that:
  \begin{equation}\label{eq:stab_bound_infty}
    x_\infty \le -Ly_0^2\;.
  \end{equation}
  To achieve this, we'll show that there exist $L_1>0$ and $L_2>0$, such that:
  \begin{displaymath}
    y_0^2 \le L_1\theta^2
    \quad\textrm{ and }\quad
    x_\infty \le -L_2\theta^2
    \;,\quad\forall \theta\in[-\pi,\;\pi]\;;
  \end{displaymath}
  in addition, $L_1$ and $L_2$ are determined by $\mathcal{D}_{x}$ and $\mathcal{D}_{xx}$, respectively.
  To this end, the constant $L$ can be chosen as $L_2/L_1$. 

  \noindent
  {\bf Part 1.}
  First let us consider $L_1$ and compute the derivative of $y_0(\theta)=-\sum_{k=-l}^ra_k\sin k\theta$.
  Following a similar procedure as in the proof of Lemma~\ref{lm:stab_dx}, we obtain:
  \begin{displaymath}
    y_0'(\theta) = -1 - \frac{(-1)^{l-r}2^{l+r}l!r!}{(l+r)!}\left(\sin\frac{\theta}{2}\right)^{2r+2}\cos^{l-r-1}\theta\;,
  \end{displaymath}
  where $l=r+1$ or $l=r+2$.
  By the mean value theorem and using $y_0(0)=0$, we integrate the latest equation from $0$ to $\theta\in[-\pi,\;\pi]$ to obtain:
  \begin{align*}
    y_0(\theta) = -\theta - \theta\frac{(-1)^{l-r}2^{l+r}l!r!}{(l+r)!}\left(\sin\frac{\theta'}{2}\right)^{2r+2}\cos^{l-r-1}\theta'\;,
  \end{align*}
  where $\theta'$ is some number between $0$ and $\theta$.
  It follows immedinately that:
  \begin{displaymath}
    y_0^2\le \left(1+\frac{2^{l+r}l!r!}{(l+r)!}\right)^2\theta^2\;,\quad \forall \theta\in[-\pi,\;\pi]\;.
  \end{displaymath}

  \medskip

  \noindent
  {\bf Part 2.}
  Now we focus on $L_2$.
  Because $x_\infty(\theta) = b_0+2\sum_{k=-q}^qb_k\cos k\theta$ is an even function, we may assume $\theta\in[0,\;\pi]$.
  In the proof of the previous theorem, it was obtained that $x_\infty(\theta)=-\theta^2+(-1)^{q+1}c_3\theta^{2q+2}+O(\theta^{2q+4})$.
  Hence $f(\theta)\eqdef-x_\infty(\theta)/\theta^2$ belongs to $C[0,\;\pi]$ and it achieves the minimum $L_2$ at some $\theta'\in[0,\;\pi]$.
  Following Lemma~\ref{lm:stab_dxx} and its proof, $f(\theta)>0$ for all $0<\theta\le\pi$; combining with $f(0)=1$, we obtain immediately $L_2>0$.
  Because $f(\cdot)$ is determined by $\mathcal{D}_{xx}$, so is $L_2$.
\end{proof}

\section{Fully Discretized Systems}
\label{sec:full}
The previous stability result is extended to fully-discretized methods by combining a stable semi-discretization scheme with an explicit Runge-Kutta (ERK) method for time integration.
Suppose the spatial discretization gives rise to an ODE system:
\begin{equation}\label{eq:full_ode}
  \frac{d\bs{W}}{dt} = -\frac{1}{h}\bs{A}\bs{W} + \frac{\nu}{h^2}\bs{B}\bs{W}\;,
\end{equation}

This ODE system is integrated by an ERK method defined by the Butcher tableau~\cite{JCButcher:2016a}:
\begin{equation}\label{eq:full_butcher_erk}
  \begin{array}{c|cccccc}
    0      & 0      & 0      & 0      & \cdots & 0         & 0      \\
    c_2    & a_{21} & 0      & 0      & \cdots & 0         & 0      \\
    c_3    & a_{31} & a_{32} & 0      & \cdots & 0         & 0      \\
    \vdots & \vdots & \vdots & \vdots & \ddots & \vdots    & \vdots \\
    c_s    & a_{s1} & a_{s2} & a_{s3} & \cdots & a_{s,s-1} & 0      \\ \hline
           & b_1    & b_2    & b_3    & \cdots & b_{s-1}   & b_s
  \end{array}\;,
\end{equation}
where $c_i = \sum_{j=1}^{i-1}a_{ij}\;,\ 2\le i\le s$, $s$ is the stage number, and $\sum_{j=1}^sb_j=1$.
Then updating the solution from one time step $t_n$ to the next $t_{n+1}=t_n+\delta t$ follows:
\begin{align*}
  \bs{W}^{(1)} &= \bs{W}^n\;, \\
  \bs{W}^{(i)} &= \bs{W}^n + \sum_{j=1}^{i-1}a_{ij}\delta t\left(-\frac{1}{h}\bs{A}\bs{W}^{(j)}+\frac{\nu}{h^2}\bs{B}\bs{W}^{(j)}\right)\;,\quad 2\le i\le s\;, \\
  \bs{W}^{n+1} &= \bs{W}^n + \sum_{j=1}^sb_j\delta t\left(-\frac{1}{h}\bs{A}\bs{W}^{(j)}+\frac{\nu}{h^2}\bs{B}\bs{W}^{(j)}\right)\;.
\end{align*}
Let $\mu=\delta t/h$, which is usually used in practice to determine the time step size by the Courant condition, then $-(\delta t/h)\bs{A} + (\nu\delta t/h^2)\bs{B} = \mu\bs{M}$ with $\bs{M}=-\bs{A}+R\bs{B}$ as before.
Then one has $\bs{W}^{(i)} = p_{i-1}(\mu\bs{M})\bs{W}^n\;,\ 1\le i\le s$ and $\bs{W}^{n+1}=p_s(\mu\bs{M})\bs{W}^n$, where $p_i\;,\ 0\le i\le s$ is a polynomial of degree no larger than $i$ defined recursively by:
\begin{displaymath}
  p_0(z)     = 1\;;\quad
  p_{i-1}(z) = 1 + \sum_{j=1}^{i-1}a_{ij}z\,p_{j-1}(z)\;,\quad 2\le i\le s\;;\quad
  p_s(z)     = 1 + \sum_{j=1}^sb_jz\,p_{j-1}(z)\;.
\end{displaymath}
Suppose the method is $m$-th order accurate, one must have $p_s(z) = \sum_{k=0}^mz^k/k!+O(z^{m+1})$ and thusly $s\ge m$.
The {\it stability region} of the ERK method~(\ref{eq:full_butcher_erk}) is defined:
\begin{equation}\label{eq:full_erk_sr}
  \mathcal{S} = \{z\in\mathbb{C}:\;\abs{p_s(z)}\le 1\}\;.
\end{equation}
Because the numerical solution at a time step $t_n=n\delta t$ is $\bs{W}^n=\left[p_s(\mu\bs{M})\right]^n\bs{W}^0$, one sees that a necessary condition for the numerical method to be stable is $\mu\lambda\in\mathcal{S}$, where $\lambda$ is any eigenvalue of $\bs{M}$.
Note that fixing $h$, the eigenvalues of $\bs{M}$ are contained in a closed set $\Lambda(R)$ given by (\ref{eq:stab_traj}), one expects $\mu\Lambda(R)$ shrinks to zero from the left as $\delta t\to0$.
Here $\mu\Lambda(R)$ is defined as the set of $\mu\lambda$ for all $\lambda\in\Lambda(R)$.

For all spatial discretizations chosen according to Theorem~\ref{thm:stab_ade}, $\bs{M}$ is semistable and $\Lambda^\ast(R)$ is contained in the open left complex plane.
In this case, it is not difficult to see that $\mu\Lambda(R)\subseteq\mathcal{S}$ is also a sufficient condition for ensuring the semistability of $p_s(\mu\bs{M})$, hence the corresponding fully-discretized method is stable.
The following theorem shows that for any time-integrator that is at least first-order accurate, the fully-discretized method is always conditionally stable.
\begin{theorem}\label{thm:full_hord}
  Let a spatial discretization in Theorem~\ref{thm:stab_ade} be  paired with an explicit Runge-Kutta method with order $m\ge1$, then there exist positive numbers $\alpha_0$, $\beta_0$, and $\gamma_0$, which only depend on the discretizations $\mathcal{D}_x$, $\mathcal{D}_{xx}$, and the time-integrator, such that for all $\delta t>0$ satisfying:
  \begin{equation}\label{eq:full_cfl}
    \delta t < \nu\gamma_0\quad\textrm{ and }\quad 
    \left(\alpha_0+\frac{\nu\beta_0}{h}\right)\frac{\delta t}{h} < 1\;,
  \end{equation}
  the fully-discretized method is stable.
\end{theorem}
\noindent
{\bf Remark}. 
The second of constraints~(\ref{eq:full_cfl}) takes the same form of usual Courant conditions for advection-diffusion equations.
\begin{proof}
  By Theorem~\ref{thm:stab_bound}, there exists a positive number $L>0$ that only depends on $\mathcal{D}_x$ and $\mathcal{D}_{xx}$, such that:
  \begin{displaymath}
    x \le -RLy^2\;\quad\forall x+iy\in\Lambda(R)\;.
  \end{displaymath}
  Furthermore, it is clearly that there exists positive numbers $Y_0$, $X_0$, and $X_1$ that depends only on $\mathcal{D}_x$ and $\mathcal{D}_{xx}$, such that $\abs{y}<Y_0$ and $\abs{x}<X_0+RX_1$ for all $x+iy\in\Lambda(R)$.

  Thus for any $x+iy\in\mu\Lambda(R)$, one has:
  \begin{displaymath}
    x \le -\frac{RL}{\mu}y^2 = -\frac{\nu L}{\delta t}y^2\;,\ \abs{x}<\frac{\delta t(X_0+RX_1)}{h}\;,\ \textrm{ and }\ \abs{y}<\frac{\delta tY_0}{h}\;.
  \end{displaymath}
  To this end, it suffices to show that there exists a $\varepsilon_0>0$ and $M_0>0$, such that:
  \begin{equation}\label{eq:full_quadset}
    \mathcal{D}(\varepsilon_0,M_0) \eqdef \{z=x+iy:\;-M_0\varepsilon_0^2<x<-M_0y^2,\,\abs{y}<\varepsilon_0\} \subseteq \mathcal{S}\;,
  \end{equation}
  with $\mathcal{S}$ being the stability region of the chosen ERK method.
  Indeed, if (\ref{eq:full_quadset}) is true, then for all $\delta t$ such that:
  \begin{displaymath}
    \delta t < \min\left(\frac{\varepsilon_0h}{X_0+RX_1},\;\frac{\varepsilon_0h}{Y_0},\;\frac{\nu L}{M_0}\right)\;, 
  \end{displaymath}
  one has $\mu\Lambda(R)\backslash\{0\}\subseteq\mathcal{D}(\varepsilon_0,M_0)\subseteq\mathcal{S}$; thus the fully-discretized method is stable.
  Hence the constants can be chosen as $\alpha_0 = \max(X_0/(M_0\varepsilon_0^2),\,Y_0/\varepsilon_0)$, $\beta_0=X_1/(M_0\varepsilon_0^2)$, and $\gamma_0=L/M_0$.

  Next we focus on (\ref{eq:full_quadset}).
  Because the order of the time-integrator is $m\ge1$, one has:
  \begin{displaymath}
    p_s(z) = 1 + z + C(z)z^2\;,\quad
  \end{displaymath}
  where $C(z)$ is a polynomial in $z$ and it is bounded by some constant $C_0$ for all $\abs{z}<1$.
  For all such $z$, there is the estimate:
  \begin{displaymath}
    \abs{p_s(z)}^2 
    \le \abs{1+z}^2+2C_0\abs{z}^2\abs{1+z}+C_0^2\abs{z}^4 
    \le \abs{1+z}^2 + \left(4C_0+C_0^2\right)\abs{z}^2\;.
  \end{displaymath}
  Denote the set of all pairs of positive numbers $(\varepsilon,M)$ such that $\mathcal{D}(\varepsilon,M)\subseteq\{z\in\mathbb{C}:\;\abs{z}<1\}$ by $\mathcal{P}$; we aim at finding a $(\varepsilon_0,M_0)\in\mathcal{P}$ such that $\mathcal{D}(\varepsilon_0,M_0)\subseteq\mathcal{S}$.

  Let us fix $(\varepsilon,M)\in\mathcal{P}$. 
  Then for any $z=x+iy\in\mathcal{D}(\varepsilon,M)$ such that $y\ne0$, we may write $x=-\tilde{M}y^2$ where $\tilde{M}>M$ and $\abs{y}<\sqrt{M/\tilde{M}}\varepsilon_0<\varepsilon_0$.
  Using the previous estimate, one has:
  \begin{displaymath}
    \abs{p_s(z)}^2 
    \le \abs{1-\tilde{M}y^2+iy}^2+(4C_0+C_0^2)\abs{-\tilde{M}y^2+iy}^2 
    = 1-y^2(2\tilde{M}-C_1-C_1\tilde{M}^2y^2)\;,
  \end{displaymath}
  where $C_1=1+4C_0+C_0^2$.

  To this end, let us fix $M_0>C_1$, then there exists a $\varepsilon_0'>0$ such that for all $0<\varepsilon_1<\varepsilon_0'$, one has $M_0\varepsilon_1^2<1/C_1$ and $(\varepsilon_1,M_0)\in\mathcal{P}$; in addition given any $z=x+iy\in\mathcal{D}(\varepsilon_1,M_0)$ with $y\ne0$ and $x=-\tilde{M}y^2$, there is:
  \begin{displaymath}
    2\tilde{M}-C_1-C_1\tilde{M}^2y^2 > 2\tilde{M}-C_1-C_1\tilde{M}M_0\varepsilon_1^2 > 2\tilde{M}-C_1-\tilde{M} > M_0-C_1>0\;,
  \end{displaymath}
  hence following the previous analysis one obtains $\abs{p_s(z)}<1$ and $z\in\mathcal{S}$.

  Lastly, let us consider the intersection of $\mathcal{D}$ and the real axis.
  In particular, let $z=x\in(-M_1(\varepsilon_0')^2,\;0)$ (which is contained in $(-1,0)$):
  \begin{displaymath}
    \abs{p_s(z)} = \abs{1+x+x^2C(x)} \le 1+x+C_0x^2 = 1 - \abs{x}(1-C_0\abs{x})\;.
  \end{displaymath}
  Thus for any $\varepsilon_0>\in(0,\;\varepsilon_0')$ such that $\varepsilon_0<1/\sqrt{C_0M_0}$, the set $\mathcal{D}(\varepsilon_0,M_0)$ satisfies (\ref{eq:full_quadset}).
\end{proof}

In the second half of this section, we prove some interesting results in the special case $R=0$, i.e., solving the advection equation $w_t+w_x=0$.
General stability result seems to be difficult to derive in this case since the scaling between the real part and the imaginary part of $\Lambda(0)$ near $z=0$ depends highly on the order of the method (see Theorem~\ref{thm:stab_asym}).
For this reason, we focus on several widely used ERK listed below, most of which can be found in the text by Hairer, N{\o}rsett, and Wanner~\cite{EHairer:1993a} whereas others include the strong stability preserving (SSP) methods~\cite{SGottlieb:2009a,SGottlieb:2001a} and the low-storage methods~\cite{JHWilliamson:1980a}:
\begin{enumerate}
  \item The first-order forward Euler method (FE), where $p_s(z)=1+z$.
  \item Any two-stage, second-order method (RK2), where $p_s(z)=1+z+\frac{1}{2}z^2$, such as the original method by Runge and a later SSP version.
  \item Any three-stage, third-order method (RK3), where $p_s(z)=1+z+\frac{1}{2}z^2+\frac{1}{6}z^3$, which includes the earlier one by Heun and a later SSP version.
  \item A low-storage, four-stage, third-order method (LSRK3) by Runge, whose Butcher tableau is given by:
    \begin{equation}\label{eq:full_butcher_lsrk3}
      \begin{array}{c|cccc}
        0   & \\ 
        1/2 & 1/2 \\ 
        1   & 0   & 1 \\ 
        1   & 0   & 0   & 1 \\ \hline
            & 1/6 & 2/3 & 0 & 1/6
      \end{array}
    \end{equation}
    Correspondingly, $p_s(z)=1+z+\frac{1}{2}z^2+\frac{1}{6}z^3+\frac{1}{12}z^4$.
  \item Any four-stage, fourth-order method (RK4), where $p_s(z)=1+z+\frac{1}{2}z^2+\frac{1}{6}z^3+\frac{1}{24}z^4$.
    Note that this is the highest-order ERK one can construct, such that the order is the same as the number of stages.
\end{enumerate}
The next result shows that in general a high-order spatial discretization cannot be paired with some low-order temporal schemes to yield a conditionally stable method under the usual Courant condition.
\begin{theorem}\label{thm:full_adv_instab}
  Let the advection equation $w_t+w_x=0$ be discretized by an $\mathcal{D}_x$ with the upwind stencil $l\ge2$, and let the time-integrator be FE.
  Then for any positive number $\mu_c$, the method is unstable in the limit $h\to0$ if the time step size is calculated as $\delta t = \mu_c h$.

  Furthermore, if either RK2 or LSRK3 is used, any fully-discretized method built in combination with an $\mathcal{D}_x$ such that $l\ge3$ is unstable in the limit $h\to0$ given the fixed Courant number $\mu_c>0$.
\end{theorem}
\begin{proof}
  By Theorem~\ref{thm:stab_asym}, the trajectory $\Lambda(0)$ behaves as $x=-C_1y^{2l}+O(y^{2l+1})$ for some constant $C_1>0$ near the origin; thus the trajectory $\mu_c\Lambda(0)$ behaves as $x=-C_1\mu_c^{1-2l}y^{2l}+O(y^{2l+1})$ in the same limit.
  First let us suppose the time-integrator is given by the forward Euler method, then $p_s(z)=1+z$.
  Consider the value of $p_s(z)$ along the path $x_0(\theta)+iy_0(\theta)\in\mu_c\Lambda(0)$ as $\theta\to0$, one has:
  \begin{align*}
    \abs{p_s(x_0+iy_0)}^2 
    &= \left[1 - C_1\mu_c^{1-2l}y_0^{2l}+O(y_0^{2l+1})\right]^2+y_0^2 \\
    &= 1 + y_0^2\left[1-2C_1\mu_c^{1-2l}y_0^{2l-2}+C_1^2\mu_c^{2-4l}y_0^{4l-2}+O(y_0^{2l-1})\right]\;.
  \end{align*}
  Following the proof of Theorem~\ref{thm:stab_asym}, we have $y_0(\theta)=-\mu_c\theta+O(\theta^{2r+3})$. 
  Hence there exists a $\theta_0>0$ such that for all $\abs{\theta}<\theta_0$, $y_0(\theta)\ne0$ as long as $\theta\ne0$ and the quantity in the square bracket on the right-hand side of the latest equation is positive.
  Thus for all $\abs{\theta}<\theta_0$ and $\theta\ne0$, $\abs{p_s(x_0(\theta)+iy_0(\theta))}>1$.
  For sufficiently small $h$, there is always eigenvalues of $p_s(\mu_c\bs{M})$ correspond to a non-zero $\theta$ with magnitude small than $\theta_0$; hence for these $h$, the corresponding fully-discretized method is unstable.

  \smallskip

  Next, suppose RK2 is used, where $p_s(z) = 1 + z + \frac{1}{2}z^2$.
  Consider the path $x_0(\theta)+iy_0(\theta)$ as $\theta\to0$ again:
  \begin{align*}
    \abs{p_s(x_0+iy_0)}^2
    &= 1 + 2x_0 + 2x_0^2 + x_0^3 + \frac{1}{4}x_0^4 + \frac{1}{4}y_0^4+x_0y_0^2+\frac{1}{2}x_0^2y_0^2 \\
    &= 1 + y_0^4\left[\frac{1}{4}-2C_1\mu_c^{1-2l}y_0^{2l-4}+O(y_0^{2l-3})\right]\;.
  \end{align*}
  If $l\ge3$, one has $2l-4>0$ in the square bracket and the instability of the fully-discretized method for sufficiently small $h$ follows a similar argument as before.

  \smallskip

  For the LSRK3 method, where $p_s(z) = 1 + z + \frac{1}{2}z^2 + \frac{1}{6}z^3 + \frac{1}{12}z^4$, along the path $x_0(\theta)+iy_0(\theta)$ as $\theta\to0$ one has:
  \begin{align*}
    \abs{p_s(x_0+iy_0)}^2
    &= \left[1+x-\frac{1}{2}y_0^2+\frac{1}{12}y_0^4+O(y_0^{2l+1})\right]^2 + y_0^2\left[1-\frac{1}{6}y_0^2+O(y_0^{2l})\right]^2 \\
    &= 1 + y_0^4\left[\frac{1}{12}-\frac{1}{18}y_0^2+\frac{1}{144}y_0^8-2C_1\mu_c^{1-2l}y_0^{2l-4}+O(y_0^{2l-3})\right]\;.
  \end{align*}
  And the conclusion follows from a similar argument if $l\ge3$.
\end{proof}
{\bf Remark}.
This theorem concerns the stability with fixed Courant number, i.e., the ratio between $\delta t$ and $h$ is kept constant while refining the grids.
It does not, however, indicate instability in the limit $\delta t\to0$ while fixing $h$.
For example in the case of the FE time-integrator, substituting $\mu=\frac{\delta t}{h}$ one has $y_0=-\frac{\delta t}{h}(\theta+O(\theta^{2r+3})$ as well as an estimate on the leading terms of the quantity in the square bracket as $1-2C_1\mu^{1-2l}y_0^{2l-2}=1-2C_1h\theta^{2l-2}/\delta t+O(\theta^{2r+3})$.
Let the grid be fixed, the smallest non-zero $\theta$ corresponds to an eigenvalue of the discrete system that scales linearly with $h$, thus the square bracket could be negative in the limit $\delta t\to0$ hence it renders a stable fully-discretized method.

Finally, we demonstrate a simple criterion for ERKs, which could easily be extended to other time-integrators such as the implicit and multi-step ones, so that they result in a conditionally stable method when combined with any $\mathcal{D}_x$ that is given by Lemma~\ref{lm:stab_dx}.
\begin{theorem}\label{thm:full_adv_stab}
  Defining the set $\mathcal{D}^-(\varepsilon)=\{z\in\mathbb{C}:\;\abs{z}<\varepsilon\ \textrm{ and }\ \oname{Re}\;z<0\}$.
  If there exists a $\varepsilon_0>0$ such that $\mathcal{D}^-(\varepsilon_0)\subseteq\mathcal{S}$, then for any $\mathcal{D}_x$ as given by Lemma~\ref{lm:stab_dx}, there exists a positive number $\alpha_0>0$ that is independent of $h$ and $\delta t$ such that the corresponding fully-discretized method is stable for all $\delta t>0$ such that $\alpha_0\frac{\delta t}{h}<1$.
\end{theorem}
\begin{proof}
  The eigenvalues of the discrete system belong to $\mu\Lambda(0)$.
  Because the trajectory $\Lambda(0)$ is closed and independent of $h$, in the view of Lemma~\ref{lm:stab_dx} all but one zero eigenvalue of the fully-discretized system has negative real part.
  Furthermore, there exists an $X_0>0$ such that for all $z\in\Lambda(0)$, $\abs{z}<M_0$; hence the modulus of any eigenvalue belonging to $\mu\Lambda(0)$ is smaller than $M_0\delta t/h$.
  To this end, for all $\delta t>0$ such that $M_0\delta t/h<\varepsilon_0$, one has $\mu\Lambda^\ast(0)\subseteq\mathcal{D}^-(\varepsilon_0)\subseteq\mathcal{S}$, i.e., the method is stable under the Courant condition with $\alpha_0=M_0/\varepsilon_0$.
\end{proof}
As the theorem does not require an explicit time-integrator, an immediately consequence is that one can obtain an unconditionally stable method by combining such $\mathcal{D}_x$ with any A-stable time-integrator\footnote{Hence it has to be implicit.}, because the $\varepsilon_0$ in the theorem can be chosen as an arbitrarily large number.
Within the range of explicit methods, using this theorem we obtain the following stability result for several third-order and fourth order Runge-Kutta methods.
\begin{corollary}\label{cor:full_adv}
  The method obtained by combining an $\mathcal{D}_x$ given in Lemma~\ref{lm:stab_dx} with any $s$-stage, $s$-th order accurate ERK with $s=3$ or $s=4$ is conditionally stable.
\end{corollary}
\begin{proof}
  We just need to verify that there exists a $\varepsilon_0>0$ such that for all $z\in\mathcal{D}^-(\varepsilon_0)$, $\abs{p_s(z)}<1$, where $p_s(z)=1+z+\frac{1}{2}z^2+\frac{1}{6}z^3$ or $p_s(z)=1+z+\frac{1}{2}z^2+\frac{1}{6}z^3+\frac{1}{24}z^4$.

  \medskip
  
  \noindent
  {\bf Case 1}. Let $s=3$ and $z=x+iy$ with $x<0$, one has:
  \begin{align*}
    \abs{p_s(z)}^2 
    &= \left(1+x-\frac{1}{2}y^2+xO(\abs{z})\right)^2+\left(y-\frac{1}{6}y^3+xO(\abs{z})\right)^2 \\
    &= 1 + x(2+O(\abs{z})) - \frac{1}{12}y^2\left(1-\frac{1}{3}y^2\right) < 1\;,
  \end{align*}
  for sufficient small $\abs{z}$ and $x<0$.

  \medskip

  \noindent
  {\bf Case 2}. Let $s=4$ and $z=x+iy$ with $x<0$, one similarly has:
  \begin{align*}
    \abs{p_s(z)}^2 
    &= \left(1+x-\frac{1}{2}y^2+\frac{1}{24}y^4+xO(\abs{z})\right)^2 + \left(y-\frac{1}{6}y^3+xO(\abs{z})\right)^2 \\
    &= 1 + x(2+O(\abs{z})) - \frac{1}{72}y^2\left(1-\frac{1}{8}y^2\right) < 1\;,
  \end{align*}
  for sufficiently small $\abs{z}$ and $x<0$.
\end{proof}

\section{A Partially Dissipative Wave System}
\label{sec:wave}
In this section, we extend the analysis to a partially dissipative wave system:
\begin{equation}\label{eq:wave_eqn}
  \left\{\begin{array}{l}
    v_t + p_x - \nu v_{xx} = 0\;, \\ \vspace*{-.1in} \\
    p_t + v_x = 0\;.
  \end{array}\right.
\end{equation}
This serves as a model for a common practice in many areas of fluid mechanics that a viscous stress presents in the momentum equation whereas the energy (or pressure) equation is not complemented by dissipation of heat.

The hyperbolic part of (\ref{eq:wave_eqn}) contains a right going wave $(v+p)/2$ and a left going wave $(v-p)/2$. 
To this end, we apply a left-biased FDO $\mathcal{D}_x^-$ to discretize $\partial_x(v+p)$ and a right-biased one $\mathcal{D}_x^+$ to discretize $\partial_x(v-p)$:
\begin{subequations}\label{eq:wave_semi}
  \begin{align}
    \label{eq:wave_semi_v}
    &\frac{dv_j}{dt} + \frac{1}{2}\mathcal{D}_x^-(v_j+p_j) - \frac{1}{2}\mathcal{D}_x^+(v_j-p_j) - \nu\mathcal{D}_{xx}v_j = 0 \\
    \label{eq:wave_semi_p}
    &\frac{dp_j}{dt} + \frac{1}{2}\mathcal{D}_x^-(v_j+p_j) + \frac{1}{2}\mathcal{D}_x^+(v_j-p_j) = 0\;,
  \end{align}
\end{subequations}
where the three operators $\mathcal{D}_x^-$, $\mathcal{D}_x^+$, and $\mathcal{D}_{xx}$ are respectively given by
\begin{equation}\label{eq:wave_ddo}
  \mathcal{D}_x^-v_j = \frac{1}{h}\sum_{k=-l^-}^{r^-}a_k^-v_{j+k}\;,\quad
  \mathcal{D}_x^+v_j = \frac{1}{h}\sum_{k=-l^+}^{r^+}a_k^+v_{j+k}\;,\quad
  \mathcal{D}_{xx}v_j = \frac{1}{h^2}\sum_{k=-q}^qb_kv_{j+k}\;,
\end{equation}
such that they satisfy the requirement of Theorem~\ref{thm:stab_ade} (hence $l^--r^-,r^+-l^+\in\{1,2\}$).

To write (\ref{eq:wave_semi}) in matrix form, let us define the solution vectors
\begin{equation}\label{eq:wave_sol}
  \bs{V} = [v_0,\;v_1,\;\cdots,\;v_{N-1}]\;,\quad
  \bs{P} = [p_0,\;p_1,\;\cdots,\;p_{N-1}]\;;
\end{equation}
then the ODE system is given by:
\begin{equation}\label{eq:wave_semi_mat}
  \frac{d}{dt}\left[\begin{array}{c}
    \bs{V} \\ \bs{P} 
  \end{array}\right] = 
  - \frac{1}{2h}\left[\begin{array}{cc}
    \bs{A}^--\bs{A}^+ & \bs{A}^-+\bs{A}^+ \\ \bs{A}^-+\bs{A}^+ & \bs{A}^--\bs{A}^+
  \end{array}\right]
  \left[\begin{array}{c}
    \bs{V} \\ \bs{P} 
  \end{array}\right] + 
  \frac{\nu}{h^2}\left[\begin{array}{cc}
    \bs{B} & \bs{0} \\ \bs{0} & \bs{0}
  \end{array}\right]
  \left[\begin{array}{c}
    \bs{V} \\ \bs{P} 
  \end{array}\right]\;,
\end{equation}
where $\bs{A}^-=\sum_{k=-l^-}^{r^-}a_k^-\bs{S}^k$, $\bs{A}^+=\sum_{k=-l^+}^{r^+}a_k^+\bs{S}^k$, and $\bs{B} = \sum_{k=-q}^qb_k\bs{S}^k$, with $\bs{S}$ given by (\ref{eq:prob_mat_s}).

Define the reciprocal cell Reynolds number $R=\nu/h$ as before, we want to investigate the stability of the matrix:
\begin{equation}\label{eq:wave_mat}
  \bs{M} = -\frac{1}{2}
  \left[\begin{array}{cc}
    \bs{A}^--\bs{A}^+ & \bs{A}^-+\bs{A}^+ \\
    \bs{A}^-+\bs{A}^+ & \bs{A}^--\bs{A}^+
  \end{array}\right] + 
  R\left[\begin{array}{cc}
    \bs{B} & \bs{0} \\
    \bs{0} & \bs{0}
  \end{array}\right]
\end{equation}
By assumption, both $-\bs{A}^-$ and $\bs{A}^+$ are semistable; hence the first term of (\ref{eq:wave_mat}) is also semistable following the similarity transform:
\begin{displaymath}
  -\frac{1}{2}\left[\begin{array}{cc}
    \bs{A}^--\bs{A}^+ & \bs{A}^-+\bs{A}^+ \\
    \bs{A}^-+\bs{A}^+ & \bs{A}^--\bs{A}^+
  \end{array}\right] = 
  \left[\begin{array}{cc}
    \bs{I} & \bs{I} \\ \bs{I} & -\bs{I}
  \end{array}\right] 
  \left[\begin{array}{cc}
    -\bs{A}^- & \bs{0} \\ \bs{0} & \bs{A}^+
  \end{array}\right] 
  \left[\begin{array}{cc}
    \bs{I} & \bs{I} \\ \bs{I} & -\bs{I}
  \end{array}\right]^{-1}\;.
\end{displaymath}
Thus $\bs{M}$ is the sum of a semistable matrix and a symmetric semistable matrix.
However, it is well known that the set of semistable matrices is not closed under matrix summation; to see this, the next example shows that the sum of a semistable matrix (even with semistable symmetric part) and a symmetric semistable matrix could be unstable:
\begin{displaymath}
  \left[\begin{array}{cc}
    -2 & 1/\epsilon \\ \epsilon & -2
  \end{array}\right] + 
  \left[\begin{array}{cc}
    -2 & 1 \\ 1 & -2
  \end{array}\right] = 
  \left[\begin{array}{cc}
    -4 & 1+1/\epsilon \\ 1+\epsilon & -4
  \end{array}\right]\;,
\end{displaymath}
where $\epsilon>0$ is sufficiently small.
Hence in this article, we take a different approach and show that $\bs{M}$ given by (\ref{eq:wave_mat}) is semistable for all $R>0$.

\begin{lemma}\label{lm:wave_eigs} 
  The Jordan normal form of $\bs{M}$ can be arranged into $N$ $2\times2$ blocks, each of which is (1) either diagonal with eigenvalues:
  \begin{subequations}\label{eq:wave_eigs}
    \begin{align}
      \label{eq:wave_eigs_1}
      \lambda_{k,1} &= \frac{1}{2}\left\{Rb(s_k)-\left[a^-(s_k)-a^+(s_k)\right]+\sqrt{R^2b(s_k)^2+\left[a^-(s_k)+a^+(s_k)\right]^2}\right\} \\
      \label{eq:wave_eigs_2}
      \textrm{and }\ \lambda_{k,2} &= \frac{1}{2}\left\{Rb(s_k)-\left[a^-(s_k)-a^+(s_k)\right]-\sqrt{R^2b(s_k)^2+\left[a^-(s_k)+a^+(s_k)\right]^2}\right\}\;,
    \end{align}
  \end{subequations}
  where $k$ is an integer between $1$ and $N$, and $s_k=e^{i2k\pi/N}$; or (2) a $2\times2$ Jordan block, whose eigenvalue has negative real part.
  Here the three Laurent polynomials are given by:
  \begin{equation}\label{eq:wave_lpoly}
    a^-(s) = \sum_{k=-l^-}^{r^-}a_k^-s^k\;,\quad
    a^+(s) = \sum_{k=-l^+}^{r^+}a_k^+s^k\;,\quad
    b(s) = \sum_{k=-q}^qb_ks^k\;.
  \end{equation}
  (Hence $\bs{A}^{\pm}=a^{\pm}(\bs{S})$ and $\bs{B}=b(\bs{S})$.)
\end{lemma}
\begin{proof}
  Clearly $s_k\;,\ 1\le k\le N$ are the distinct eigenvalues of $\bs{S}$; and we can assume the corresponding eigenvectors are $\bs{U}_k\in\mathbb{C}^N$.
  For any $s_k$, we define a $2\times2$ complex matrix:
  \begin{equation}\label{eq:wave_redmat}
    \bs{M}_k = \left[\begin{array}{cc}
      m_{k,11} & m_{k,12} \\ m_{k,21} & m_{k,22}
    \end{array}\right] \eqdef 
    \left[\begin{array}{cc} 
      Rb(s_k)-\frac{1}{2}[a^-(s_k)-a^+(s_k)] & -\frac{1}{2}[a^-(s_k)+a^+(s_k)] \\
      -\frac{1}{2}[a^-(s_k)+a^+(s_k)] & -\frac{1}{2}[a^-(s_k)-a^+(s_k)]
    \end{array}\right]\;
  \end{equation}
  then it is not difficult to verify that it has two eigenvalues $\lambda_{k,1}$ and $\lambda_{k,2}$ given by~(\ref{eq:wave_eigs}).
  Let $\bs{M}_k = \bs{V}_k\bs{J}_k\bs{V}_k^{-1}$ where $\bs{J}_k$ is the Jordan normal form of $\bs{M}_k$ and denote $\bs{V}_k=[v_{k,ij}]_{1\le i,j\le2}$, then by direct computation:
  \begin{align*}
    \bs{M}(\bs{V}_k\otimes\bs{U}_k) &= 
    \left[\begin{array}{cc} 
      Rb(\bs{S})-\frac{1}{2}[a^-(\bs{S})-a^+(\bs{S})] & -\frac{1}{2}[a^-(\bs{S})+a^+(\bs{S})] \\
      -\frac{1}{2}[a^-(\bs{S})+a^+(\bs{S})] & -\frac{1}{2}[a^-(\bs{S})-a^+(\bs{S})]
    \end{array}\right] \left[\begin{array}{cc}
      v_{k,11}\bs{U}_k & v_{k,12}\bs{U}_k \\
      v_{k,21}\bs{U}_k & v_{k,22}\bs{U}_k
    \end{array}\right] \\
    &= \left[\begin{array}{cc}
      (m_{k,11}v_{k,11}+m_{k,12}v_{k,21})\bs{U}_k & (m_{k,11}v_{k,12}+m_{k,12}v_{k,22})\bs{U}_k \\
      (m_{k,21}v_{k,11}+m_{k,22}v_{k,21})\bs{U}_k & (m_{k,21}v_{k,12}+m_{k,22}v_{k,22})\bs{U}_k
    \end{array}\right] \\
    &= (\bs{M}_k\bs{V}_k)\times\bs{U}_k = (\bs{V}_k\bs{J}_k)\otimes\bs{U}_k 
     = (\bs{V}_k\otimes\bs{U}_k)\bs{J}_k\;.
  \end{align*}
  Here we used the fact that $a^-(\bs{S})\bs{U}_k=a^-(s_k)\bs{U}_k$, $a^+(\bs{S})\bs{U}_k=a^+(s_k)\bs{U}_k$, and $b(\bs{S})\bs{U}_k=b(s_k)\bs{U}_k$.
  Hence, the Jordan normal form of $\bs{M}$ is composed of diagonal blocks $\bs{J}_1\,,\cdots\,,\bs{J}_N$.
  Now we focus on each such block $\bs{J}_k$.

  \smallskip

  \noindent
  {\bf Case 1: $\bs{J}_k$ is diagonal.}
  From $\bs{M}(\bs{V}_k\otimes\bs{U}_k)=(\bs{V}_k\otimes\bs{U}_k)\bs{J}_k$, we see immediately that $\lambda_{k,1}$ and $\lambda_{k,2}$ are eigenvalues of $\bs{M}$ with eigenvectors $\bs{V}_{k,1}\otimes\bs{U}_k$ and $\bs{V}_{k,2}\otimes\bs{U}_k$, respectively.
  Note that this also includes the case when $k=N$, i.e., $s_N=1$ and $\bs{M}_N$ is the zero matrix.

  \smallskip

  \noindent
  {\bf Case 2: $\bs{J}_k$ is a $2\times2$ Jordan block.}
  In this case, it is necessary $\lambda_{k,1}=\lambda_{k,2}$ and hence $R^2b(s_k)^2+[a^-(s_k)+a^+(s_k)]^2=0$.
  By Lemma~\ref{lm:stab_dxx}, $b(s_k)$ is a negative real number; hence $\oname{Re}(a^-(s_k)+a^+(s_k))=0$.
  To this end:
  \begin{displaymath}
    \oname{Re}\,\lambda_{k,1}=\oname{Re}\,\lambda_{k,2} 
    = \frac{1}{2}\left\{Rb(s_k)+\oname{Re}\,[-a^-(s_k)+a^+(s_k)]\right\}
    = \frac{1}{2}\left\{Rb(s_k)+\oname{Re}\,[-2a^-(s_k)]\right\} < 0\;,
  \end{displaymath}
  where we used in addition that $\oname{Re}(-a^-(s_k))<0$ by Lemma~\ref{lm:stab_dx}.
\end{proof}

Similar as in the ADE case, we define the set $\Lambda(R)$:
\begin{equation}\label{eq:wave_traj}
  \Lambda(R) = \left\{\frac{1}{2}\left[Rb(s)-a^-(s)+a^+(s)\pm\sqrt{R^2b(s)^2+[a^-(s)+a^+(s)]^2}\right]:\;\abs{s}=1\right\}\;,
\end{equation}
then all eigenvalues of $\bs{M}$ are on the trajectory defined by $\Lambda(R)$.
Now we are in a position of showing that the semi-discretization (\ref{eq:wave_semi}) is always stable.

\begin{theorem}\label{thm:wave_stab}
  The matrix $\bs{M}$ given by (\ref{eq:wave_mat}) is semistable for all $R>0$.
\end{theorem}
\begin{proof}
  By Lemma~\ref{lm:wave_eigs}, it suffices to show that for all $s$ such that $\abs{s}=1$ and $s\ne1$, there is:
  \begin{equation}\label{eq:wave_stab_negre}
    \oname{Re}\left[Rb(s)-[a^-(s)-a^+(s)]\pm\sqrt{R^2b(s)^2+[a^-(s)+a^+(s)]^2}\right] < 0\;.
  \end{equation}
  Note that by the (stable) choice of the discrete differential operators, we have:
  \begin{displaymath}
    b(s) < 0\;,\quad
    \oname{Re}(-a^-(s)) < 0\;,\quad\textrm{ and }\quad
    \oname{Re}( a^+(s)) < 0\;;
  \end{displaymath}
  hence (\ref{eq:wave_stab_negre}) is equivalent to (we suppress the dependence on $s$ for simplicity and use overbar to denote the complex conjugate):
  \begin{align*}
    &\pm\oname{Re}\sqrt{R^2b^2+(a^-+a^+)^2} < -Rb + \oname{Re}\,(a^--a^+) \\
    \Longleftrightarrow\quad&
    \left(\sqrt{R^2b^2+(a^-+a^+)^2}+\sqrt{R^2b^2+(\overline{a^-}+\overline{a^+})^2}\right)^2 < \left[-2Rb+2\oname{Re}\,(a^--a^+)\right]^2 \\
    \Longleftrightarrow\quad& \sqrt{R^4b^4+2R^2b^2\oname{Re}(a^-+a^+)^2+\abs{a^-+a^+}^4} \\
    <\ &\ R^2b^2 - 4Rb\oname{Re}(a^--a^+) + R^2\left[2\left(\oname{Re}(a^--a^+)\right)^2-\oname{Re}(a^-+a^+)^2\right]\;.
  \end{align*}
  For easier calculation, the latest inequality is rewritten:
  \begin{equation}\label{eq:wave_stab_to_prove}
    \sqrt{R^4b^4 + C_1R^2b^2 + C_2} < R^2b^2 + D_1Rb + D_2\;,
  \end{equation}
  where:
  \begin{align*}
     &C_1 = 2\oname{Re}(a^-+a^+)^2\;, 
    &&C_2 = \abs{a^-+a^+}^4 \ge 0\;, \\
     &D_1 = -4\oname{Re}(a^--a^+) < 0\;, 
    &&D_2 = 2\left(\oname{Re}(a^--a^+)\right)^2-\oname{Re}(a^-+a^+)^2\;.
  \end{align*}
  Taking the square of both sides of (\ref{eq:wave_stab_to_prove}), we obtain the equivalent inequality:
  \begin{equation}\label{eq:wave_stab_polyr}
    0 < 2D_1R^3b^3 + (D_1^2+2D_2-C_1)R^2b^22 + 2D_1D_2Rb + (D_2^2-C_2)\;.
  \end{equation}
  In what follows, we show that all coefficients of this $R$-polynomial are positive:
  \begin{enumerate}[i)]
    \item $2D_1R^3b^3$. 
      The coefficient is clearly positive since $b<0$ and $D_1<0$.
    \item $2D_1D_2Rb$. It suffices to show $D_2>0$; to this end, let us write $a^-=E_1+iE_2$ and $a^+=F_1+iF_2$, where $E_{1,2}, F_{1,2}\in\mathbb{R}$ (so $E_1>0$ and $F_1<0$) and compute:
      \begin{displaymath}
        D_2 = (E_1+F_1)^2-8E_1F_1+(E_2+F_2)^2 > 0\;.
      \end{displaymath}
    \item $(D_1^2+2D_2-C_1)R^2b^2$.
      Noticing that $D_2 = \frac{1}{8}D_1^2-\frac{1}{2}C_1$, the positivity of the coefficient comes from:
      \begin{displaymath}
        D_1^2+2D_2-C_1 = D_1^2+2D_2-2\left(\frac{1}{8}D_1^2-D_2\right) = \frac{3}{4}D_1^2+4D_2 > 0\;.
      \end{displaymath}
    \item $(D_2^2-C_2)$. Following (ii):
      \begin{displaymath}
        D_2 = \abs{a^-+a^+}^2-8E_1F_1 > \abs{a^-+a^+}^2\;.
      \end{displaymath}
      Thus $D_2^2 > \abs{a^-+a^+}^4 = C_2$.
  \end{enumerate}
\end{proof}

In the second half of this section, we establish similar bounds on the trajectory $\Lambda(R)$ as in the ADE case.
For the general combination of $\mathcal{D}_x^-$, $\mathcal{D}_x^+$, and $\mathcal{D}_{xx}$, such a bound is difficult to establish, as in the limit $R\to+\infty$, half of the eigenvalues converge to zero.
Note that in practice, the same discretization technique is frequently applied to waves in both directions.
To this end, we consider a special case when $\mathcal{D}_x^-$ and $\mathcal{D}_x^+$ are {\it symmetric}, that is, $l^-=r^+$, $r^-=l^+$, and thusly $a_k^-+a_{-k}^+=0$ for all $-l^-=-r^+\le k\le r^-=l^+$.
And we obtain a similar bound as in Theorem~\ref{thm:stab_bound}, which is given below.
\begin{theorem}\label{thm:wave_sym}
  Suppose $\mathcal{D}_x^-$ and $\mathcal{D}_x^+$ are symmetric, then there exists a constant $L>0$ that is determined by $\mathcal{D}_x^{\pm}$ and $\mathcal{D}_{xx}$ such that for all $x+iy\in\Lambda(R)$, there is $x\le-RL\abs{y}^2$.
\end{theorem}
\begin{proof}
  Let $a^-(s)=x_0^-(\theta)+iy_0^-(\theta)$ and $a^+(s)=x_0^+(\theta)+iy_0^+(\theta)$; then due to the symmetry we have $x^-_0(\theta)=-x^+_0(\theta)$ and $y_0^-(\theta)=y_0^+(\theta)$; thusly:
  \begin{displaymath}
    a^-(s) - a^+(s) = 2x_0^-(\theta)\;,\quad\textrm{ and }\quad
    a^-(s) + a^+(s) = 2iy_0^-(\theta)\;.
  \end{displaymath}
  Using in addition $b(s)=x_\infty(\theta)$, any element $x(\theta)+iy(\theta)$ of $\Lambda(R)$ can be written as:
  \begin{equation}\label{eq:wave_sym_eig}
    x(\theta)+iy(\theta) = \frac{1}{2}\left(Rx_\infty(\theta)-2x_0^-(\theta)\pm\sqrt{R^2x_\infty(\theta)^2-4y_0^-(\theta)^2}\right)\;.
  \end{equation}
  By the construction of the FDOs and previous results, one has $x_\infty(\theta)\le0$ and $-x_0^-(\theta)\le0$.
  To proceed, given any $\theta\in[-\pi,\;\pi]$ we distinguish between two scenarioes.

  \smallskip

  \noindent
  {\bf Case 1: $R^2x_\infty(\theta)^2-4y_0^-(\theta)^2\ge0$.}
  In this case, $y(\theta)=0$ and
  \begin{align*}
    x(\theta) 
    &= \frac{1}{2}\left(Rx_\infty(\theta)-2x_0^-(\theta)\pm\sqrt{R^2x_\infty(\theta)^2-4y_0^-(\theta)^2}\right) \\
    &\le \frac{1}{2}\left(Rx_\infty(\theta)-2x_0^-(\theta)+\sqrt{R^2x_\infty(\theta)^2}\right) = -x_0^-(\theta)\le0 = -RL\abs{y(\theta)}^2\;,
  \end{align*}
  for any positive number $L$.

  \smallskip

  \noindent
  {\bf Case 2: $R^2x_\infty(\theta)^2-4y_0^-(\theta)^2<0$.}
  In this case:
  \begin{displaymath}
    x(\theta) = \frac{1}{2}\left(Rx_\infty(\theta)-2x_0^-(\theta)\right)\quad\textrm{ and }\quad
    \abs{y(\theta)} = \frac{1}{2}\sqrt{4y_0^-(\theta)^2-R^2x_\infty(\theta)^2}\;.
  \end{displaymath}
  Following the proof of Theorem~\ref{thm:stab_bound}, there exists an $L_1>0$ such that $y_0^-(\theta)^2\le L_1\theta^2$ for all $\theta\in[-\pi,\;\pi]$ and an $L_2>0$ such that $x_\infty(\theta)\le -L_2\theta^2$.
  To this end, we have:
  \begin{displaymath}
    x(\theta) \le \frac{1}{2}Rx_\infty(\theta) \le -\frac{RL_2}{2}\theta^2\quad\textrm{ and }\quad
    \abs{y(\theta)}^2 \le y_0^-(\theta)^2 \le L_1\theta^2\;.
  \end{displaymath}
  Thus the desired estimate is established with $L_2/(2L_1)$.
\end{proof}

Using the same argument as in the proof of Theorem~\ref{thm:full_hord}, we obtain the following conditional stability result:
\begin{theorem}\label{thm:wave_sym_hord}
  We consider the full discretization of (\ref{eq:wave_eqn}) combining a symmetric pair of $\mathcal{D}_x^-$ and $\mathcal{D}_x^+$ and $\mathcal{D}_{xx}$ with an explicit Runge-Kutta method with order $p\ge1$ in the context of method of lines.
  Then there exist positive numbers $\alpha_0$, $\beta_0$, and $\gamma_0$, which are determined by $\mathcal{D}_x^{\pm}$, $\mathcal{D}_{xx}$, and the chosen time-integrator, such that for all $\delta t>0$ satisfying:
  \begin{equation}\label{eq:wave_sym_cfl}
    \delta t < \nu\gamma_0\quad\textrm{ and }\quad 
    \left(\alpha_0+\frac{\nu\beta_0}{h}\right)\frac{\delta t}{h} < 1\;,
  \end{equation}
  the fully-discretized method is stable.
\end{theorem}

Finally, we note that the {\bf Case 1} in the proof of Theorem~\ref{thm:wave_sym} can appear quite frequently, especially when $\nu$ is large.
To this end, let us define a finite subset $\Lambda_h(R)$ of $\Lambda(R)$, which contains those eigenvalues corresponding to $\theta=2k\pi h$ with $k\in\mathbb{Z}$, where $h=1/N$ is a cell size for a grid dividing $\Omega=[0,\;1]$ into $N$ uniform sub-intervals.
The set $\Lambda_h^\ast(R)$ is defined similarly by excluding the eigenvalues corresponding to $\theta=0$ from $\Lambda_h(R)$.
It is clear that the eigenvalues of the finite dimensional ODE system~(\ref{eq:wave_semi_mat}) are given by $\Lambda_h(R)$.

\begin{theorem}\label{thm:wave_sym_h}
  Suppose $\mathcal{D}_x^-$ and $\mathcal{D}_x^+$ are symmetric, then:
  \begin{enumerate}[(1)]
    \item there exists a $\nu_1>0$ that depends on $\mathcal{D}_x^{\pm}$ and $\mathcal{D}_{xx}$, such that for all $\nu>\nu_1$, $\Lambda_h(R)\subset\mathbb{R}^-$ for all $h>0$.
    \item suppose $\nu<1/(2\pi)$, then $\Lambda_h(R)\cap(\mathbb{C}\backslash\mathbb{R})\ne\varnothing$ for sufficiently small $h$.
  \end{enumerate}
\end{theorem}
\begin{proof}
  Using the same notation as before, the elements of $\Lambda_h(R)$ are given by (\ref{eq:wave_sym_eig}) with $\theta = 2k\pi h$, $k\in\mathbb{Z}$.
  In the rest of the proof, the dependence on $\theta$ is frequently suppressed for simplicity.

  \smallskip 

  \noindent
  {\bf (1)}. 
  By definition, both $x_\infty$ and $y_0^-$ are analytic functions of $\theta$.
  Additionally, following the proof of Theorem~\ref{thm:stab_asym}, there is:
  \begin{displaymath}
    \lim_{\theta\to0}\frac{y_0^-(\theta)}{\theta} = 1\quad\textrm{ and }\quad
    \lim_{\theta\to0}\frac{x_\infty(\theta)}{\theta^2} = -1\;;
  \end{displaymath}
  and according to (\ref{lm:stab_dxx}), $x_\infty(\theta)<0$ for all $-\pi\le\theta\le\pi$ and $\theta\ne0$.
  Hence there exist two constants $C_1>0$ and $C_2>0$ that are determined by $\mathcal{D}_x^-$ and $\mathcal{D}_{xx}$, respectively, such that for all $\theta\in[-\pi,\;\pi]$:
  \begin{displaymath}
    \abs{y_0^-(\theta)}\le C_1\abs{\theta}\quad\textrm{ and }\quad
    \abs{x_\infty(\theta)} \ge C_2\theta^2\;.
  \end{displaymath}
  To this end, the term inside the square root of (\ref{eq:wave_sym_eig}) is:
  \begin{displaymath}
    R^2x_\infty^2-4(y_0^-)^2 \ge R^2C_2^2\theta^4-4C_1^2\theta^2 = 4C_1^2\theta^2\left(\frac{\nu^2C_2^2\theta^2}{4C_1^2h^2}-1\right)
  \end{displaymath}
  Noticing that if $x(\theta)+iy(\theta)\in\Lambda_h^\ast(R)$ and $-\pi\le\theta\le\pi$, one must have $\abs{\theta}\ge2\pi h$; hence for all these eigenvalues:
  \begin{displaymath}
    R^2x_\infty^2-4(y_0^-)^2 \ge 4C_1^2\theta^2\left(\frac{\nu^2C_2^2\pi^2}{C_1^2}-1\right)\;,
  \end{displaymath}
  which is positive for all $\nu>\nu_1$ that is defined as $\nu_1= C_1/(C_2\pi)$.
  Hence for these $\nu$, all eigenvalues in $\Lambda_h(R)=\Lambda_h^\ast(R)\cup\{0\}$ are real; and by Theorem~\ref{thm:wave_stab}, they'are all non-positive.

  \smallskip

  \noindent
  {\bf (2)}.
  Let us consider an eigenvalue in $\Lambda_h(R)$ corresponding to $\theta_1=2\pi h$, denoted by $\lambda_1=x_1+iy_1$.
  Following the proof of Theorem~\ref{thm:stab_asym} again, there exist analytic functions $d_1(\theta)$, $d_2(\theta)$, and $d_3(\theta)$, such that:
  \begin{displaymath}
    x_0^-(\theta)=C_3\theta^{2l^-}+d_1(\theta)\theta^{2l^-+2}\;,\quad
    y_0^-(\theta)=\theta+d_2(\theta)\theta^{2r^-+3}\;,\quad
    x_\infty(\theta)=-\theta^2+d_3(\theta)\theta^{2q+2}\;,
  \end{displaymath}
  where $C_3>0$ is a constant determined by $\mathcal{D}_x^-$.
  Let $D_k>0$ be an upperbound of $d_k(\theta)$ on the closed interval $[-\pi,\;\pi]$ for $k=1,2,3$, then one has the estimates on the term under the square root of (\ref{eq:wave_sym_eig}):
  \begin{align*}
    R^2x_\infty(\theta)^2-4y_0^-(\theta)^2 
    &=   R^2(\theta^4-2d_3\theta^{2q+4}+d_3^2\theta^{4q+4})-4(\theta^2+2d_2\theta^{2r+4}+d_2^2\theta^{4r+6}) \\
    &\le R^2\theta^4(1+2D_3\theta^{2q}+D_3^2\theta^{4q})-4\theta^2(1-2D_2\theta^{2r+2})\;.
  \end{align*}
  Suppose $0<\nu<1/(2\pi)$ and $h>0$ is sufficiently small such that:
  \begin{displaymath}
    \theta_1=2\pi h<\min\left((4D_3)^{-\frac{1}{2q}},\;(2D_3^2)^{-\frac{1}{4q}},\;(4D_2)^{-\frac{1}{2r+2}}\right)\;,
  \end{displaymath}
  then:
  \begin{align*}
    R^2x_\infty(\theta_1)^2-4y_0^-(\theta_1)^2 < R^2\theta_1^4\left(1+\frac{1}{2}+\frac{1}{2}\right)-4\theta_1^2\left(1-\frac{1}{2}\right)
    = 2\theta_1^2(4\pi^2\nu^2-1) < 0\;.
  \end{align*}
  Hence $\lambda_1=x_1+iy_1\in\mathbb{C}\backslash\mathbb{R}$.
\end{proof}

\section{Numerical Examples}
\label{sec:num}
At last, we verify the previous results with numerical examples.
Particularly, Section~\ref{sec:num_ade} focuses on the advection-diffusion equation and Section~\ref{sec:num_wave} concentrates on the semi-dissipative wave system.
For notation simplicity, we denote an optimally accurate $\mathcal{D}_x$ with left stencil $l$ and right stencil $r$ by $\mathcal{D}_x^{l,r}$; according to Lemma~\ref{lm:stab_dx}, only $\mathcal{D}_x^{r+1,r}$ and $\mathcal{D}_x^{r+2,r}$ (and their symmetric counterpart in the case of the wave equation) will be considered.
Similarly, the optimally accurate $\mathcal{D}_{xx}$ using $2q+1$ grid points on a centered stencil is denoted $\mathcal{D}_{xx}^q$.

\subsection{Linear advection-diffusion equations}
\label{sec:num_ade}
First let us consider the semi-discretized systems and in Figures~\ref{fg:num_ade_semi}, four combinations of $\mathcal{D}_x$ and $\mathcal{D}_{xx}$ are considered: (a) $\mathcal{D}_x^{3,1}$ and $\mathcal{D}_{xx}^2$ -- they have comparable relatively low order of accuracy, (b) $\mathcal{D}_x^{21,20}$ and $\mathcal{D}_{xx}^{20}$ -- they have comparable and high order of accuracy, (c) $\mathcal{D}_x^{3,1}$ and $\mathcal{D}_{xx}^{20}$, and (d) $\mathcal{D}_x^{21,20}$ and $\mathcal{D}_{xx}^2$.
For each of the four combinations, $\Lambda(R)$ corresponding to a variety choices of $R$ is plotted.
\begin{figure}[ht]\centering
  \begin{subfigure}{.45\textwidth}
    \includegraphics[trim=1.2in 0.2in 1.2in 0.2in, clip, width=\textwidth]{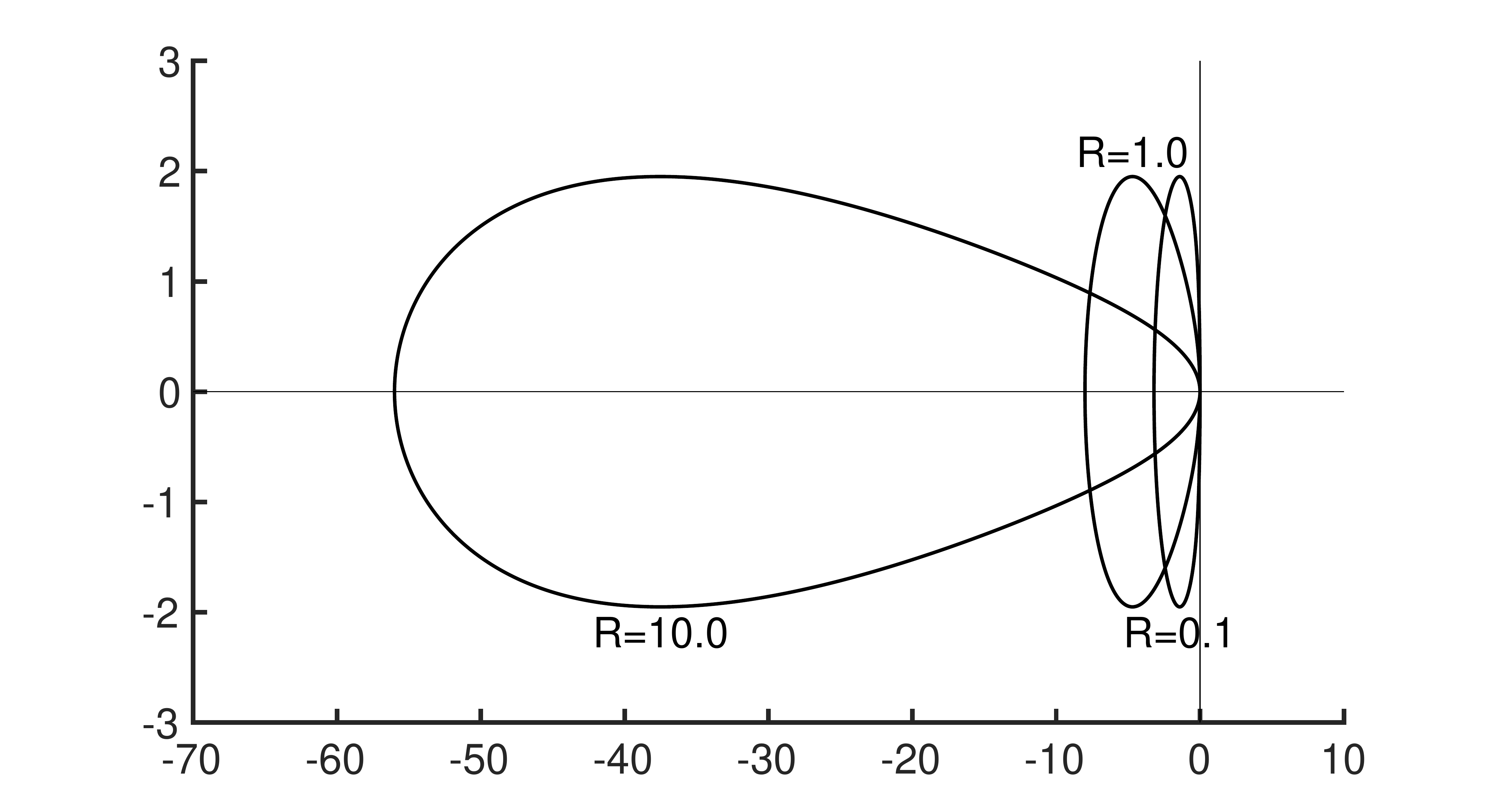}
    \caption{$\mathcal{D}_x^{3,1}$ and $\mathcal{D}_{xx}^2$.}
    \label{fg:num_ade_semi_a}
  \end{subfigure}
  \begin{subfigure}{.45\textwidth}
    \includegraphics[trim=1.2in 0.2in 1.2in 0.2in, clip, width=\textwidth]{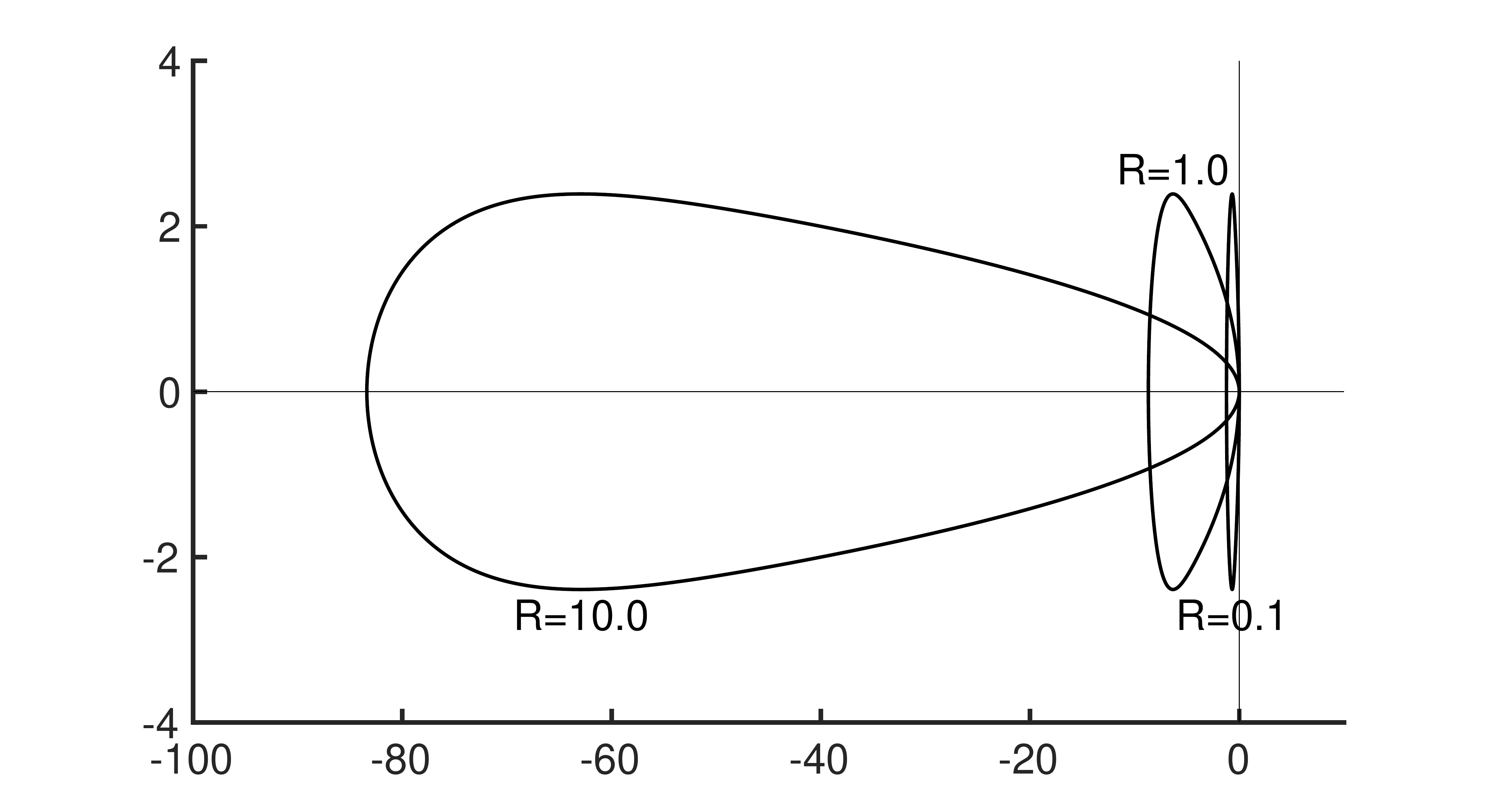}
    \caption{$\mathcal{D}_x^{21,20}$ and $\mathcal{D}_{xx}^{20}$.}
    \label{fg:num_ade_semi_b}
  \end{subfigure}
  \begin{subfigure}{.45\textwidth}
    \includegraphics[trim=1.2in 0.2in 1.2in 0.2in, clip, width=\textwidth]{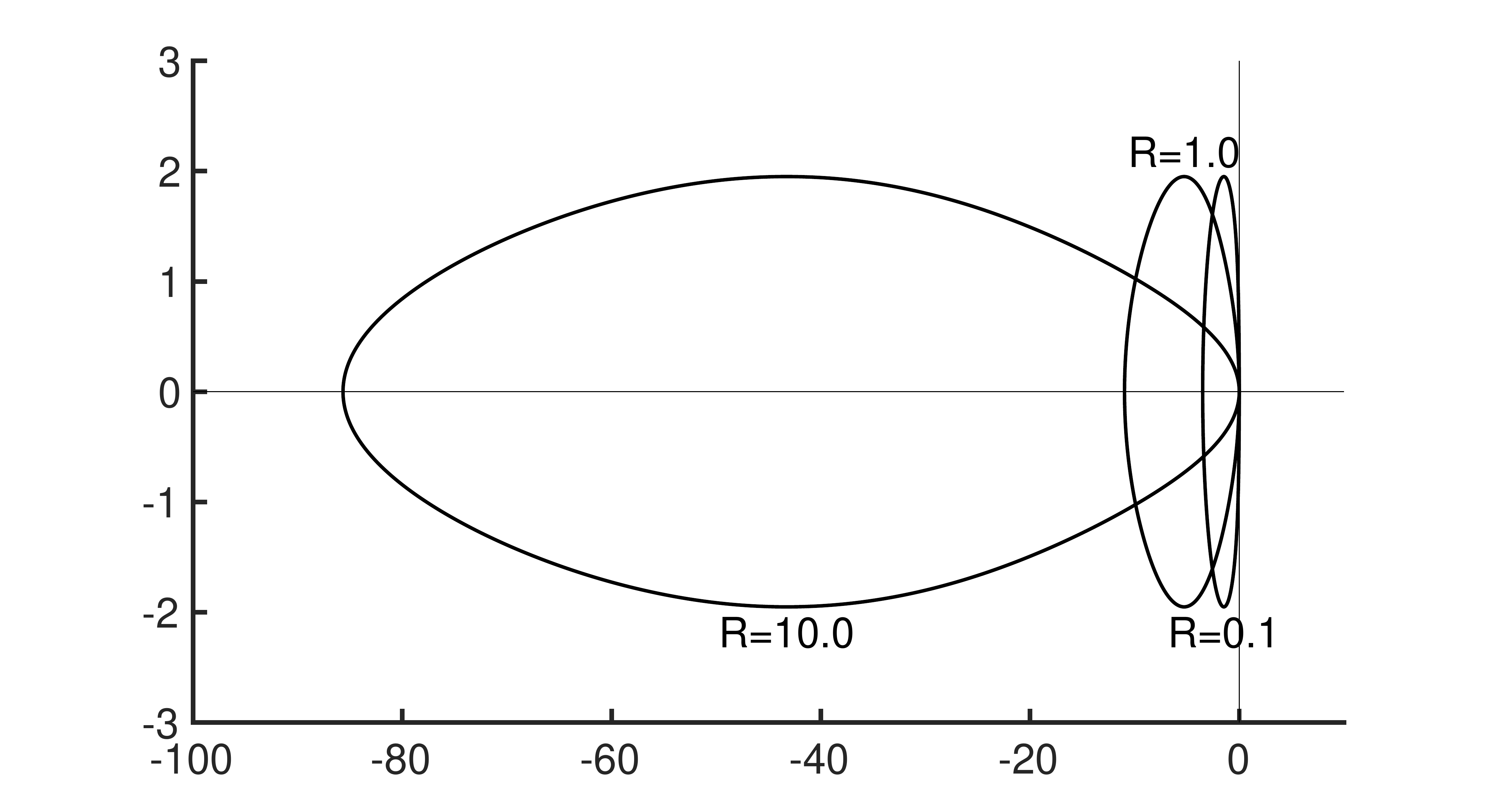}
    \caption{$\mathcal{D}_x^{3,1}$ and $\mathcal{D}_{xx}^{20}$.}
    \label{fg:num_ade_semi_c}
  \end{subfigure}
  \begin{subfigure}{.45\textwidth}
    \includegraphics[trim=1.2in 0.2in 1.2in 0.2in, clip, width=\textwidth]{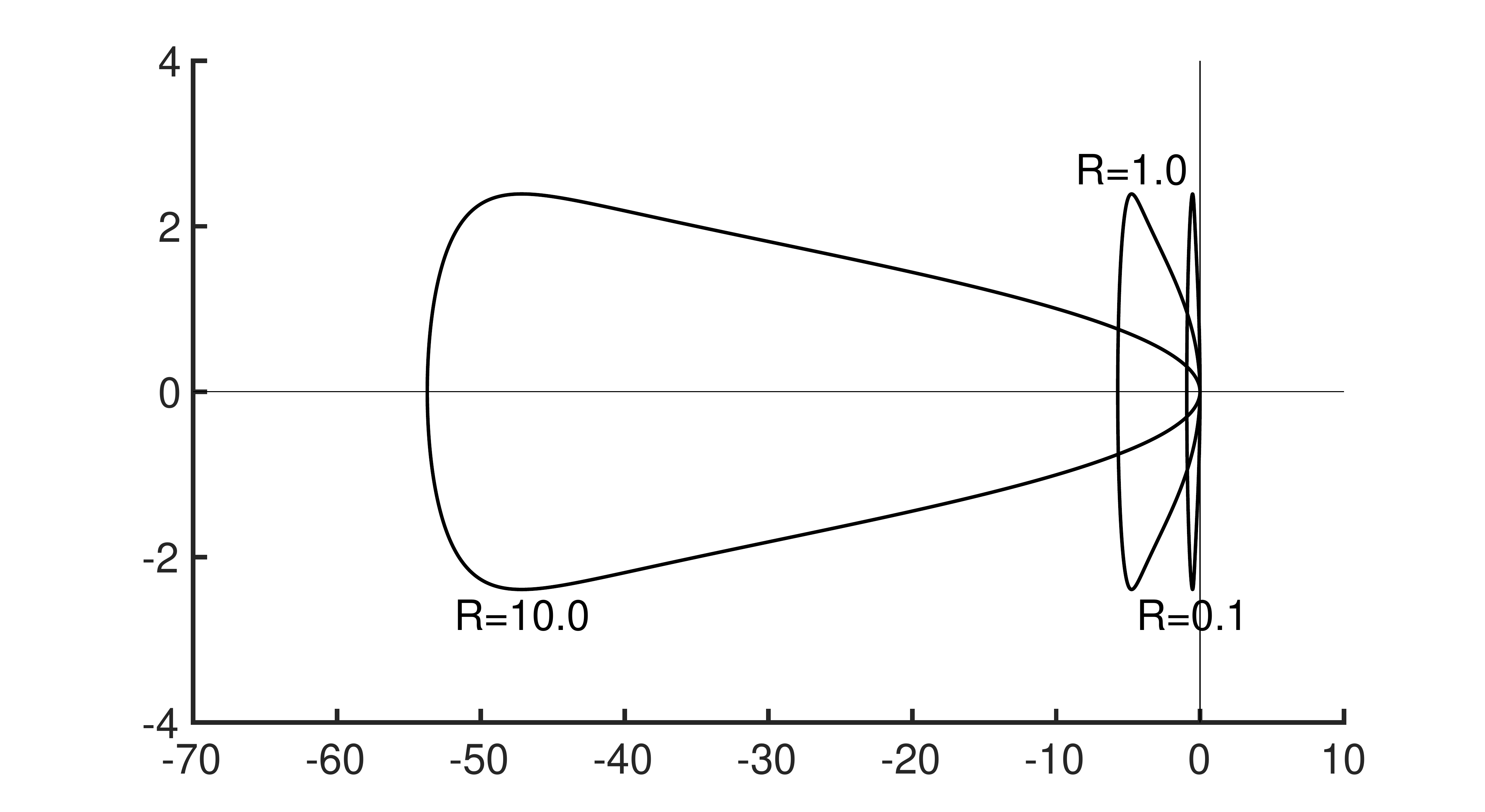}
    \caption{$\mathcal{D}_x^{21,20}$ and $\mathcal{D}_{xx}^2$.}
    \label{fg:num_ade_semi_d}
  \end{subfigure}
  \caption{Trajectories $\Lambda(R)$ with $R=0.1,\;1,\;10$ for the semi-discretized ODE system of the linear ADE by various ($\mathcal{D}_x$, $\mathcal{D}_{xx}$).}
  \label{fg:num_ade_semi}
\end{figure}

Next we verify the results given in Theorem~\ref{thm:full_adv_instab}.
To this end, given a combination of spatial discretization and a temporal method, we plot the {\it instability index}:
\begin{equation}\label{eq:num_ade_idx_mu}
  I_h = \log_{10}(\rho(p_s(\mu\bs{M}))-1)\;,
\end{equation}
against the number of cells $N$ for various Courant number $\mu$.
Here $\rho(\cdot)$ denotes the spectral radius of a matrix.
Note that $I_h$ is only defined for unstable methods, i.e., if $\rho(p_s(\mu\bs{M}))>1$.

In Figure~\ref{fg:num_adv_full_fe}, the FE time-integrator is paired with $\mathcal{D}_x^{2,0}$ and $\mathcal{D}_x^{12,11}$, and Figure~\ref{fg:num_adv_full_rk2} and Figure~\ref{fg:num_adv_full_lsrk3} demonstrate $\mathcal{D}_x^{3,1}$ and $\mathcal{D}_x^{12,11}$ pairing with RK2 and LSRK3, respectively.
\begin{figure}\centering
  \begin{subfigure}{.45\textwidth}
    \includegraphics[trim=0.6in 0.0in 1.1in 0.2in, clip, width=\textwidth]{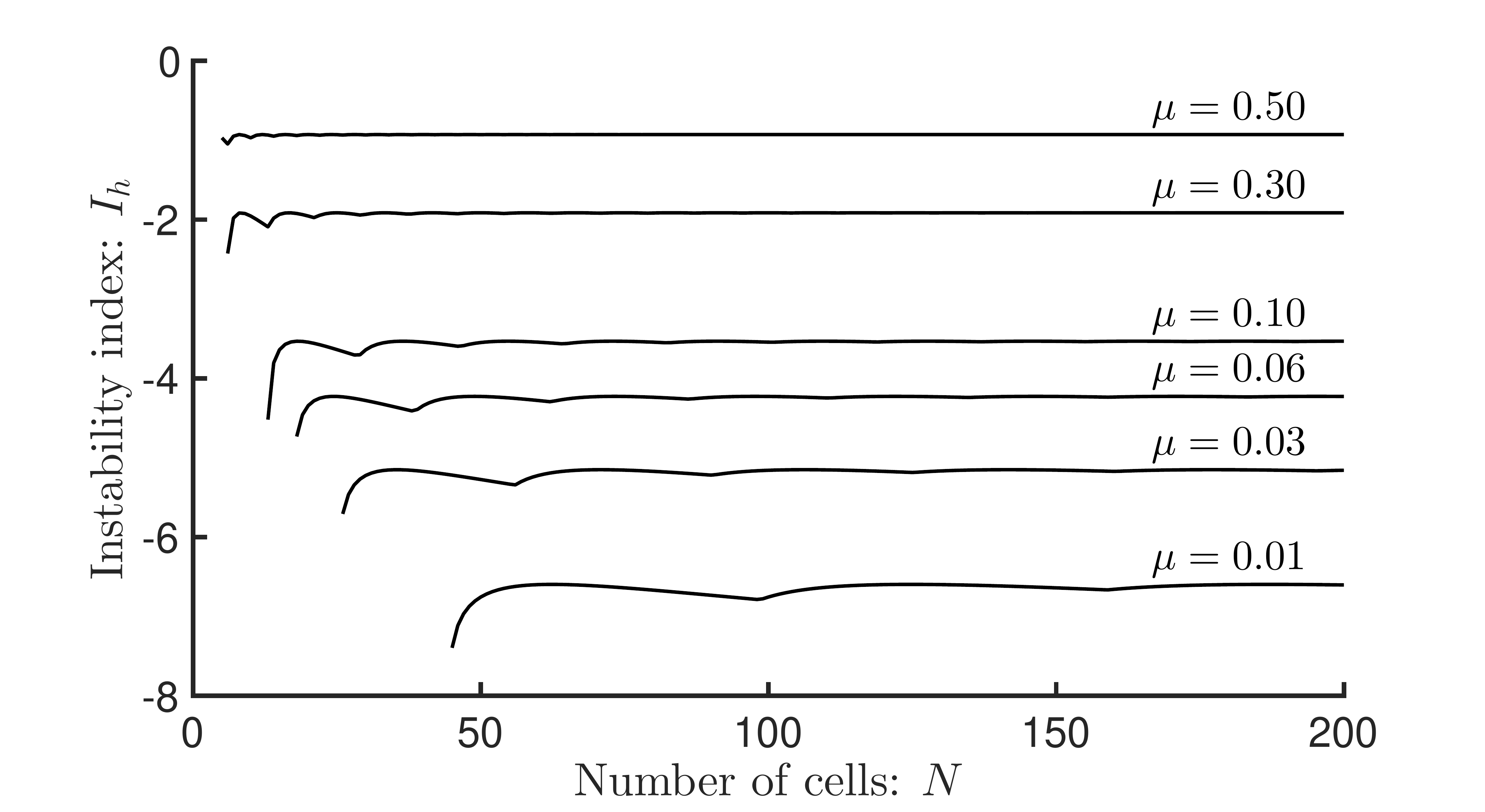}
    \caption{FE and $\mathcal{D}_x^{2,0}$.}
    \label{fg:num_adv_full_fe_a}
  \end{subfigure}
  \begin{subfigure}{.45\textwidth}
    \includegraphics[trim=0.6in 0.0in 1.1in 0.2in, clip, width=\textwidth]{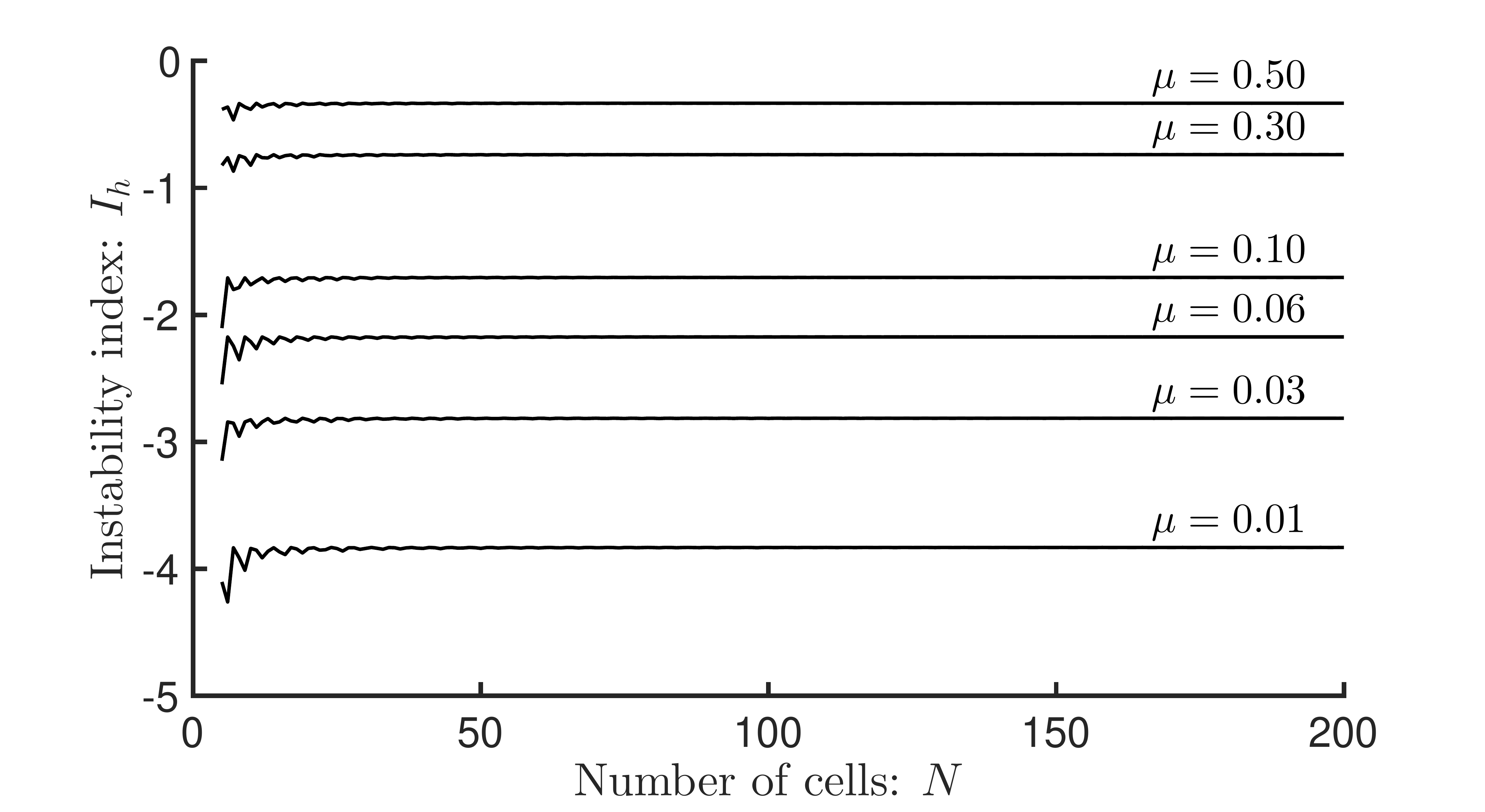}
    \caption{FE and $\mathcal{D}_x^{12,11}$.}
    \label{fg:num_adv_full_fe_b}
  \end{subfigure}
  \caption{The {\it instability index} $I_h$ vs. the number of cells $N$ at different values of $\mu=\delta t/h$ for the advection equation. FE is used in time.}
  \label{fg:num_adv_full_fe}
\end{figure}
\begin{figure}\centering
  \begin{subfigure}{.45\textwidth}
    \includegraphics[trim=0.6in 0.0in 1.1in 0.2in, clip, width=\textwidth]{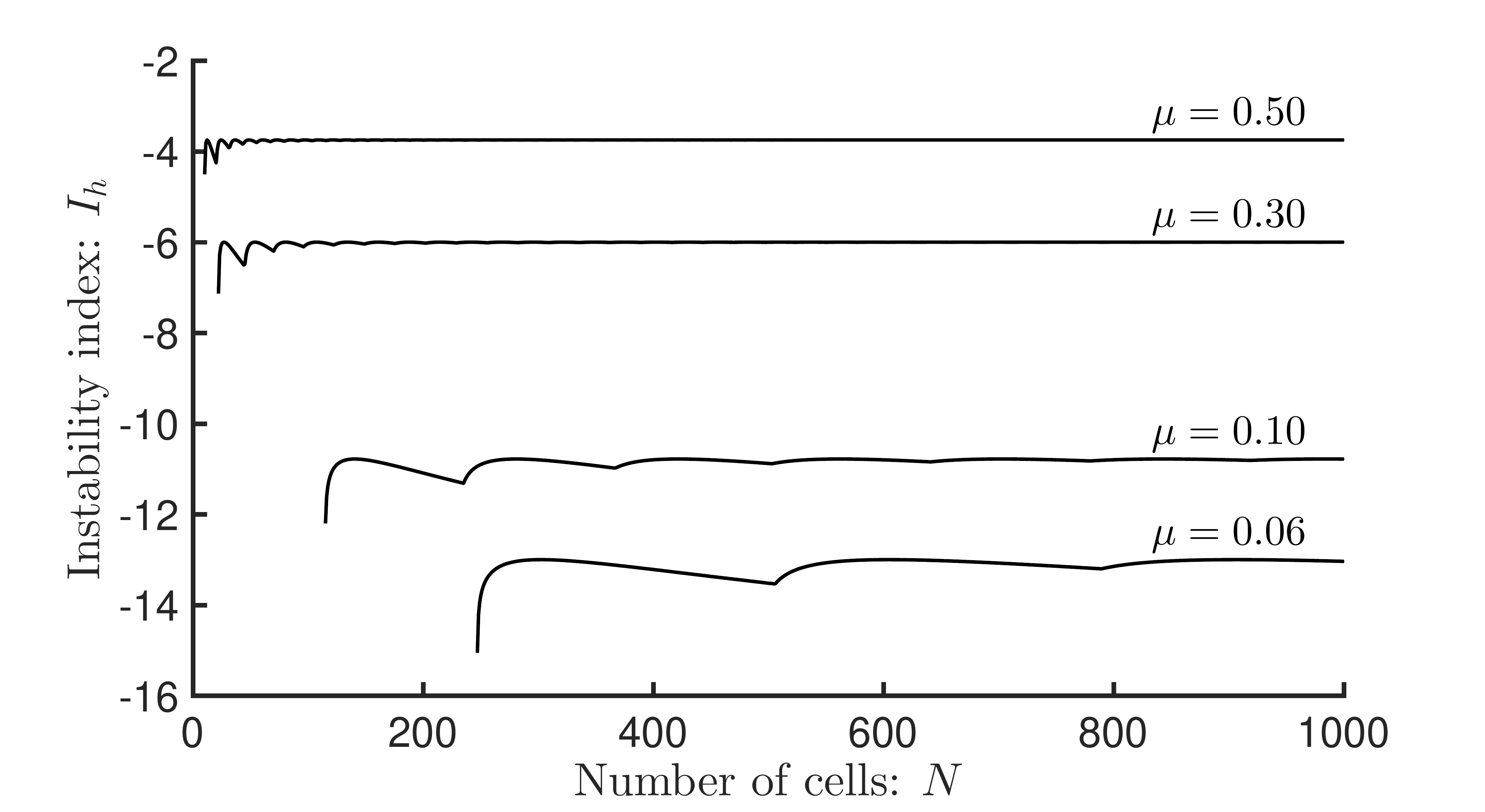}
    \caption{RK2 and $\mathcal{D}_x^{3,1}$.}
    \label{fg:num_adv_full_rk2_a}
  \end{subfigure}
  \begin{subfigure}{.45\textwidth}
    \includegraphics[trim=0.6in 0.0in 1.1in 0.2in, clip, width=\textwidth]{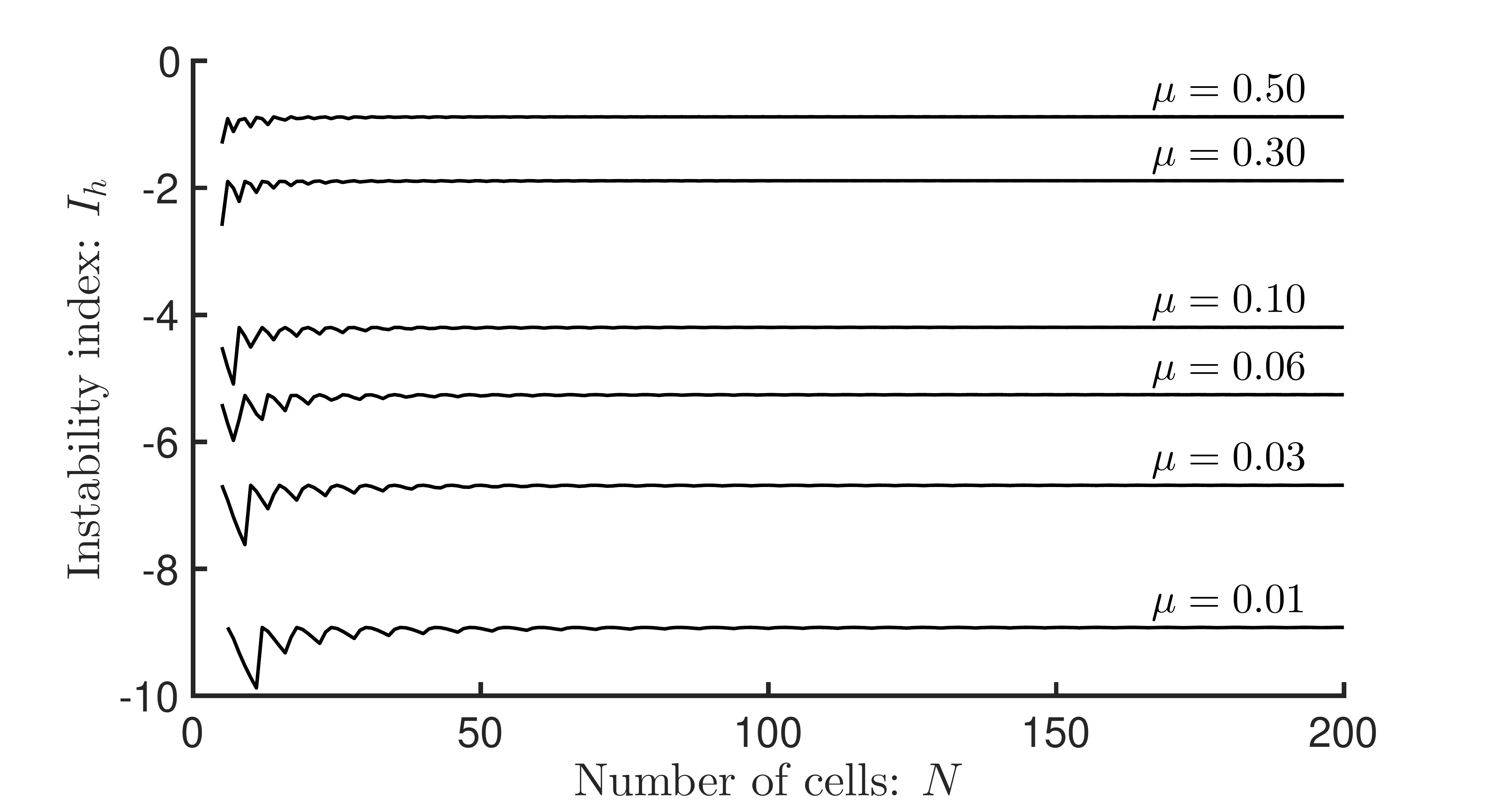}
    \caption{RK2 and $\mathcal{D}_x^{12,11}$.}
    \label{fg:num_adv_full_rk2_b}
  \end{subfigure}
  \caption{The {\it instability index} $I_h$ vs. the number of cells $N$ at different values of $\mu=\delta t/h$ for the advection equation. RK2 is used in time.}
  \label{fg:num_adv_full_rk2}
\end{figure}
\begin{figure}\centering
  \begin{subfigure}{.45\textwidth}
    \includegraphics[trim=0.6in 0.0in 1.1in 0.2in, clip, width=\textwidth]{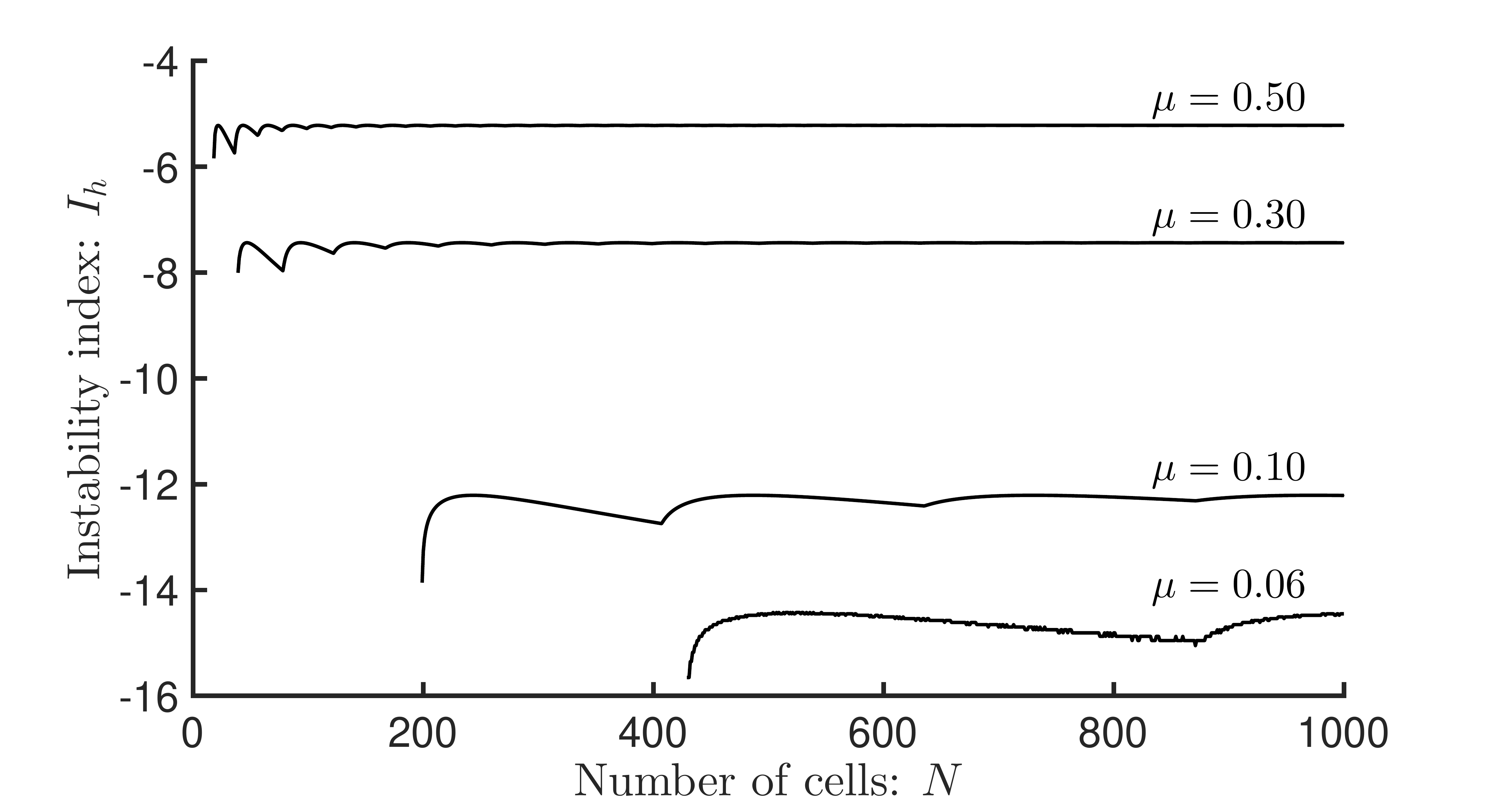}
    \caption{LSRK3 and $\mathcal{D}_x^{3,1}$.}
    \label{fg:num_adv_full_lsrk3_a}
  \end{subfigure}
  \begin{subfigure}{.45\textwidth}
    \includegraphics[trim=0.6in 0.0in 1.1in 0.2in, clip, width=\textwidth]{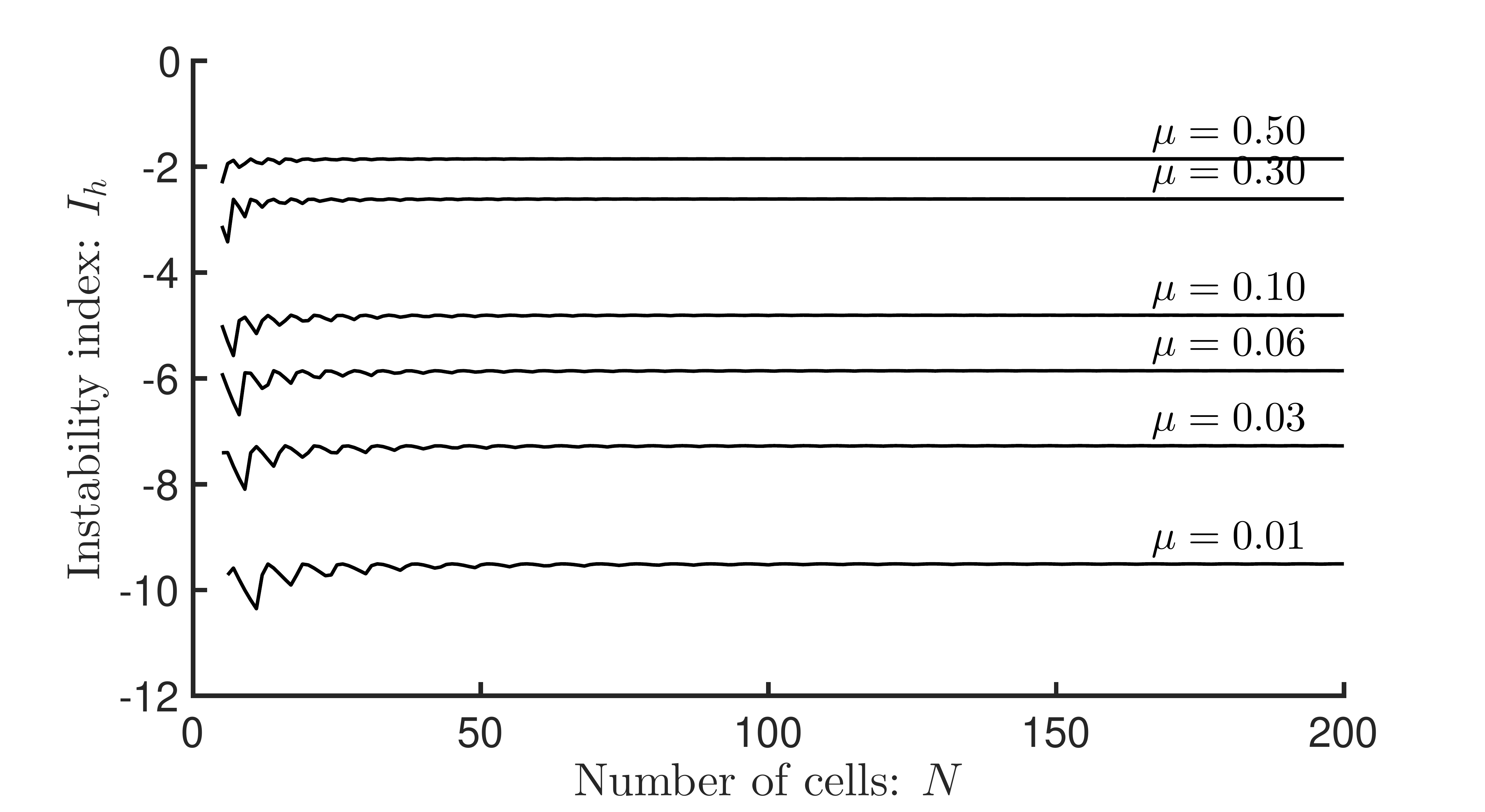}
    \caption{LSRK3 and $\mathcal{D}_x^{12,11}$.}
    \label{fg:num_adv_full_lsrk3_b}
  \end{subfigure}
  \caption{The {\it instability index} $I_h$ vs. the number of cells $N$ at different values of $\mu=\delta t/h$ for the advection equation. LSRK3 is used in time.}
  \label{fg:num_adv_full_lsrk3}
\end{figure}
These plots on the one hand verify the result in Theorem~\ref{thm:full_adv_instab} and on the other hand indicate that such instability may be difficult to observe in practice.
In particular, complementing the result in the theorem, one makes the following observations from these curves:
\begin{itemize}
  \item Decreasing the Courant number reduces the stability violation.
  \item Higher-order spatial discretization tends to introduce larger stability violation.
  \item The instability caused by FE is generally much larger than that of RK2 and LSRK3.
  \item When RK2 and LSRK3 are combined with the lower-order methods, the instability index is close to the machine precision error for small Courant numbers.
\end{itemize}

These instability can also be observed directly by solving the periodic problem for the advection equation $w_t+w_x=0$ with the initial condition given by a Gaussian pulse $w(x,0) = \exp(-100(x-1/2)^2)$.
In Figures~\ref{fg:num_adv_gaussian_fe}--\ref{fg:num_adv_gaussian_lsrk3}, the numerical solutions obtained by the same set of schemes as before are plotted to demonstrate their growth in magnitudes.
For all schemes, we pick a representative $\mu_c$ that gives an $I_h$ between $-5$ and $-6$ -- istability can still be seen with smaller $\mu_c$ but it usually takes an extremely long simulaton to show up; and for all computations, a uniform grid with $100$ uniform cells is used.
\begin{figure}\centering
  \begin{subfigure}{.45\textwidth}
    \includegraphics[trim=0.6in 0.0in 1.1in 0.2in, clip, width=\textwidth]{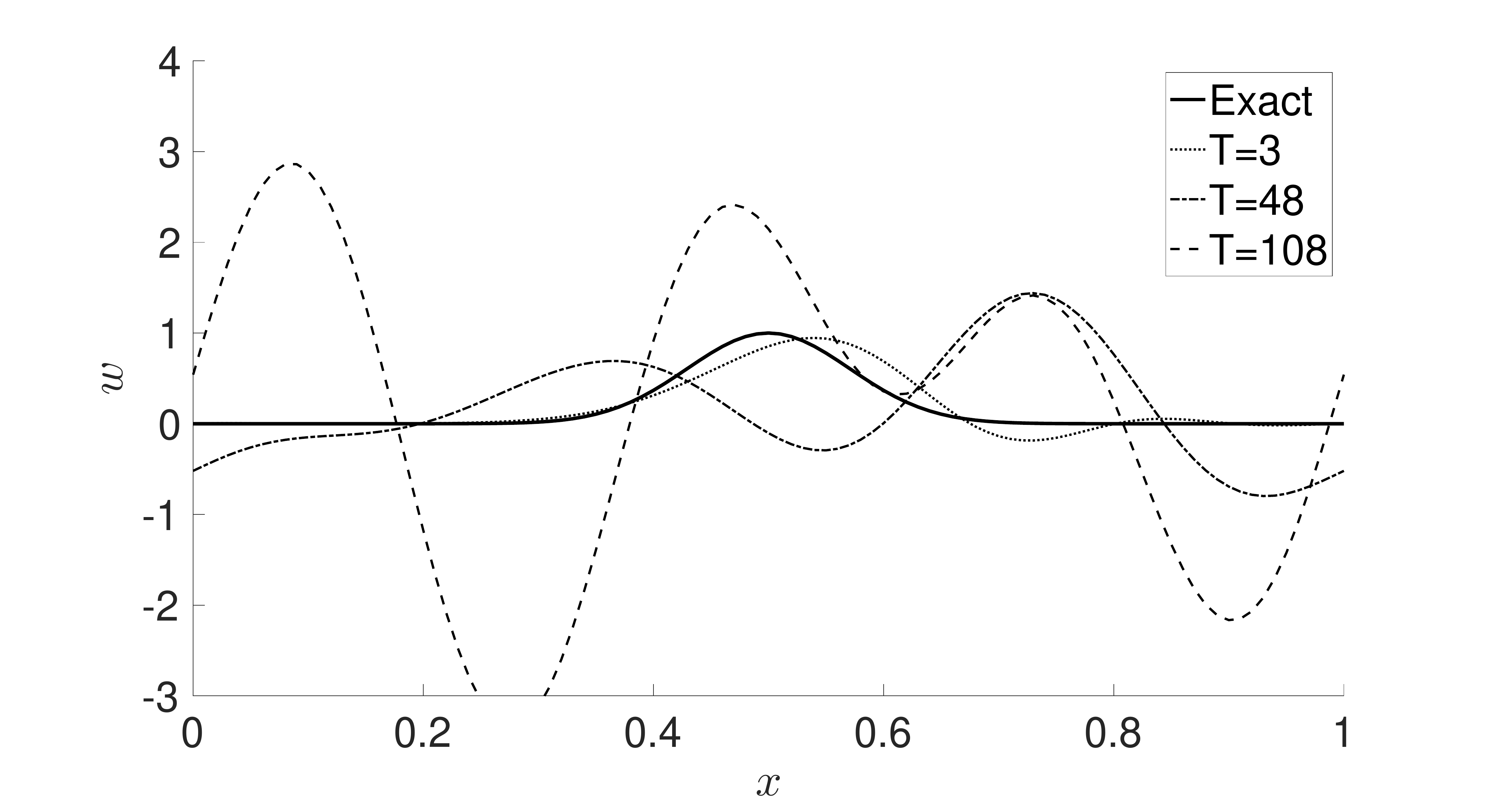}
    \caption{FE and $\mathcal{D}_x^{2,0}$, $\mu=0.03$.}
    \label{fg:num_adv_gaussian_fe_a}
  \end{subfigure}
  \begin{subfigure}{.45\textwidth}
    \includegraphics[trim=0.6in 0.0in 1.1in 0.2in, clip, width=\textwidth]{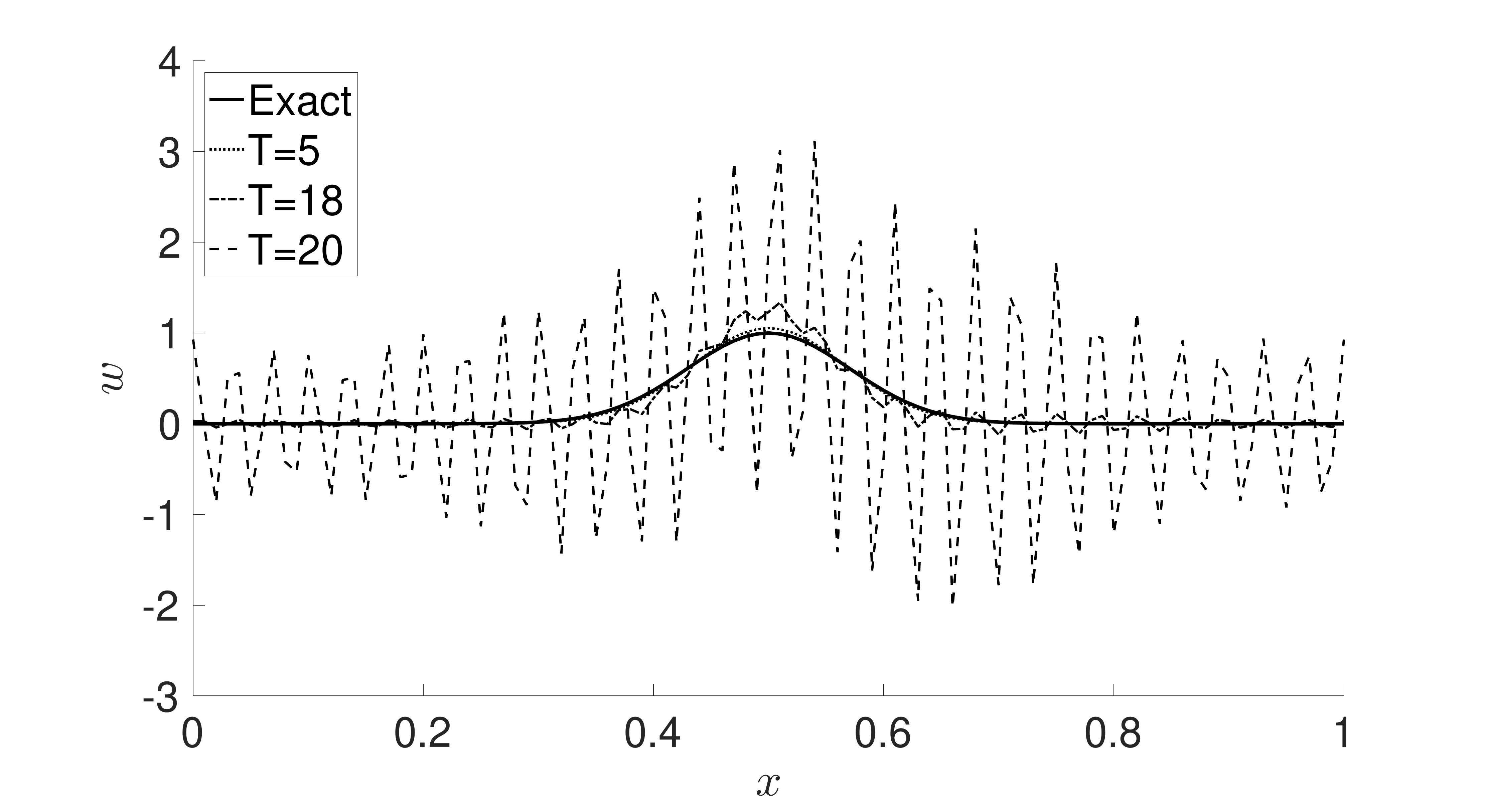}
    \caption{FE and $\mathcal{D}_x^{12,11}$, $\mu=0.01$.}
    \label{fg:num_adv_gaussian_fe_b}
  \end{subfigure}
  \caption{Advection of a Gaussian pulse by FE in time and two $\mathcal{D}_x$'s.}
  \label{fg:num_adv_gaussian_fe}
\end{figure}
\begin{figure}\centering
  \begin{subfigure}{.45\textwidth}
    \includegraphics[trim=0.6in 0.0in 1.1in 0.2in, clip, width=\textwidth]{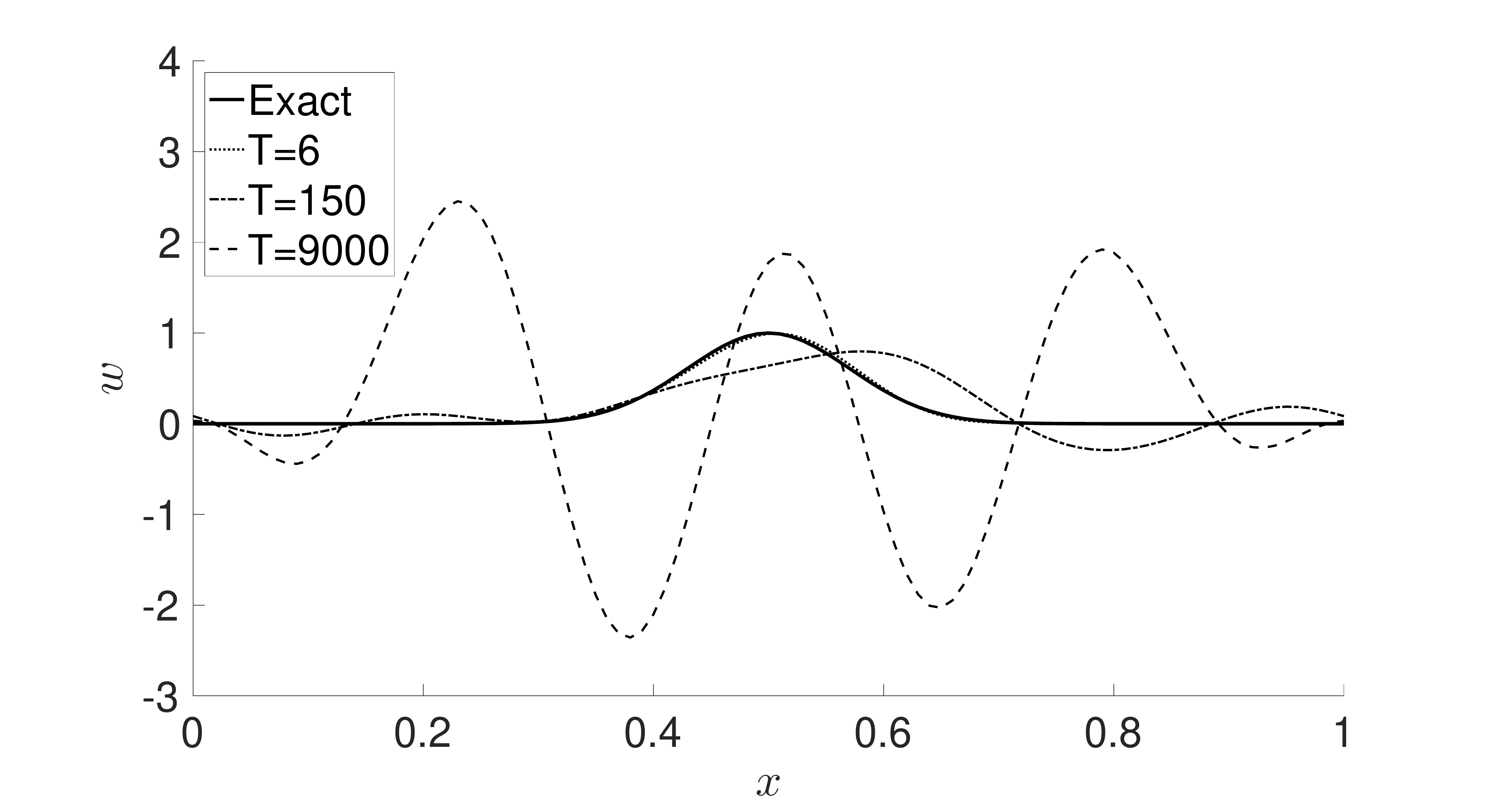}
    \caption{RK2 and $\mathcal{D}_x^{3,1}$, $\mu=0.3$.}
    \label{fg:num_adv_gaussian_rk2_a}
  \end{subfigure}
  \begin{subfigure}{.45\textwidth}
    \includegraphics[trim=0.6in 0.0in 1.1in 0.2in, clip, width=\textwidth]{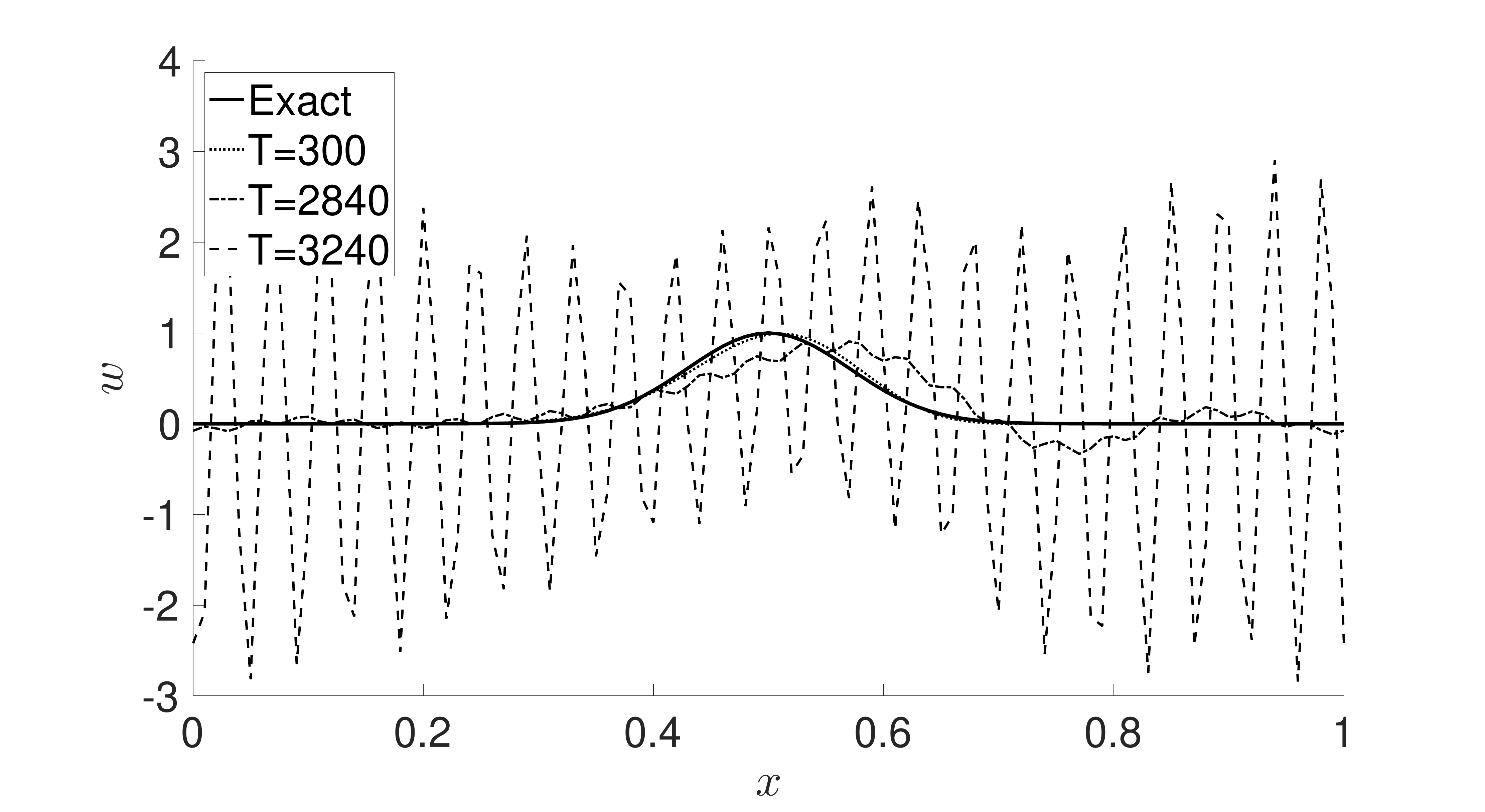}
    \caption{RK2 and $\mathcal{D}_x^{12,11}$, $\mu=0.06$.}
    \label{fg:num_adv_gaussian_rk2_b}
  \end{subfigure}
  \caption{Advection of a Gaussian pulse by RK2 in time and two $\mathcal{D}_x$'s.}
  \label{fg:num_adv_gaussian_rk2}
\end{figure}
\begin{figure}\centering
  \begin{subfigure}{.45\textwidth}
    \includegraphics[trim=0.6in 0.0in 1.1in 0.2in, clip, width=\textwidth]{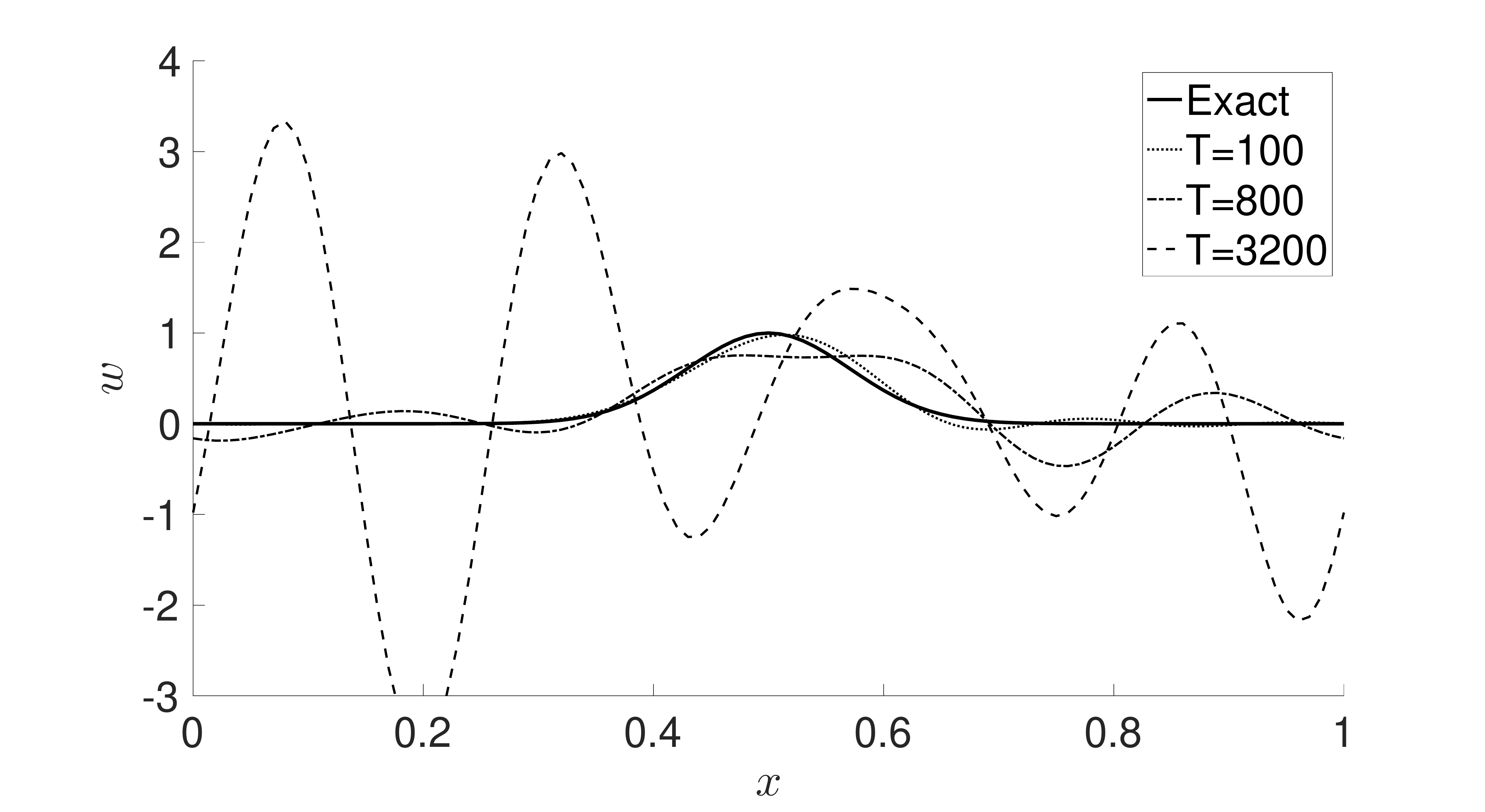}
    \caption{LSRK3 and $\mathcal{D}_x^{3,1}$, $\mu=0.5$.}
    \label{fg:num_adv_gaussian_lsrk3_a}
  \end{subfigure}
  \begin{subfigure}{.45\textwidth}
    \includegraphics[trim=0.6in 0.0in 1.1in 0.2in, clip, width=\textwidth]{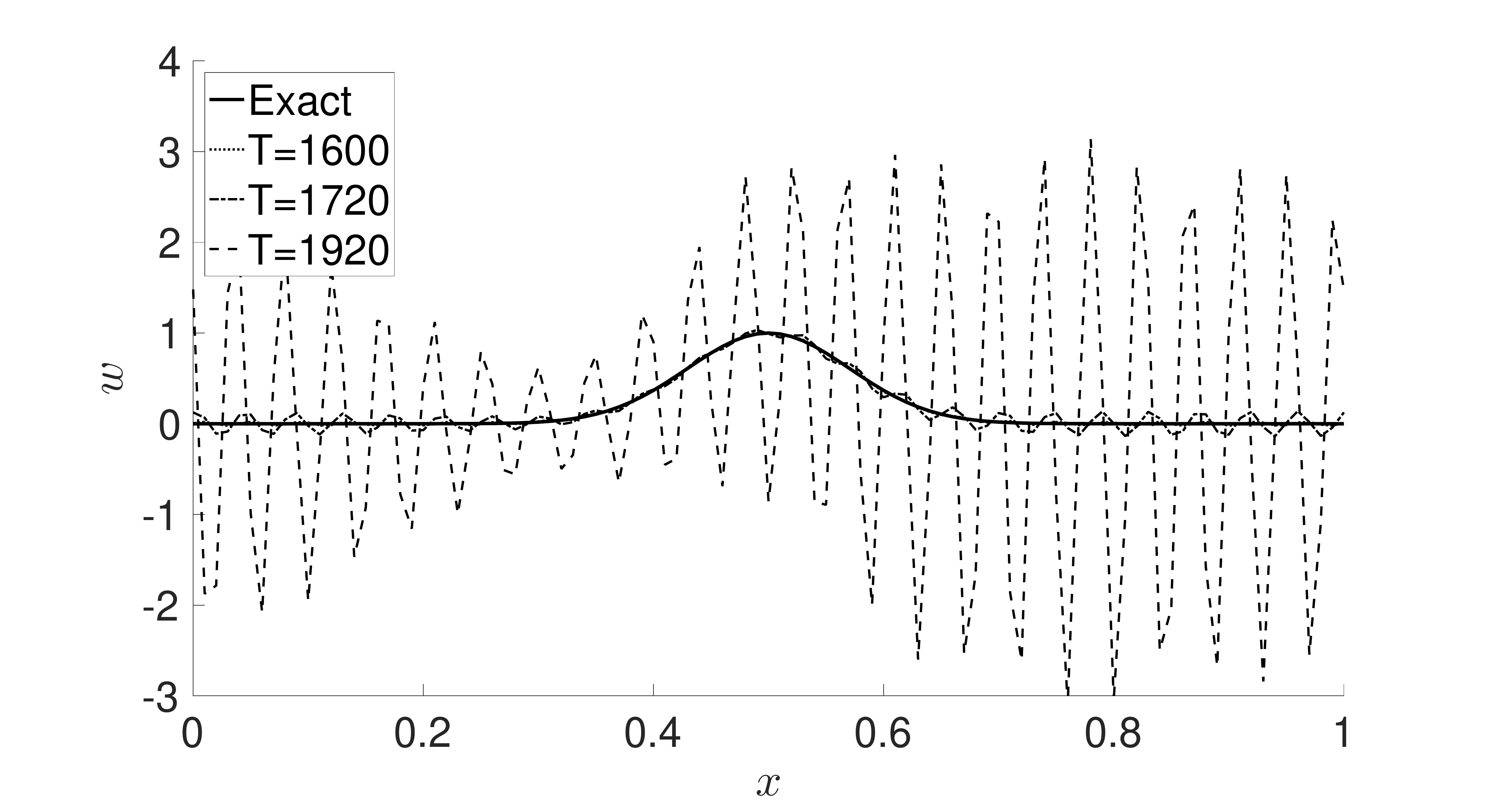}
    \caption{LSRK3 and $\mathcal{D}_x^{12,11}$, $\mu=0.1$.}
    \label{fg:num_adv_gaussian_lsrk3_b}
  \end{subfigure}
  \caption{Advection of a Gaussian pulse by LSRK3 in time and two $\mathcal{D}_x$'s.
    In the case of $\mathcal{D}_x^{12,11}$ (right), the numerical solution at $T=1600$ is on top of the exact one.}
  \label{fg:num_adv_gaussian_lsrk3}
\end{figure}
In all these plots, the numerical solutions at three different times (denoted by $T$ in the legends) are plotted against the exact solution, which happens to be the same as the initial condition for all chosen $T$.

Lastly, to verify Theorem~\ref{thm:full_hord}, we rewrite the second equation of (\ref{eq:full_cfl}) as:
\begin{displaymath}
  \mu_c\eqdef\frac{\nu\delta t}{h^2} < \left(\beta_0+\frac{\alpha_0h}{\nu}\right)^{-1}\;.
\end{displaymath}
Hence we expect stability (i.e., $I_h$ undefined) provided $\mu_c\lessapprox\beta_0^{-1}$ and $h$ is sufficiently small.
To this end, we plot $I_h$ against $N=1/h$ at different values of $\mu_c$ in Figure~\ref{fg:num_ade_full} for a variety of discretizations, which include the spatial discretization being $\mathcal{D}_x^{3,1}$ and $\mathcal{D}_{xx}^2$ in the left column or $\mathcal{D}_x^{11,10}$ and $\mathcal{D}_{xx}^{10}$ in the right column, and the time-integrator being FE (top row), RK2 (middle row), or RK4 (bottom row).
In all these plots, we set $\nu=0.1$.
\begin{figure}\centering
  \begin{subfigure}{\textwidth}
    \includegraphics[trim=0.6in 0.0in 1.1in 0.2in, clip, width=.45\textwidth]{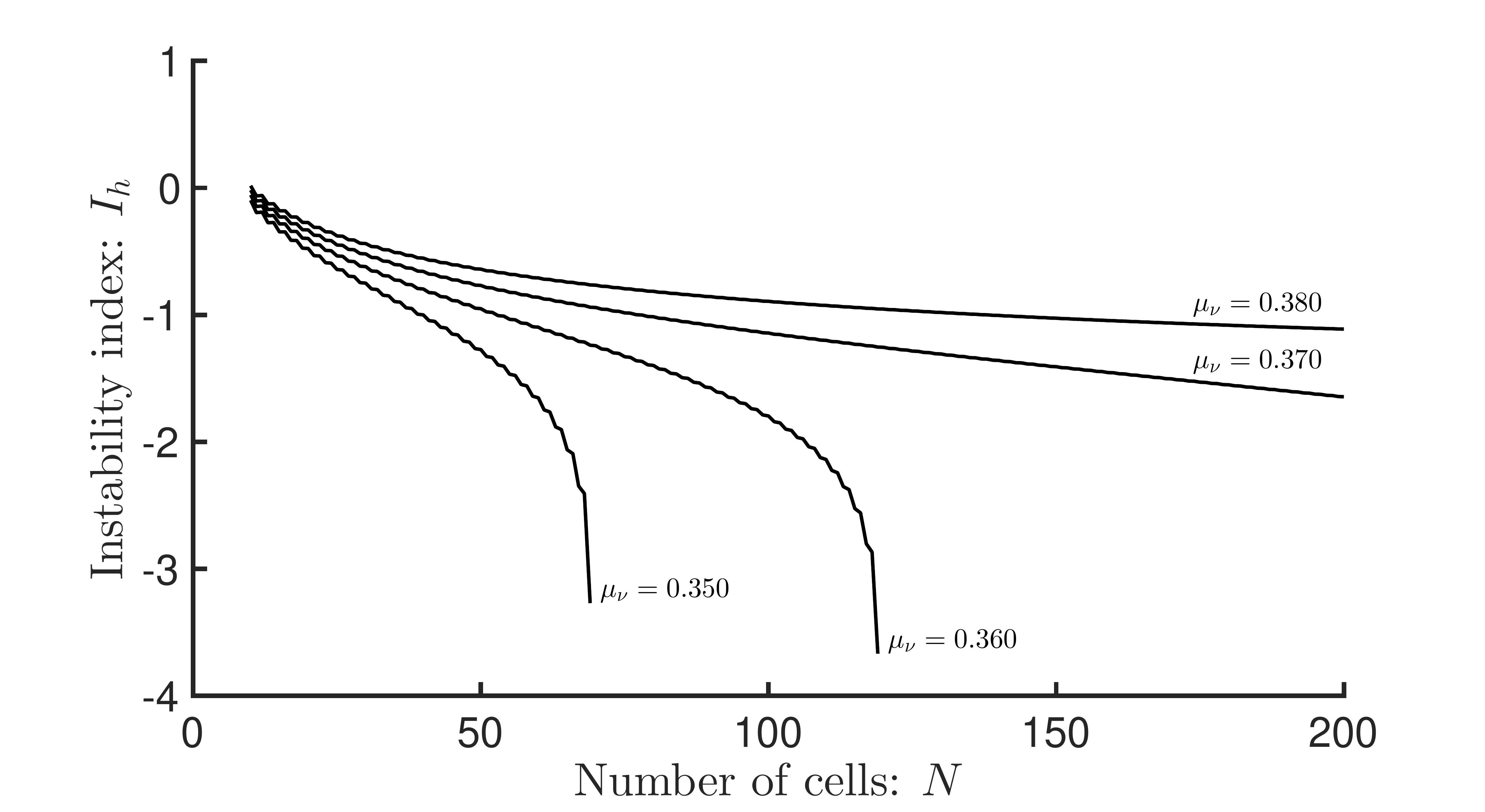}~~
    \includegraphics[trim=0.6in 0.0in 1.1in 0.2in, clip, width=.45\textwidth]{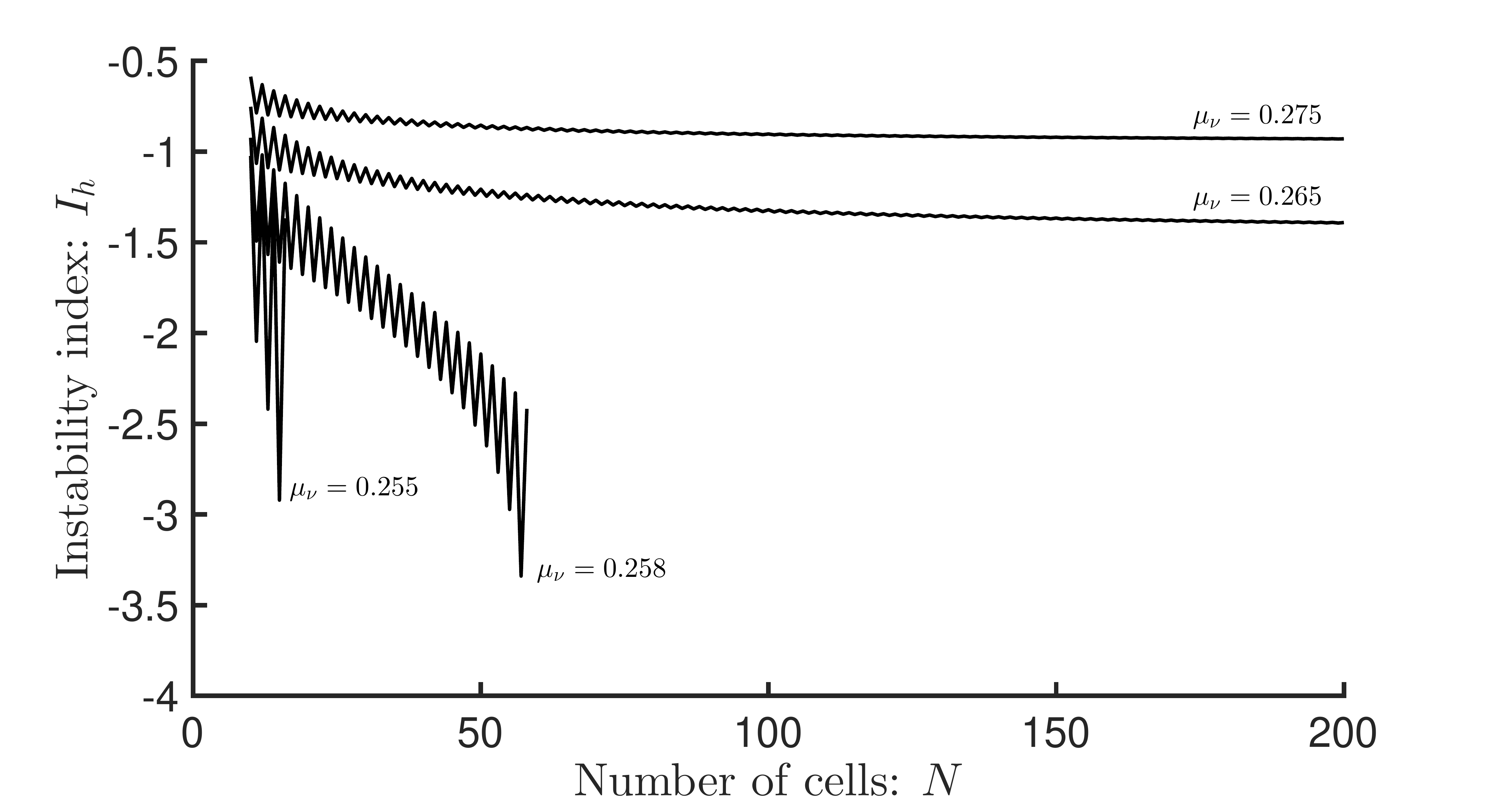}~~
    \caption{FE combined with: (left) $\mathcal{D}_x^{3,1}$ and $\mathcal{D}_{xx}^2$, (right) $\mathcal{D}_x^{11,10}$ and $\mathcal{D}_{xx}^{10}$.}
    \label{fg:num_ade_full_fe}
  \end{subfigure}
  \begin{subfigure}{\textwidth}
    \includegraphics[trim=0.6in 0.0in 1.1in 0.2in, clip, width=.45\textwidth]{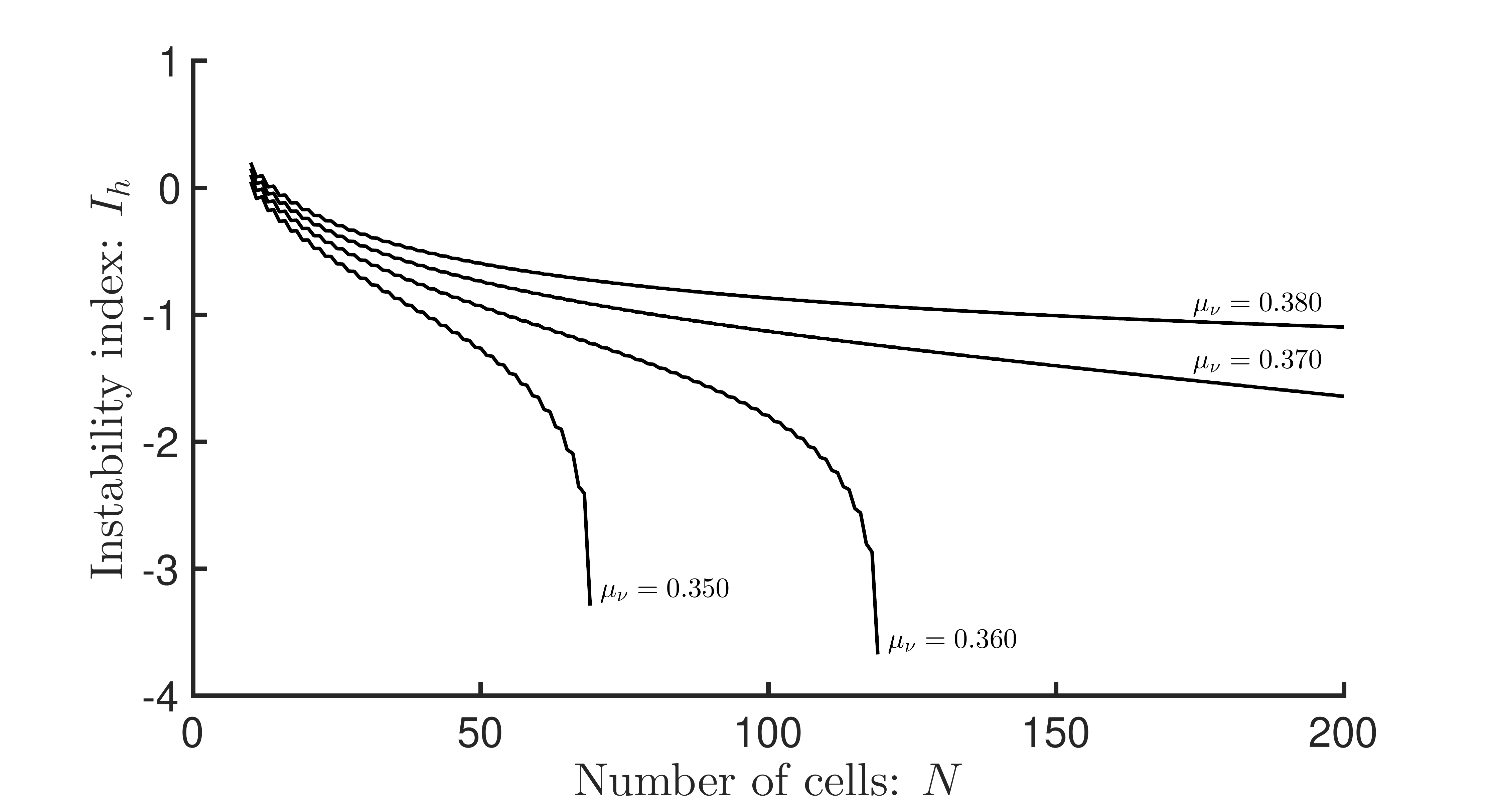}~~
    \includegraphics[trim=0.6in 0.0in 1.1in 0.2in, clip, width=.45\textwidth]{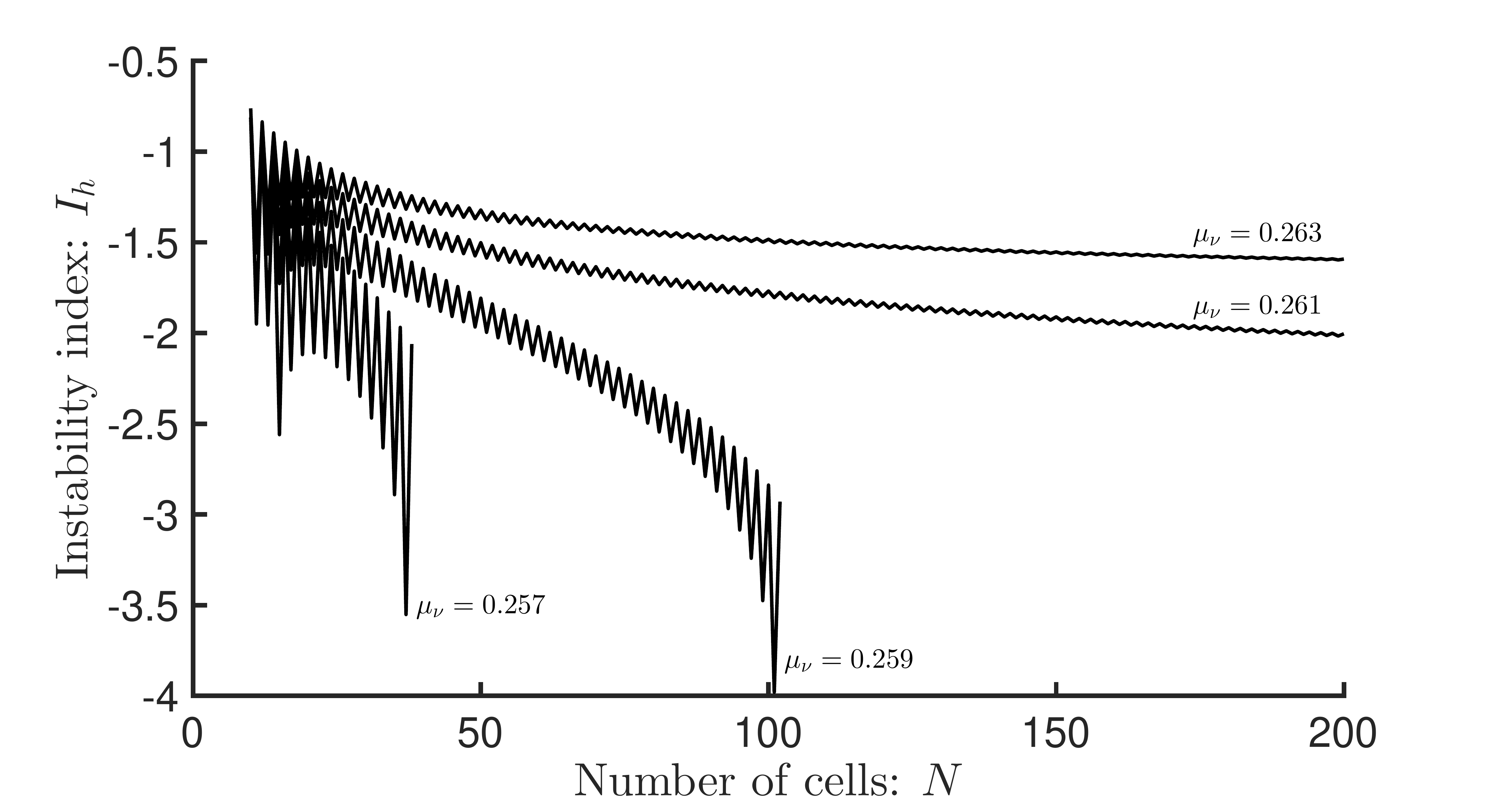}~~
    \caption{RK2 combined with: (left) $\mathcal{D}_x^{3,1}$ and $\mathcal{D}_{xx}^2$, (right) $\mathcal{D}_x^{11,10}$ and $\mathcal{D}_{xx}^{10}$.}
    \label{fg:num_ade_full_rk2}
  \end{subfigure}
  \begin{subfigure}{\textwidth}
    \includegraphics[trim=0.6in 0.0in 1.1in 0.2in, clip, width=.45\textwidth]{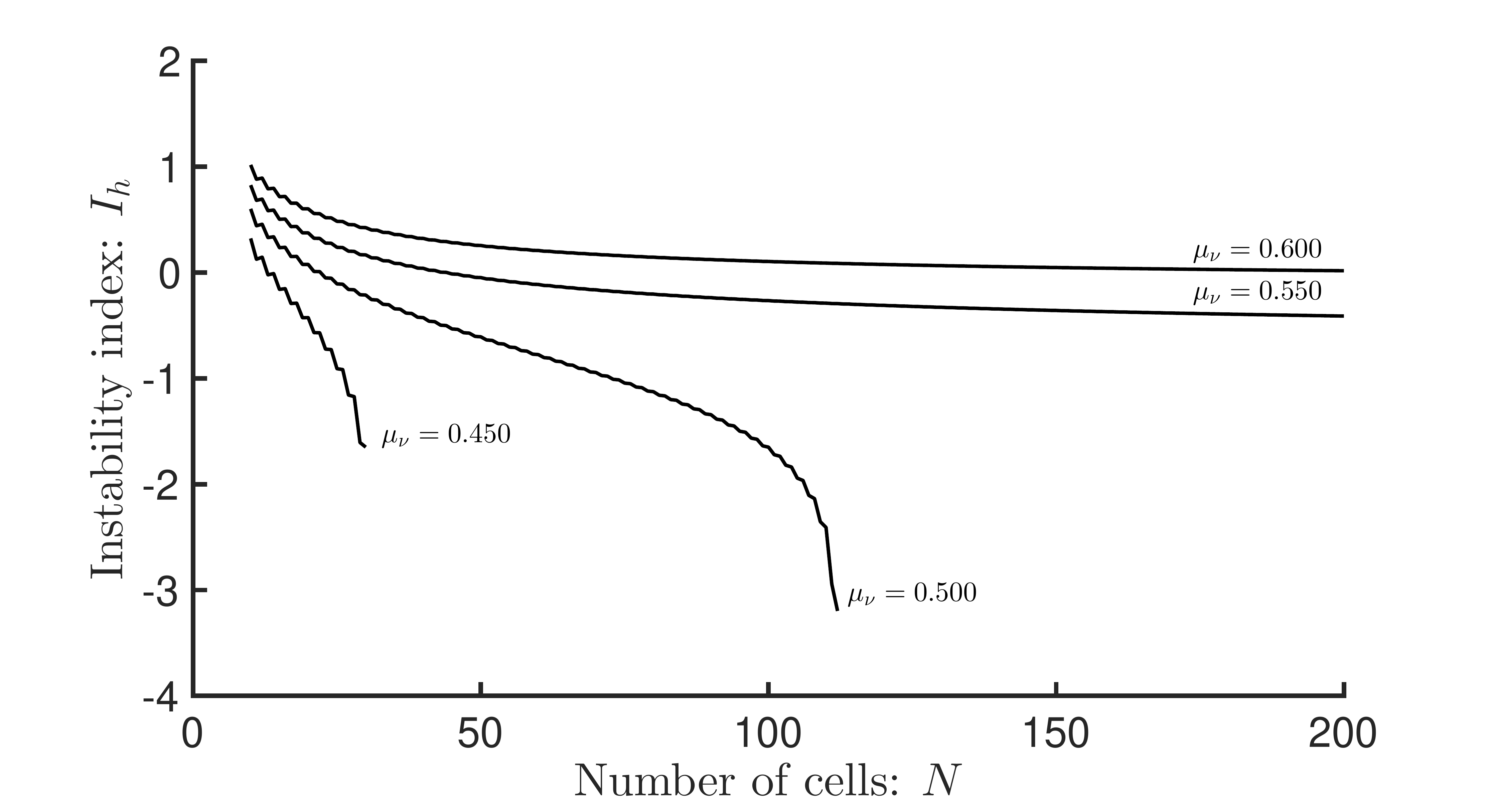}~~
    \includegraphics[trim=0.6in 0.0in 1.1in 0.2in, clip, width=.45\textwidth]{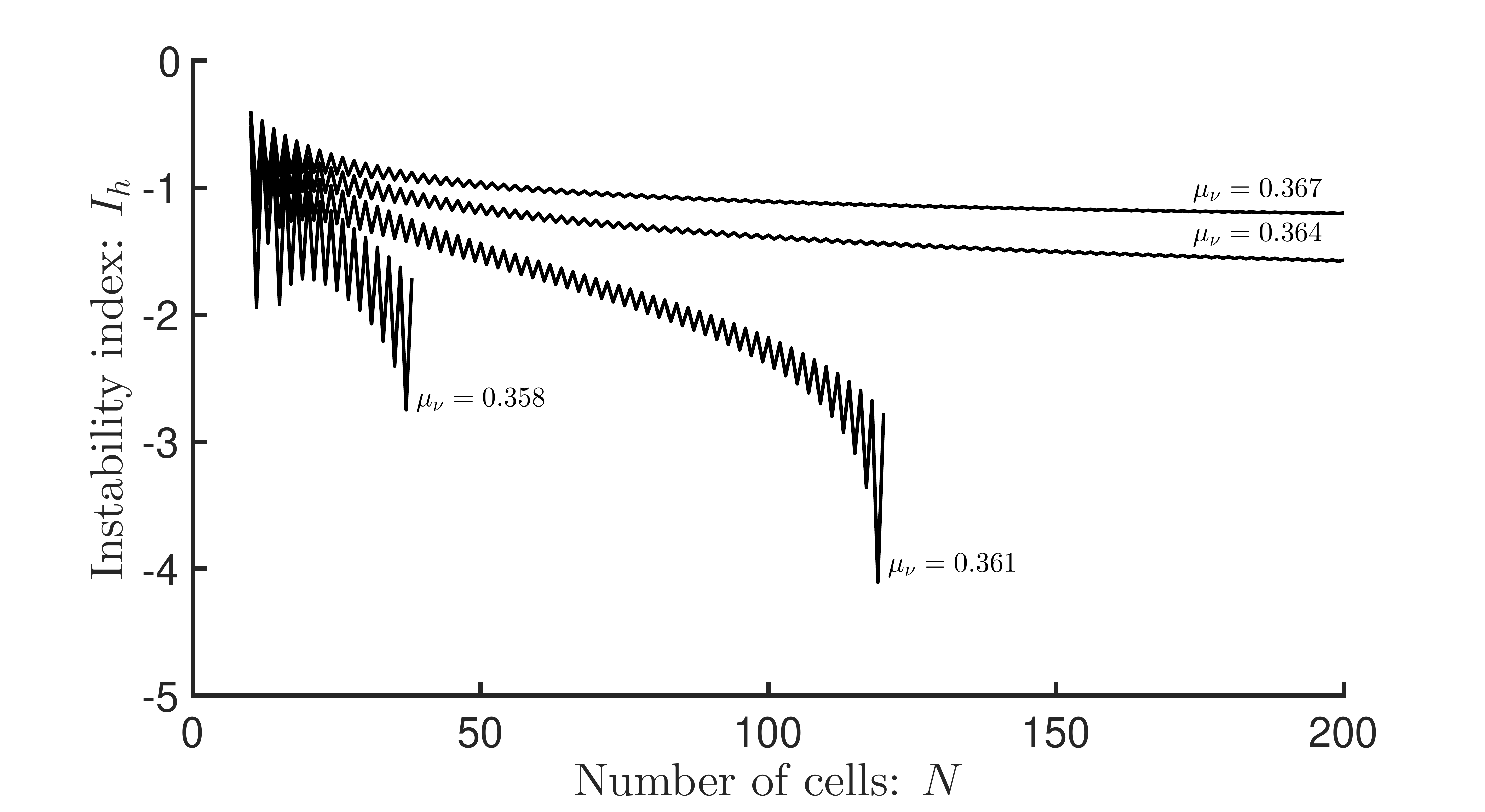}~~
    \caption{RK4 combined with: (left) $\mathcal{D}_x^{3,1}$ and $\mathcal{D}_{xx}^2$, (right) $\mathcal{D}_x^{11,10}$ and $\mathcal{D}_{xx}^{10}$.}
    \label{fg:num_ade_full_rk4}
  \end{subfigure}
  \caption{The {\it instability index} $I_h$ vs. the number of cells $N$ at different values of $\mu_{\nu}=\nu\delta t/h^2$ for linear ADEs.
    A broken curve indicates conditional stability.}
  \label{fg:num_ade_full}
\end{figure}
From the figures, one clearly observes that when $\mu_c$ is below a certain threshold ($\approx1/\beta_0$), the curve breaks at some finite value of $N_c$, which indicates stability of the fully discretized method for all $h<h_c=1/N_c$.

\subsection{Semi-dissipative wave systems}
\label{sec:num_wave}
The combination of discrete operators consists of three FDOs $\mathcal{D}_x^-$, $\mathcal{D}_x^+$, and $\mathcal{D}_{xx}$; and it will be denoted by a triple like ($\mathcal{D}_x^{2,1}$, $\mathcal{D}_x^{1,2}$, $\mathcal{D}_{xx}^2$).

First, we consider the semi-discretized method as before and plot the trajectory $\Lambda(R)$ given by (\ref{eq:wave_traj}). 
Symmetric $\mathcal{D}_x^-$ and $\mathcal{D}_x^+$ are supposed for plots in Figure~\ref{fg:num_wave_semi_sym}, where 
two combinations ($\mathcal{D}_x^{3,1}$, $\mathcal{D}_x^{1,3}$, $\mathcal{D}_{xx}^2$) and ($\mathcal{D}_x^{21,20}$, $\mathcal{D}_x^{20,21}$, $\mathcal{D}_{xx}^{20}$) are considered.
Comparing the trajectories with two different values $R=0.1$ and $R=2.0$, one observes that the ''height'' of the trajectory shrinks as $R$ increases, which is unlike the case of ADEs where the ``height'' of the trajectory seems to be less depend on the value of $R$.
\begin{figure}[ht]\centering
  \begin{subfigure}{\textwidth}\centering
    \includegraphics[trim=1.2in 0.2in 1.2in 0.2in, clip, width=.45\textwidth]{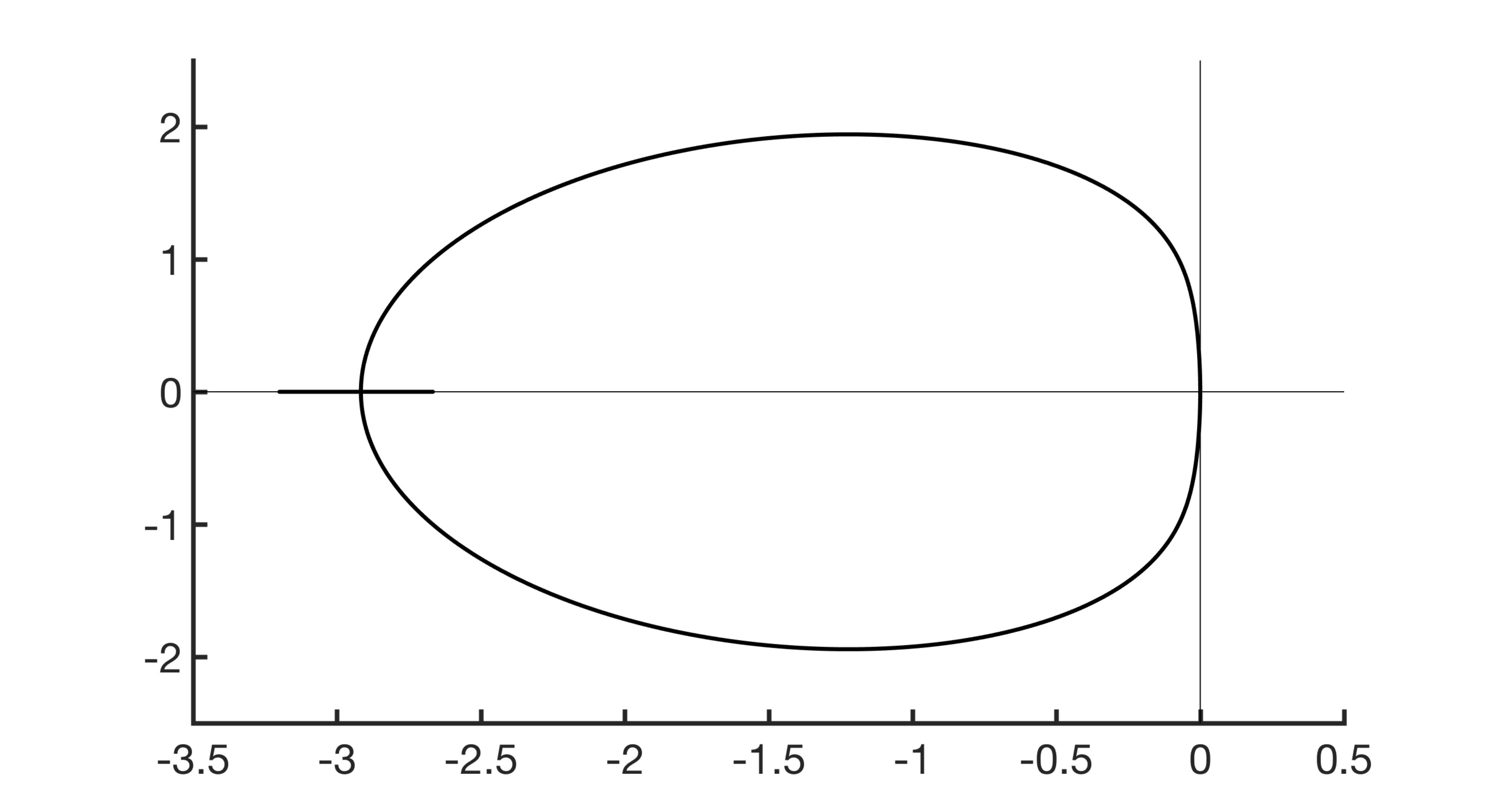}~~
    \includegraphics[trim=1.2in 0.2in 1.2in 0.2in, clip, width=.45\textwidth]{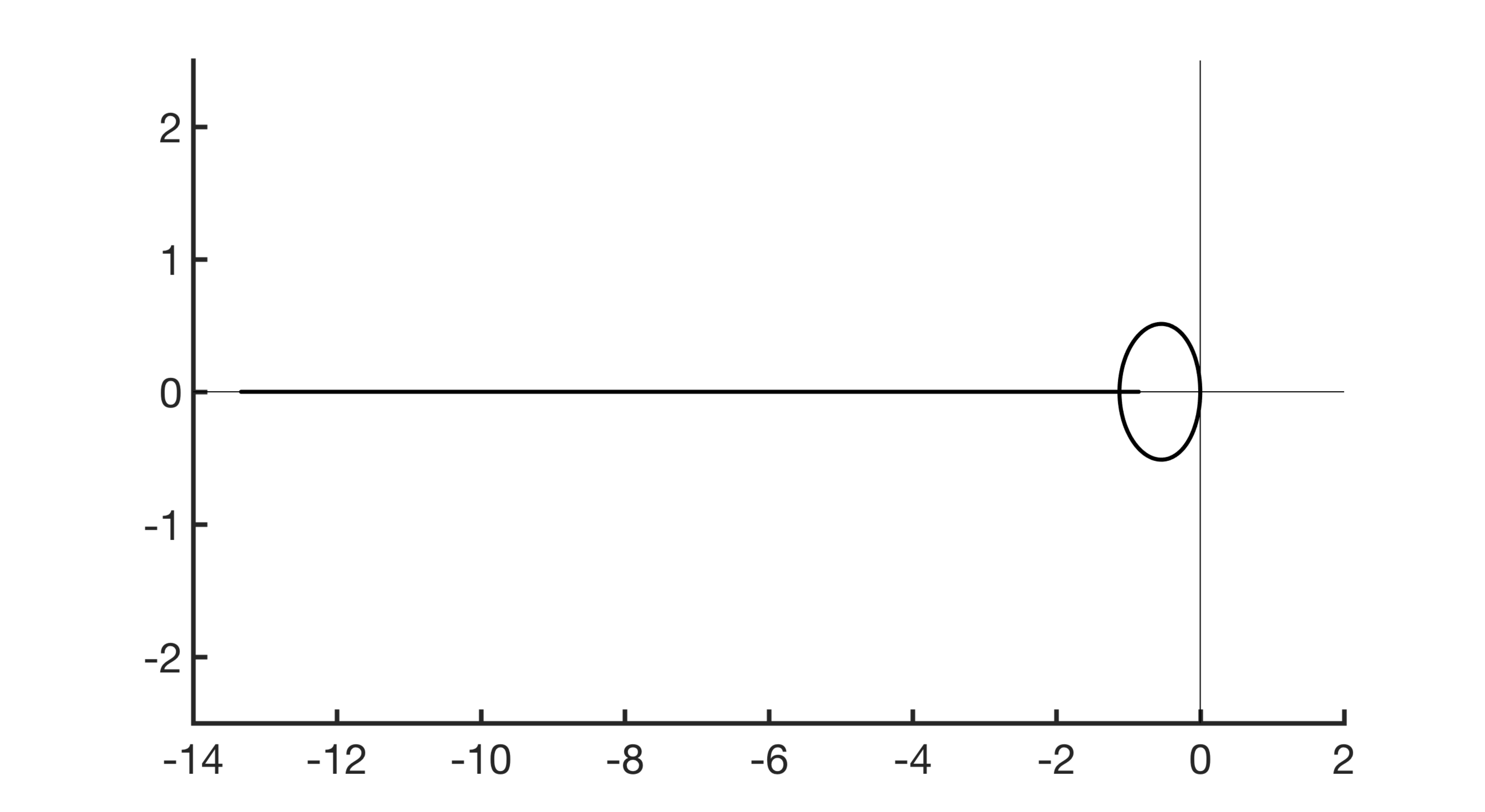}
    \caption{The trajectories $\Lambda(R)$ of ($\mathcal{D}_x^{3,1}$, $\mathcal{D}_x^{1,3}$, $\mathcal{D}_{xx}^2$) with $R=0.1$ (left) and $R=2$ (right).}
    \label{fg:num_wave_semi_sym_a}
  \end{subfigure}
  \begin{subfigure}{\textwidth}\centering
    \includegraphics[trim=1.2in 0.2in 1.2in 0.2in, clip, width=.45\textwidth]{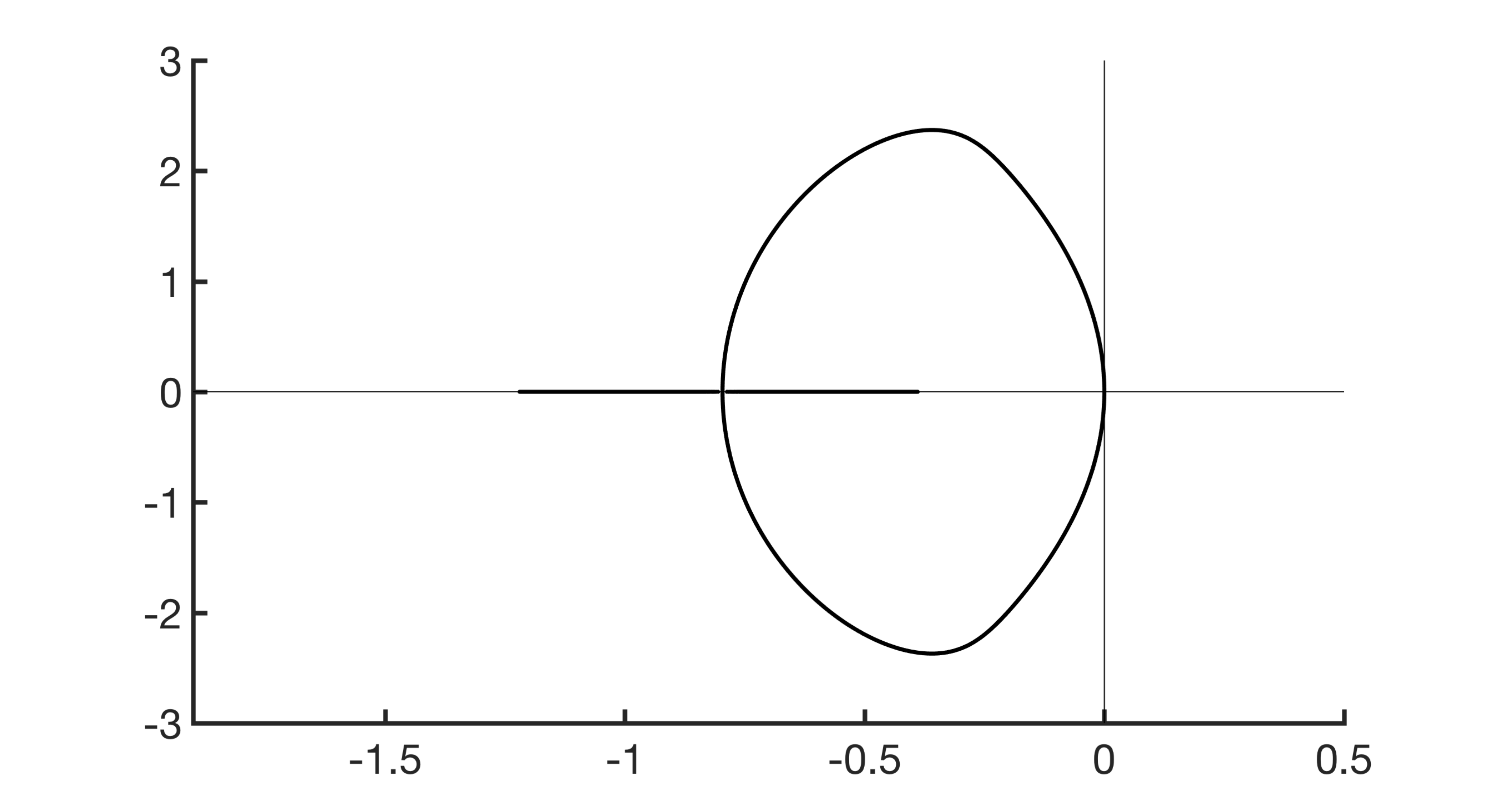}~~
    \includegraphics[trim=1.2in 0.2in 1.2in 0.2in, clip, width=.45\textwidth]{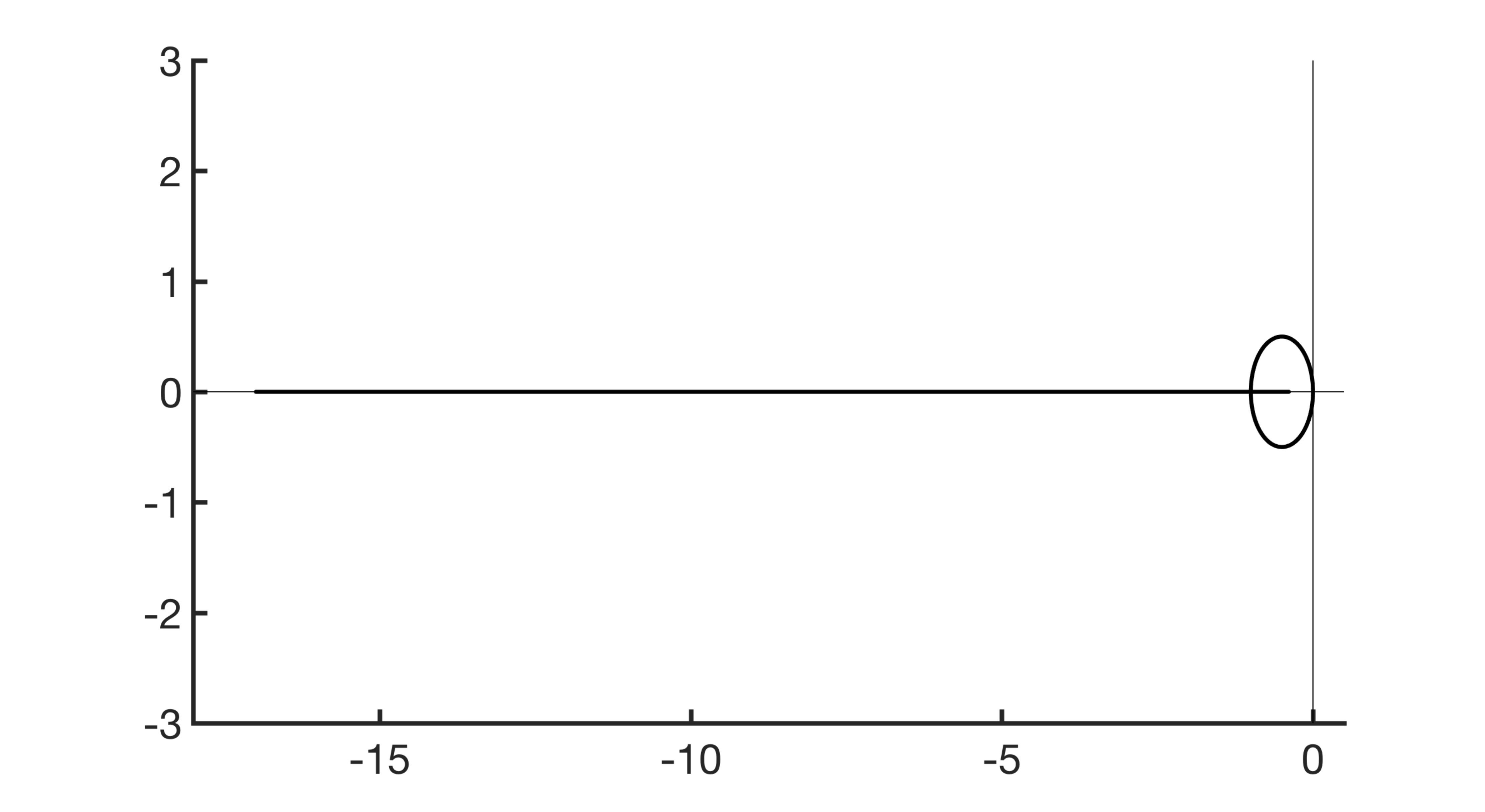}
    \caption{The trajectories $\Lambda(R)$ of ($\mathcal{D}_x^{21,20}$, $\mathcal{D}_x^{20,21}$, $\mathcal{D}_{xx}^{20}$) with $R=0.1$ (left) and $R=2$ (right).}
    \label{fg:num_wave_semi_sym_d}
  \end{subfigure}
  \caption{Trajectories $\Lambda(R)$ of the semi-discretized ODE system of the semi-dissipative wave system using symmetric $\mathcal{D}_x^-$ and $\mathcal{D}_x^+$.}
  \label{fg:num_wave_semi_sym}
\end{figure}

In a second set of the semi-discretization tests, we consider $\mathcal{D}_x^-$ and $\mathcal{D}_x^+$ that are not symmetric; and a similar trend is observed, that is, the ``height'' of $\Lambda(R)$ appears a decreasing function in $R$.
These plots are given in Figure~\ref{fg:num_wave_semi}, where 
two combinations ($\mathcal{D}_x^{3,1}$, $\mathcal{D}_x^{1,2}$, $\mathcal{D}_{xx}^2$) and ($\mathcal{D}_x^{21,20}$, $\mathcal{D}_x^{10,11}$, $\mathcal{D}_{xx}^{20}$) are used to generate the curves.
\begin{figure}[ht]\centering
  \begin{subfigure}{\textwidth}\centering
    \includegraphics[trim=1.2in 0.2in 1.2in 0.2in, clip, width=.45\textwidth]{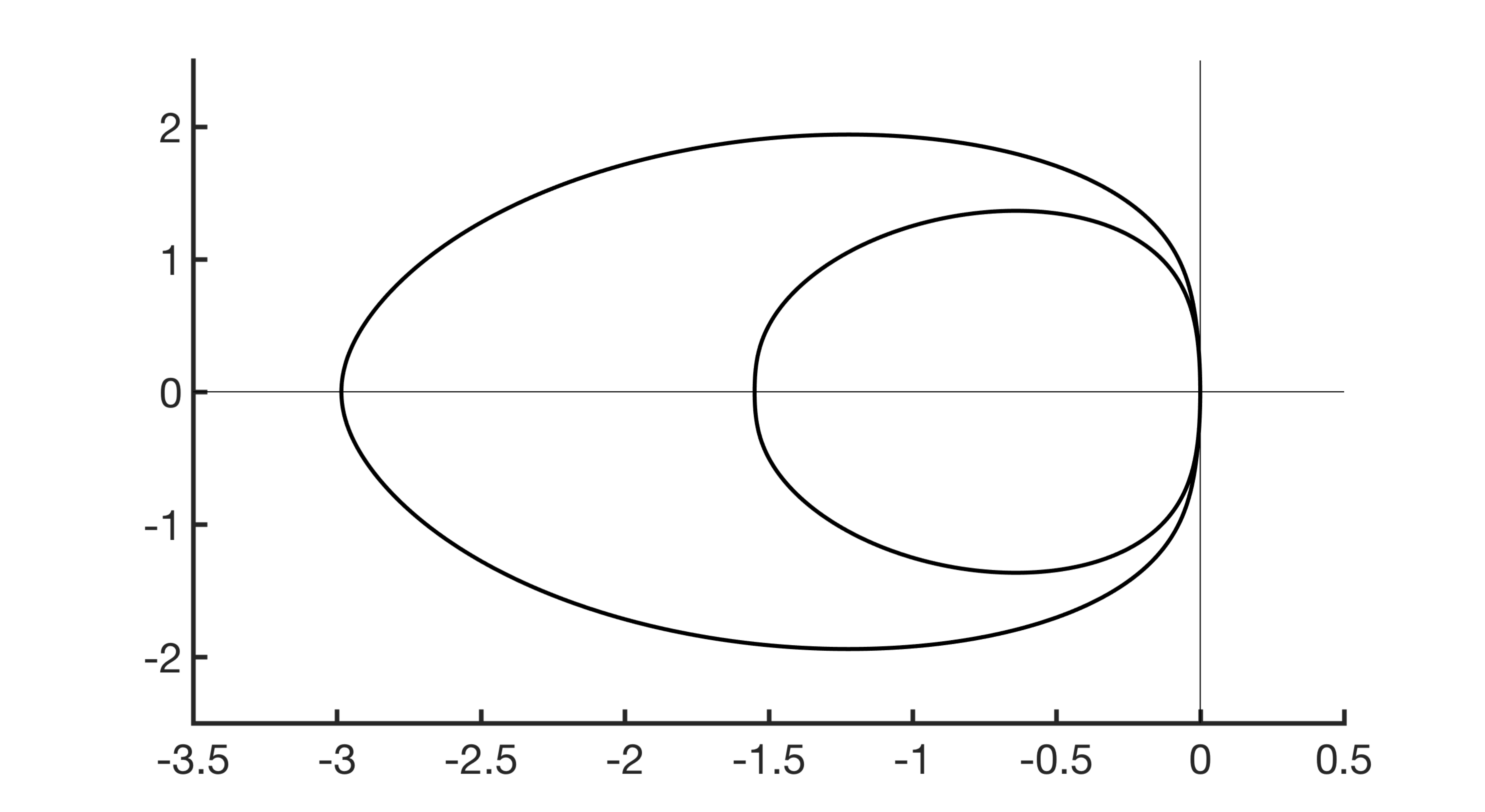}~~
    \includegraphics[trim=1.2in 0.2in 1.2in 0.2in, clip, width=.45\textwidth]{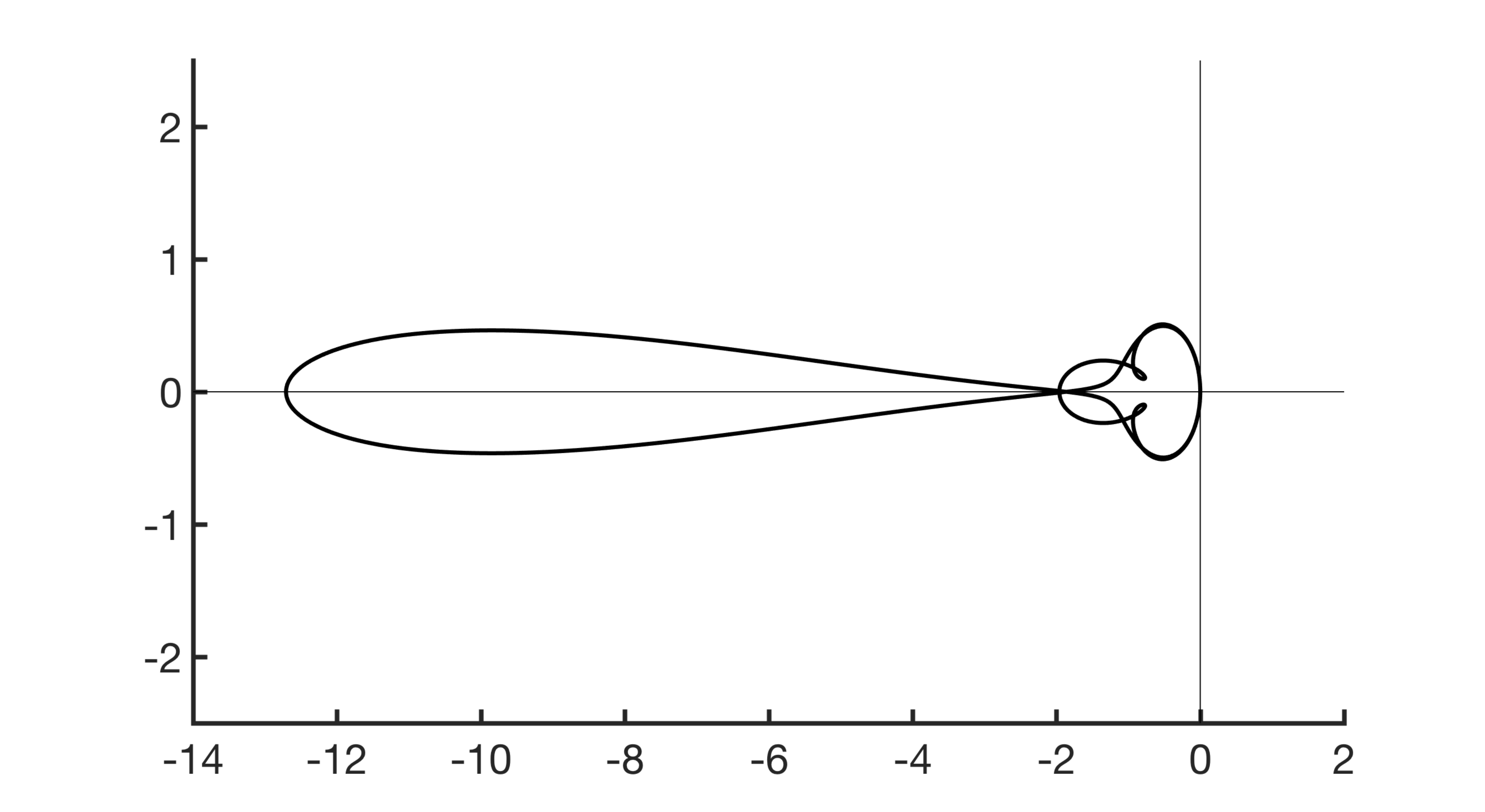}
    \caption{The trajectories $\Lambda(R)$ of ($\mathcal{D}_x^{3,1}$, $\mathcal{D}_x^{1,2}$, $\mathcal{D}_{xx}^2$) with $R=0.1$ (left) and $R=2$ (right).}
    \label{fg:num_wave_semia}
  \end{subfigure}
  \begin{subfigure}{\textwidth}\centering
    \includegraphics[trim=1.2in 0.2in 1.2in 0.2in, clip, width=.45\textwidth]{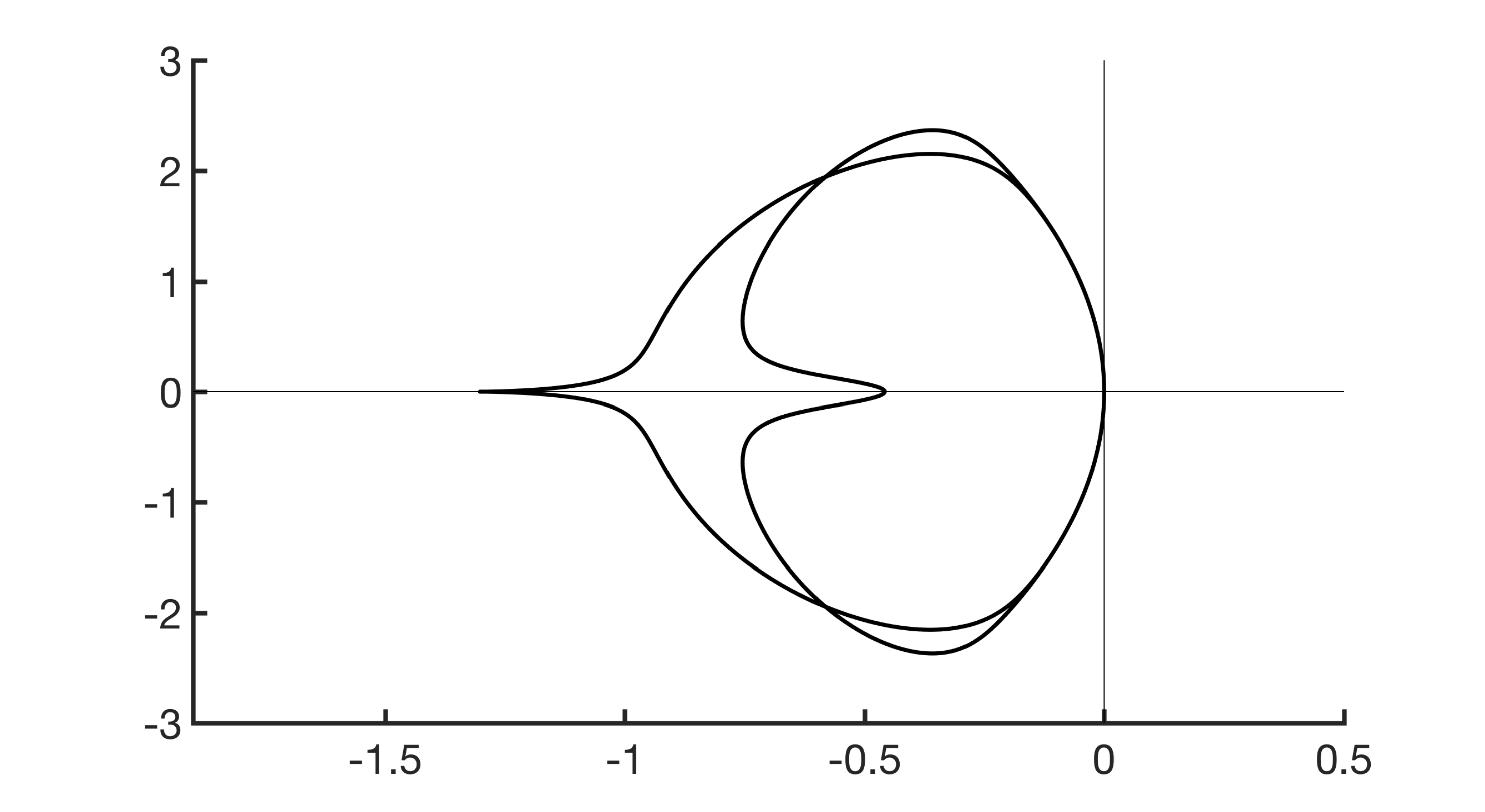}~~
    \includegraphics[trim=1.2in 0.2in 1.2in 0.2in, clip, width=.45\textwidth]{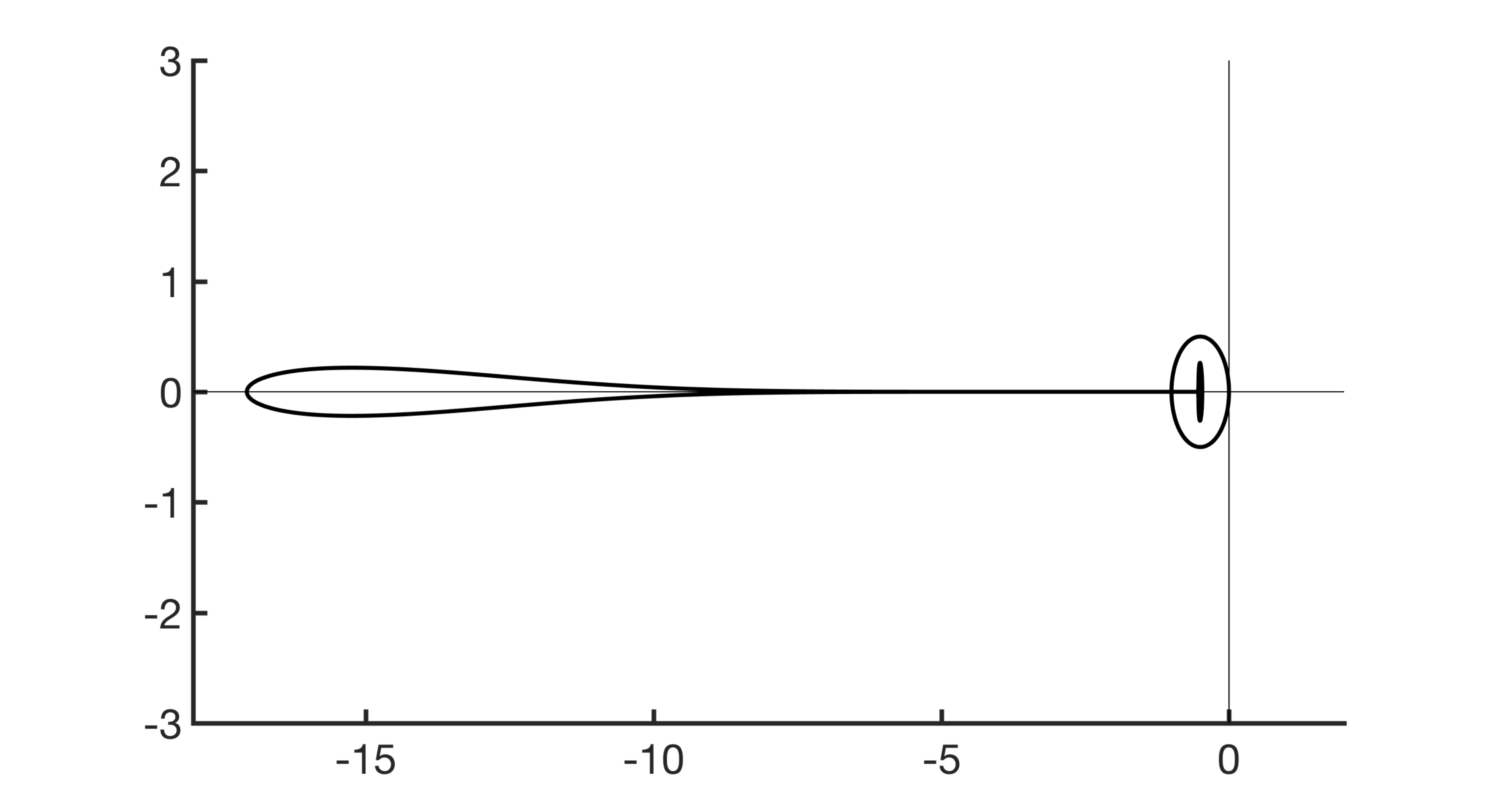}
    \caption{The trajectories $\Lambda(R)$ of ($\mathcal{D}_x^{21,20}$, $\mathcal{D}_x^{10,11}$, $\mathcal{D}_{xx}^{20}$) with $R=0.1$ (left) and $R=2$ (right).}
    \label{fg:num_wave_semi_c}
  \end{subfigure}
  \caption{Trajectories $\Lambda(R)$ of the ODE system after spatial discretization of the semi-dissipative wave system using asymmetric $\mathcal{D}_x^-$ and $\mathcal{D}_x^+$.}
  \label{fg:num_wave_semi}
\end{figure}

Finally, we verify the conditional stability using symmetric $\mathcal{D}_x^-$ and $\mathcal{D}_x^+$ as indicated by Theorem~\ref{thm:wave_sym_hord} by plotting the instability index $I_h$ against the number of cells $N$ at different values $\mu_c=\nu\delta t/h^2$, as in the ADE case.
Combining two spatial discretizations ($\mathcal{D}_x^{3,1}$, $\mathcal{D}_x^{1,3}$, $\mathcal{D}_{xx}^2$) and ($\mathcal{D}_x^{11,10}$, $\mathcal{D}_x^{10,11}$, $\mathcal{D}_{xx}^{10}$) and three time-integrators FE, RK2, RK4, the $I_h$-$N$ curves corresponding to different values of $\mu_c$ are presented in Figure~\ref{fg:num_wave_full}.
\begin{figure}\centering
  \begin{subfigure}{\textwidth}
    \includegraphics[trim=0.6in 0.0in 1.1in 0.2in, clip, width=.45\textwidth]{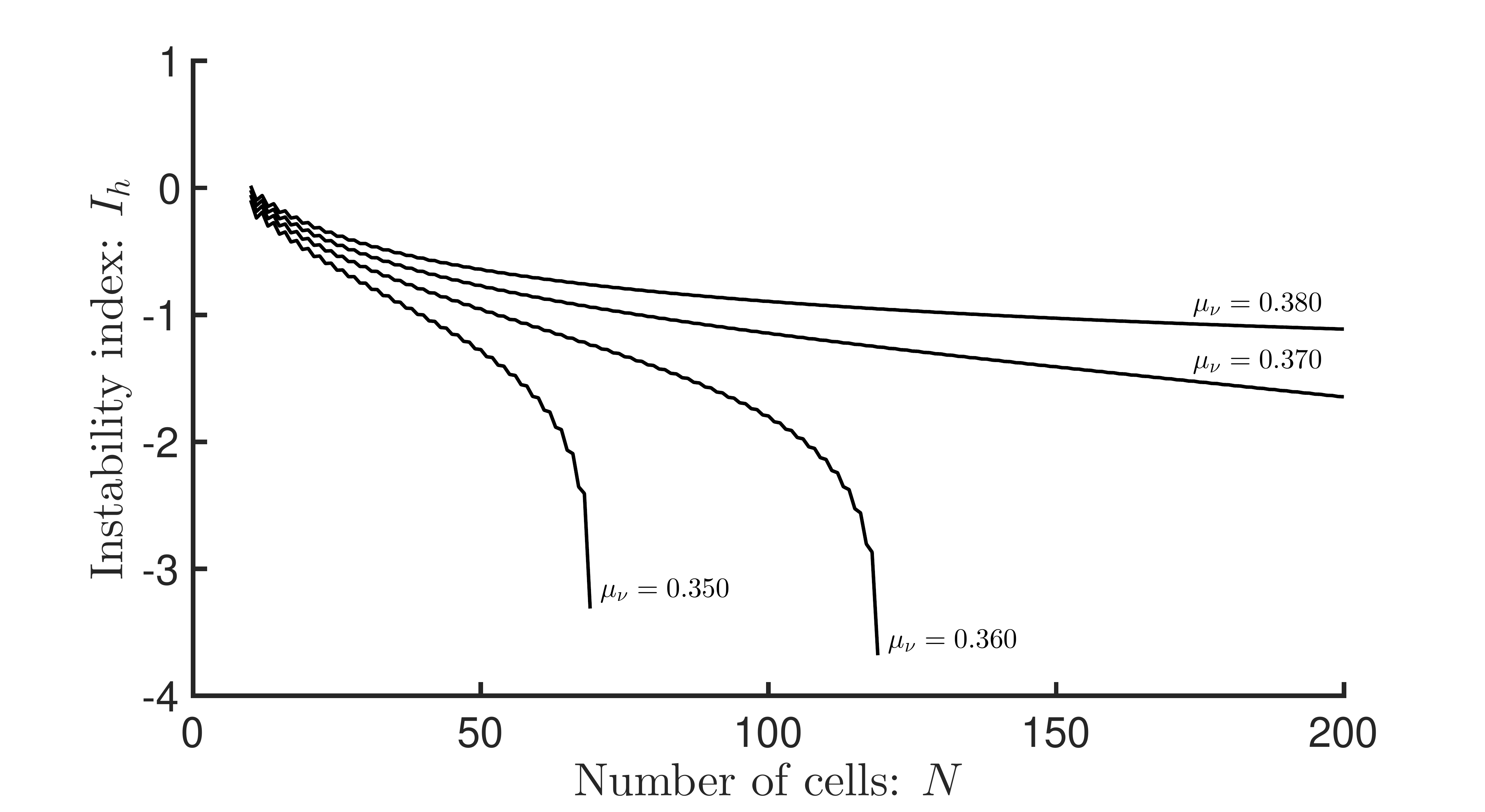}~~
    \includegraphics[trim=0.6in 0.0in 1.1in 0.2in, clip, width=.45\textwidth]{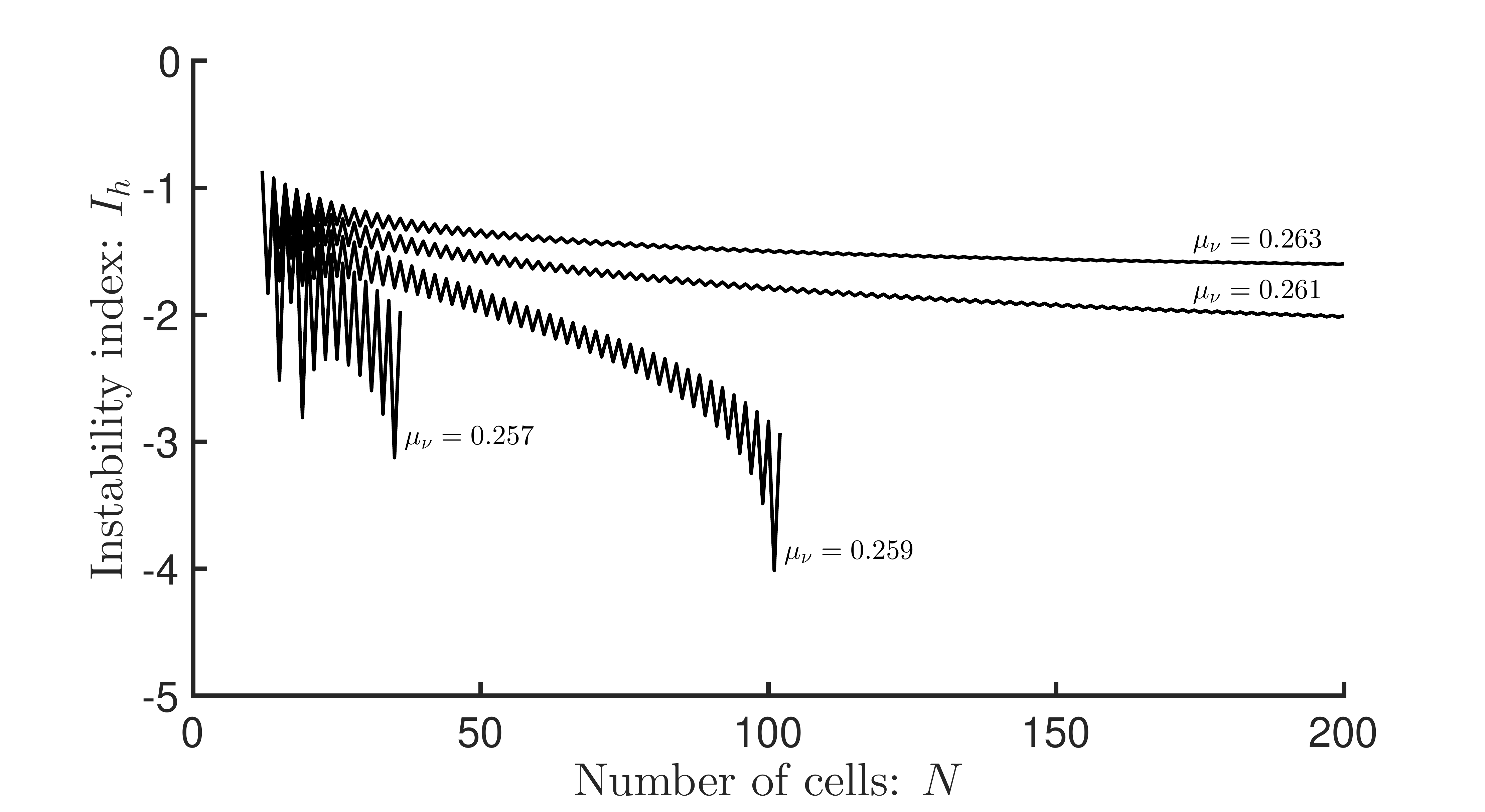}~~
    \caption{FE combined with: (left) $\mathcal{D}_x^{3,1}$, $\mathcal{D}_x^{1,3}$ and $\mathcal{D}_{xx}^2$, (right) $\mathcal{D}_x^{11,10}$, $\mathcal{D}_x^{10,11}$, and $\mathcal{D}_{xx}^{10}$.}
    \label{fg:num_wave_full_fe}
  \end{subfigure}
  \begin{subfigure}{\textwidth}
    \includegraphics[trim=0.6in 0.0in 1.1in 0.2in, clip, width=.45\textwidth]{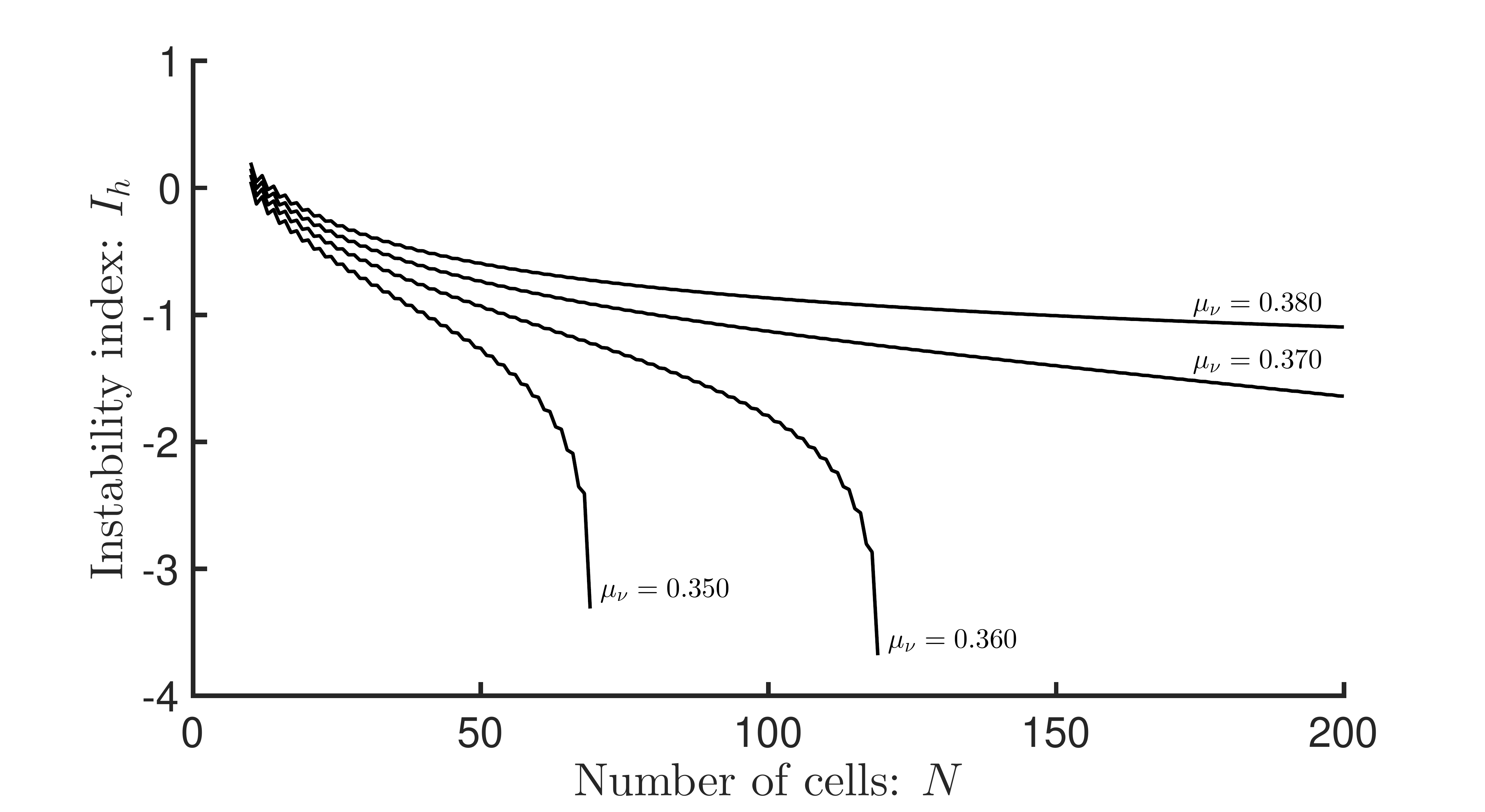}~~
    \includegraphics[trim=0.6in 0.0in 1.1in 0.2in, clip, width=.45\textwidth]{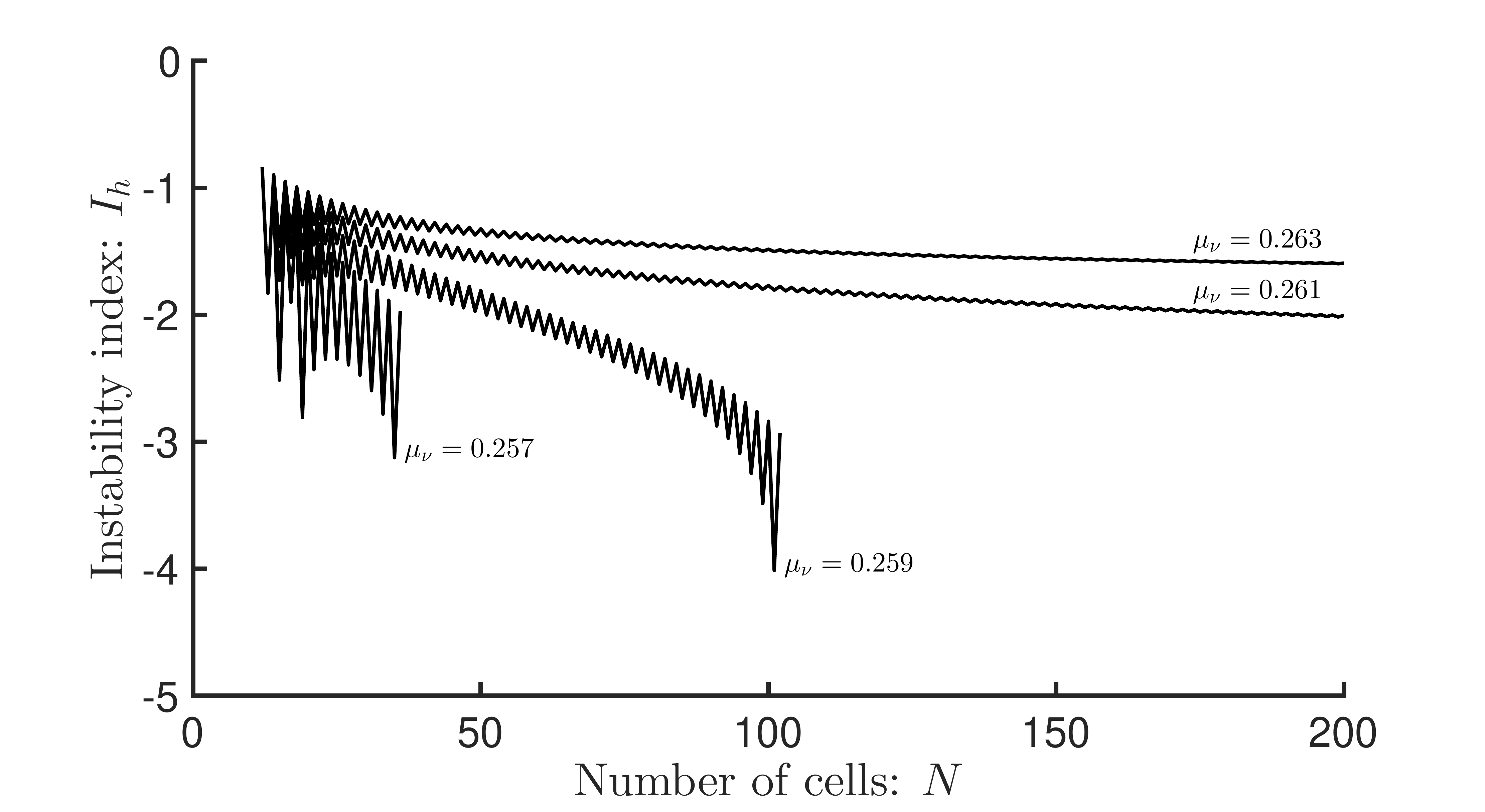}~~
    \caption{RK2 combined with: (left) $\mathcal{D}_x^{3,1}$, $\mathcal{D}_x^{1,3}$ and $\mathcal{D}_{xx}^2$, (right) $\mathcal{D}_x^{11,10}$, $\mathcal{D}_x^{10,11}$, and $\mathcal{D}_{xx}^{10}$.}
    \label{fg:num_wave_full_rk2}
  \end{subfigure}
  \begin{subfigure}{\textwidth}
    \includegraphics[trim=0.6in 0.0in 1.1in 0.2in, clip, width=.45\textwidth]{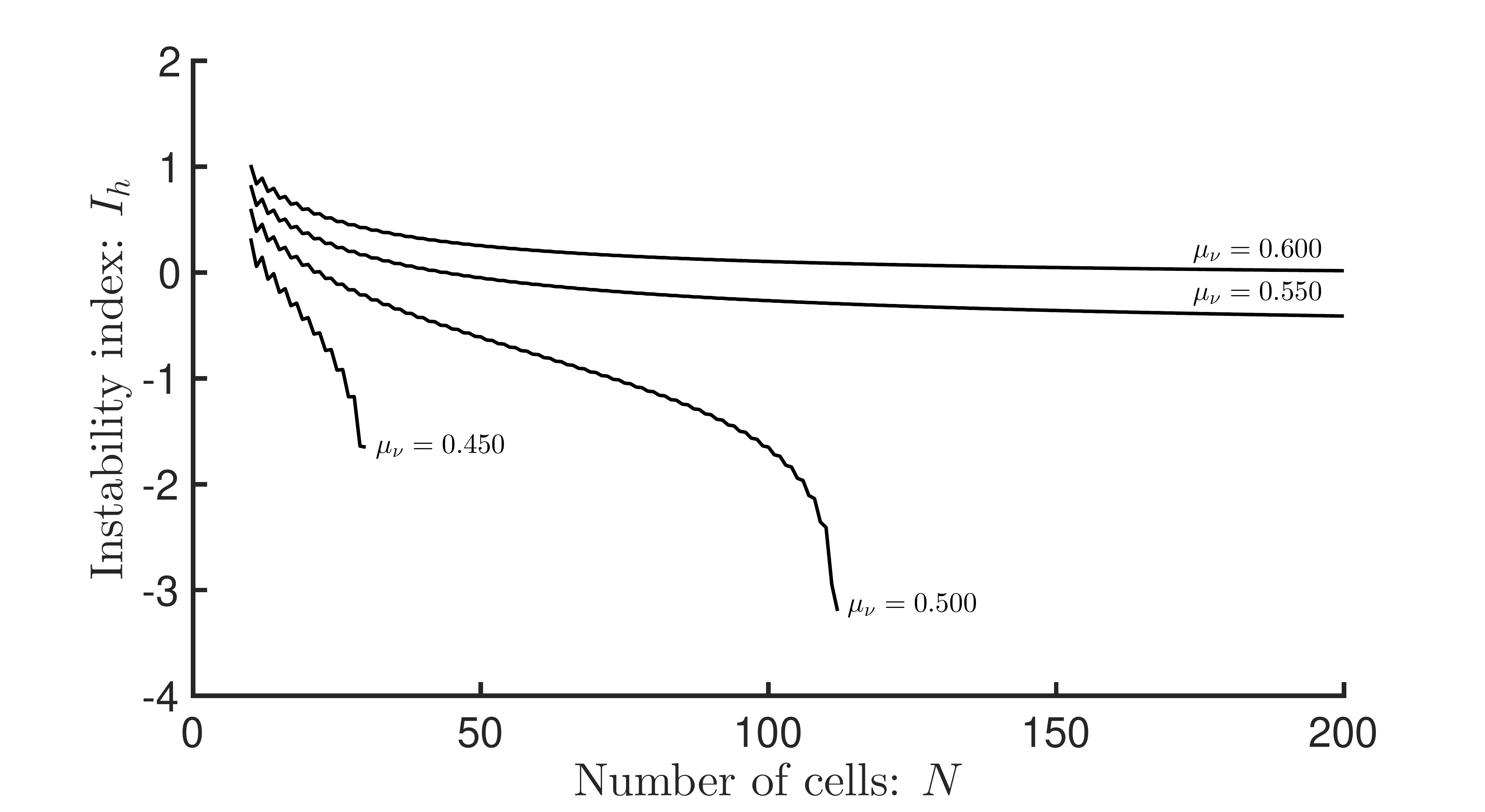}~~
    \includegraphics[trim=0.6in 0.0in 1.1in 0.2in, clip, width=.45\textwidth]{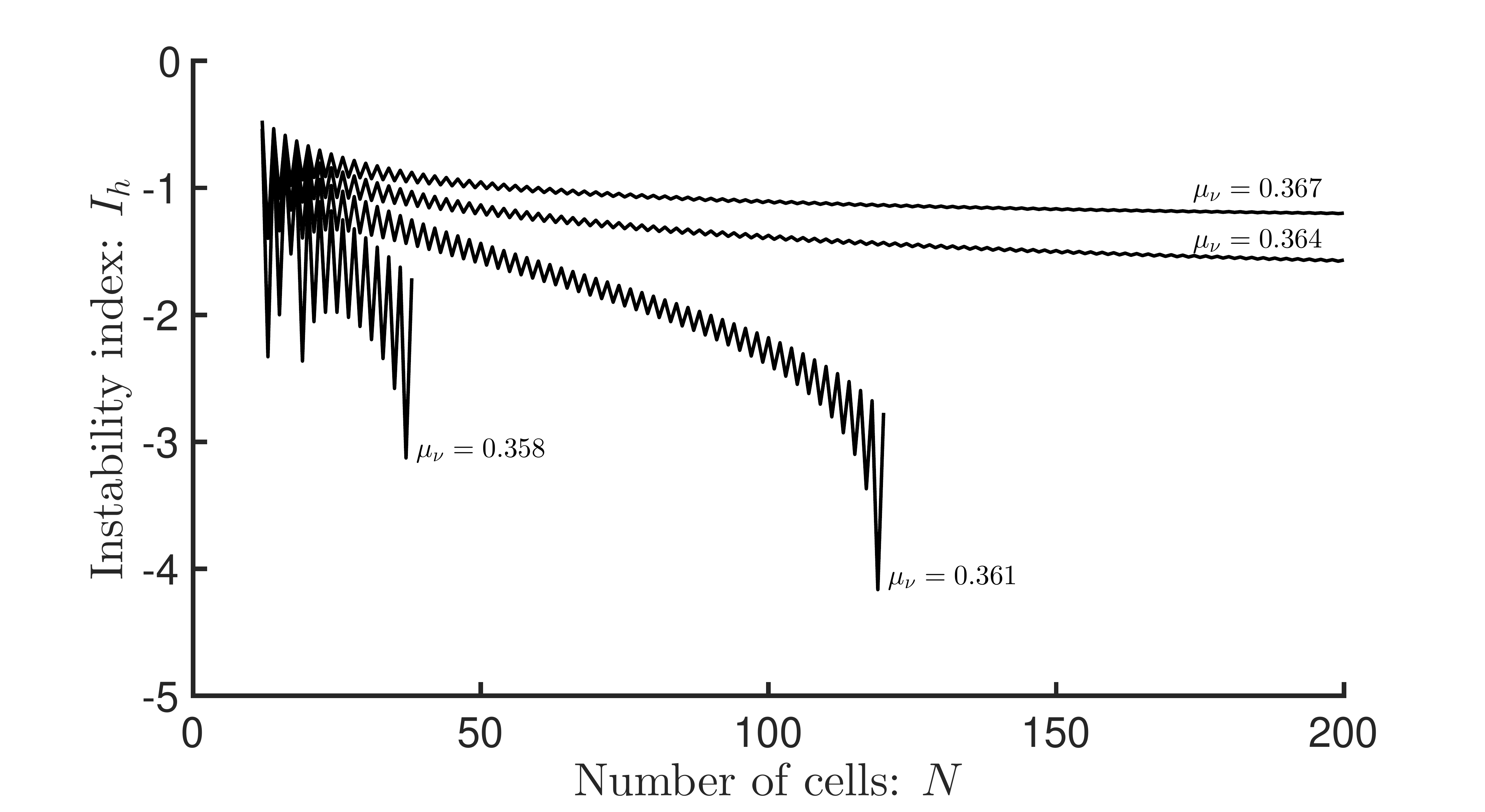}~~
    \caption{RK4 combined with: (left) $\mathcal{D}_x^{3,1}$, $\mathcal{D}_x^{1,3}$ and $\mathcal{D}_{xx}^2$, (right) $\mathcal{D}_x^{11,10}$, $\mathcal{D}_x^{10,11}$, and $\mathcal{D}_{xx}^{10}$.}
    \label{fg:num_wave_full_rk4}
  \end{subfigure}
  \caption{The {\it instability index} $I_h$ vs. the number of cells $N$ at different values of $\mu_{\nu}=\nu\delta t/h^2$ for the semi-dissipative wave system.
    A broken curve indicates conditional stability.}
  \label{fg:num_wave_full}
\end{figure}
Similar as in the ADE case, there appears to be a threshold below which the curve breaks beyond a certain point $N_c$, indicating the stability of the fully discretized method for all $h<h_c=1/N_c$.

\section{Conclusions}
\label{sec:concl}
In this work, we present some general stability results regarding finite difference discretizations with arbitrary order of accuracy for linear advection-diffusion equations and a partially dissipative wave system.
A major motivation for this study is to gain insights into how the stability may be affected in a common practice of many application areas, where an upwind-biased discretization scheme for the advection term is combined with a independently chosen central scheme for the diffusion term.
To this end, we show that if a stable scheme is selected to discretize the advection term and any central method is used in discretizing the diffusion term, the resulting semi-discretized method gives rise to a stable linear ODE system.
Furthermore, it leads to a conditionally stable fully-discretized method when combined with any time-integrator that is at least first-order accurate.
As a byproduct of the analysis, we prove that high-order spatial discretization cannot be paired with some popular lower-order time-integrators to yield a stable method for solving the linear advection equation.

For simplicity, we have assumed periodic boundary conditions and explicit Runge-Kutta time-integrators in the context of method of lines.
However, our results remain valuable when these limitations are lifted.
In particular, in the view of a classical theory presented by Godunov and Ryabenkii, the stability criterion remains necessary for arbitrary enforcement of non-periodic boundary conditions in the limit $h\to0$.
Whereas if implicit or multi-step methods are selected for integration in time, our analysis easily applies as it only makes use of the stability region of these schemes.

\bibliographystyle{unsrt}

\end{document}